\documentclass[reqno]{amsart}
\title{
Sinha's spectral sequence for long knots in codimension one and  non-formality of the  little 2-disks operad}

\author{Syunji Moriya\\
}
\thanks{This work is partially supported by JSPS KAKENHI Grant Number JP17K14192.}
\address{Osaka Central Advanced Mathematical Institute, Osaka Metropolitan University, 3-3-138 Sugimoto, Sumiyoshi-ku Osaka 558-8585 Japan}
\email{moriyasy@gmail.com}

\usepackage{setspace}
\usepackage{amsthm}
\usepackage{amscd}
\usepackage{amsmath,amsthm,amssymb,amscd,mathrsfs,picture}
\usepackage[arrow,matrix]{xy}
\usepackage{here}
\input xy
\xyoption{all}
\usepackage{enumerate}

\theoremstyle{definition}
\newtheorem{defi}{Definition}[section]

\newtheorem{exa}[defi]{Example}
\newtheorem{rem}[defi]{Remark}
\theoremstyle{plain}
\newtheorem{prop}[defi]{Proposition}

\newtheorem{lem}[defi]{Lemma}

\newtheorem{thm}[defi]{Theorem}

\newtheorem{cor}[defi]{Corollary}

\newcommand{\eps}{\epsilon}

\newcommand{\CECHC}{\check{\mathrm{C}}}

\newcommand{\TT}{\mathcal{T}} 

\newcommand{\PK}{\mathcal{PK}} 
\newcommand{\Conf}{\mathit{Conf}} 

\newcommand{\EEE}{\mathcal{E}}
\newcommand{\TPhi}{\widetilde{\Phi}}
\newcommand{\CEE}{\check{\mathbb{E}}}

\newcommand{\CH}{\mathcal{CH}}

\newcommand{\TTH}{\mathbb{T}}

\newcommand{\fat}{\mathrm{fat}}

\newcommand{\Map}{\mathrm{Map}}
\newcommand{\oper}{\mathcal{O}}
\newcommand{\ASS}{\mathcal{A}}
\newcommand{\KK}{\mathcal{K}}
\newcommand{\DD}{\mathcal{D}}

\newcommand{\aoper}{\mathcal{P}}

\newcommand{\SP}{\mathcal{SP}}
\newcommand{\Hoch}{\mathrm{CH}}

\newcommand{\CG}{\mathrm{Top}}

\newcommand{\tot}{\operatorname{Tot}}

\newcommand{\GG}{\mathsf{G}}
\newcommand{\TGG}{\tilde{\mathsf{G}}}

\newcommand{\Embbar}{\overline{\mathrm{Emb}}}
\newcommand{\Emb}{\mathrm{Emb}}

\newcommand{\RR}{\mathbb{R}}

\newcommand{\Sphere}{\mathbb{S}}

\newcommand{\EE}{\mathbb{E}}

\newcommand{\hocolim}{\mathrm{hocolim}}

\newcommand{\CECHF}{\check{\mathsf{C}}}

\newcommand{\kk}{\mathsf{k}}

\newcommand{\PP}{\mathrm{P}} 

\newcommand{\Dar}{\stackrel{D}{\longmapsto}}

\allowdisplaybreaks[1]

\begin{document}

\begin{abstract}

We compute some differentials of  Sinha's spectral sequence for cohomology of the space of long knots modulo immersions in codimension one, mainly over a field of  characteristic $2$ or $3$. This spectral sequence is closely related to Vassiliev's spectral sequence for the space of  long knots in codimension $\geq 2$. We prove that the $d_2$-differential of an element is non-zero in characteristic $2$, which has already essentially been proved by Salvatore, and the $d_3$-differential of another element is non-zero in characteristic $3$. While the geometric meaning of the sequence is unclear in condimension one, these results have some applications to non-formality of operads. The result in characteristic $3$ implies planar non-formality of the standard map $C_*(E_1)\to C_*(E_2)$ in characteristic $3$, where $C_*(E_k)$ denotes the chain little $k$-disks operad. We also reprove the result of Salvatore which states that $C_*(E_2)$ is not formal as a planar operad in characteristic $2$ using the result in characteristic $2$. \\
For the computation, we transfer the structure on configuration spaces  behind the spectral sequence onto Thom spaces over fat diagonals through a duality between configuration spaces and fat diagonals.  This procedure enables us to describe the differentials by relatively simple maps to Thom spaces. We also show that the $d_2$-differential of the generator of bidegree $(-4,2)$ is zero in characteristic $\not= 2$. This computation  illustrates how one can manage the 3-term relation using the description. Although the computations in this paper are concentrated to codimension one, our method also works for codimension $\geq 2$ and we prepare most of basic notions and lemmas for general codimension.
\end{abstract}
\maketitle
\tableofcontents
\section{Introduction}
A long knot is a smooth embedding $\RR\to \RR^d$ which coincides with a fixed linear embedding  outside of a compact set. Vassiliev's spectral sequence for the cohomology of the space of long knots \cite{vassiliev} has  drawn much attention for about 30 years especially in the case of $d=3$ where Vassiliev's sequence is related to finite type invariants of knots. Independently, Goodwillie-Weiss' fundamental work on embedding spaces called embedding calculus \cite{GK,weiss} led to the introduction of another spectral sequence for the space of long knots by Sinha \cite{sinha, sinha1}.  Sinha's  sequence is closely related to Vassiliev's sequence  and also related to the little disks operads. Lambrechts-Turchin-Voli\'c \cite{LTV} proved the rational collapse of Vassiliev's sequence for $d\geq 4$, conjectured by Vassiliev, using these relations. (See also \cite{AT} for an alternative and generalized proof. Kontsevich \cite{kontsevich1} had proved the collapse for $d=3$ in the diagonal part earlier.) Vassiliev actually conjectured more strongly the stable splitting of a filtration of a simplicial resolution of his discriminant space. If this conjecture is true, the spectral sequence also collapses in any positive characteristic. Unlike the case of characteristic zero, not so much is known about higher differentials of Vassiliev's or Sinha's spectral sequence in this case.
As a related result,  Boavida de Brito-Horel \cite{BH}  proved  infinitely many higher differentials of Sinha's sequence vanish over the $p$-adic integers. This is obtained by using an action  of Grothendieck-Teichm\"uller group on the spectral sequence. In this paper, we propose a different geometric approach to the calculation of differentials of Sinha's sequence. While the computational examples in this paper are concentrated in the case $d=2$, our method also works for $d\geq 3$ to some extent. \\
 \indent Precisely speaking, we deal with spectral sequence for the space $\Embbar_c(\RR,\RR^d)$ of long knots modulo immersions, the  homotopy fiber of the inclusion of the  space $\Emb_c(\RR,\RR^d)$ of long knots to the space of immersions $\RR\to \RR^d$ with a similar endpoint condition. The relation between $\Emb_c(\RR,\RR^d)$ and $\Embbar_c(\RR,\RR^d)$ is well-understood, including their spectral sequences, see \cite{sinha1, turchin}.  $\Embbar_c(\RR,\RR^d)$ is related to the operad more directly, and was used in the proof of rational collapse.  
In the case $d=2$, the case of { codimension one}, it is not known whether Sinha's  sequence converges to $H^*(\Embbar_c(\RR,\RR^2))$ (or some other geometric object) and it is not likely.  However,  it is relatively easy to find a new computational example since the rational collapse is not known and fewer differentials vanish by degree reason in this case.  For $d=2$, the sequence is also related to  the little 2-disks operad which is one of the most ubiquitous operads. We compute   higher differentials of three elements as follows.
\begin{thm}[Theorems \ref{Tdifferential}, \ref{Tch3} and Proposition \ref{Pd2_ch3}]\label{Tmain}
Let $\EE_r^{p,q}$ be Sinha's spectral sequence for  cohomology of the space of long knots modulo immersions in codimension one. More precisely, $\EE^{p,q}_r$ is the Bousfield-Kan type cohomology spectral sequence induced by the cosimplicial space $\KK^\bullet_2$ defined in \cite{sinha1} which consists of compactified configuration spaces of ordered points in $\RR^2$  (see Definitions \ref{DKontsevich}, \ref{Dspecseq}).

\begin{enumerate}
\item (\cite{salvatore}) In characteristic $2$, there exists an element in $\EE_2^{-4,2}$ which is not annihilated by the $d_2$-differential.
\item In characteristic  $\not=2$, the generator of $\EE_2^{-4,2}$ is annihilated by the $d_2$-differential. 
\item In characteristic $3$, there exists an element in $\EE_3^{-5,3}$ which is not annihilated by the $d_3$-differential. 
\end{enumerate}
\end{thm}
The homological version of  part 1 of this theorem has already been proved implicity by Salvatore \cite{salvatore}. The computation of an obstruction given there essentially includes the computation of the dual differential. Turchin \cite{turchin} introduced a notion of divided product which produces elements of Sinha's (or Vassiliev's) sequence in terms of graphs. The  products  of certain graphs  form an interesting subcomplex of a page  of Sinha's sequence. In our cohomological computation, the elements in  parts 1 and 3 of Theorem \ref{Tmain} correspond to the products $\langle Z_1,Z_1\rangle$ and $\langle Z_1, Z_2\rangle$ in the notation of \cite{turchin}, respectively while we use a version reversing the order of vertices (see Remark \ref{Rdiscussion}). The product $\langle Z_1, Z_{p-1}\rangle$ also represents non-trivial cycle in other positive characteristic $p$, and it might be interesting to compute its differentials. Since the  elements involved in  part 2  come from non-torsion elements,  this  might follow from a computation using techniques specific to characteristic 0 such as the configuration space integral, but this is the simplest example which  illustrates how  the 3-term relation is managed in our method, and the result is used in the proof of  part 3.\\
\indent Theorem \ref{Tmain} can be related to non-formality of the little 2-disks operad. Let $f:\mathcal{O}\to \mathcal{P}$ and $f':\mathcal{O}'\to \mathcal{P}'$ be maps of chain operads. A {\em quasi-isomorphism from $f$ to $f'$} is a commutative square
\[
\xymatrix{\mathcal{O} \ar[r]\ar[d] ^f& \mathcal{O'} \ar[d]^{f'} \\
\mathcal{P}\ar[r]& \mathcal{P'}}
\]
 of maps of operads where the horizontal maps induce quasi-isomorphisms of complexes at each arity. A map $f:\mathcal{O}\to\mathcal{P}$ is said to be { formal} if it is connected with the induced map $H_*(f)$  on  homology (with zero differential) by a zigzag of quasi-isomorphisms. A chain operad $\mathcal{O}$ is formal if the map from the initial object to $\mathcal{O}$ is formal.  
Formality of the little disks operads over reals proved in \cite{tamarkin, kontsevich, LV} was essentially used in the proof of rational collapse of Vassiliev's sequence. Part 3 of Theorem \ref{Tmain} 
has the following  corollary about relative planar non-formality.
\begin{cor}\label{Cnon-formality_ch3}
Let $E_k$ be the little $k$-disks operad and $E_1\to E_2$  the map induced by the inclusion $\RR\to \RR^2$ to the first coordinate. The induced map $C_*(E_1)\to C_*(E_2)$  between singular chains is not formal as a map of planar (or non-symmetric) operads in characteristic $3$.
\end{cor}
See subsection \ref{SS3rd} for the proof. 
Using  part 1 of Theorem \ref{Tmain}, we also reprove the (absolute) planar non-formality in characteristic $2$ first proved in \cite{salvatore}.
\begin{cor}[\cite{salvatore}, see section \ref{Sabsolute}]\label{Cnon-formality_ch2}
The chain little $2$-disks operad $C_*(E_2)$ is not formal as a planar  operad  in characteristic $2$.
\end{cor}
In positive characteristic, planar non-formality is much more difficult to prove than non-formality as a symmetric operad which follows from non-formality of  $C_*(E_k)$  as a complex with an action of a symmetric group (see e.g. \cite{CH}).
As related works, Turchin-Willwacher \cite{TW} proved symmetric non-formality of the map $C_*(E_k)\to C_*(E_{k+1})$ over reals for $k\geq 1$. They especially proved that a $d_3$-differential on  Sinha's spectral sequence of codimension one is non-trivial over characteristic $0$. In \cite{livernet}, Livernet proved  symmetric non-formality of the Swiss-cheese operads. In \cite{moriya2}, planar non-formality of the framed little disks operads of odd dimension $\geq 5$ was proved.\\
 \indent After this paper was written up, the author learned that  planar non-formality of the map $C_*(E_1)\to C_*(E_2)$ over any positive characteristic follows from a theorem of Goodwillie\cite{goodwillie}. This observation is due to Salvatore, see the thesis of Andrea Marino \cite{marino}. Neverthless, the computation of differentials in Theorem \ref{Tmain}(2),(3) is still new and having a concrete obstruction will be useful in a study of absolute non-formality. \\
\indent Our method of computation is different from \cite{salvatore} which makes use of  obstruction theory and  a combinatorial model of the operad. We use a duality between configuration spaces and fat diagonals, which is similar to the one given in our previous paper \cite{moriya1}, to transfer necessary structure to fat diagonals and resolve the structure by $\check{\text{C}}$ech complex. Sinha's cosimplicial space (or any other equivalent structure behind the spectral sequence) is a diagram consisting  of the ordered configuration spaces $Conf_p(D^d)$ of $p$ points in the $d$-dimensional unit disk with various $p$. Our duality reduces to the Poincar\'e-Lefschetz duality between $H^*(Conf_p(D^d))$ and  $H_*((D^d)^p,\Delta_{fat}\cup \partial (D^d)^p)$, where $\Delta_{fat}$ denotes the fat diagonal of $(D^d)^{p}$ i.e.  the union of all diagonals. \\
\indent We need to transfer the cosimplicial (or cubical) structure onto the fat diagonal on the point-set level in order to deal with higher differentials. It is difficult to do so using the standard  duality map which uses the cap product with fundamental cycle because of the presence of various chain homotopies. Instead, we use a form of $S$-duality called Atiyah duality, or its refinement by Cohen \cite{cohen}, which states an equivalence between the Spanier-Whitehead dual of a manifold and Thom spectrum of its normal bundle (see the Introduction of \cite{moriya1} for more detailed explanation).   We construct a diagram of Thom spaces of fat diagonals encoding the spectral sequence, then we resolve it by $\CECHC$ech complexes associated to the covering by diagonals. This replacement of a diagram brings  benefits for keeping track of the higher differentials. \\
\indent To compute differentials, we use the well-known description of differential by a zigzag of the horizontal and vertical differentials so we need to construct a bounding chain of a given cycle, a chain whose boundary is the cycle.  If we use the singular cochain of Sinha's cosimplicial model itself, it contains simple cocycles representing generators of $H^*(Conf_p(D^d))$, but the commutativity homotopies between their products make it difficult to construct succesive bounding (co)chains. The more serious problem is that a manageable bounding cochain of the 3-term relation $g_{ij}g_{jk}+g_{jk}g_{ki}+g_{ki}g_{ij}=0$ is not known except for the case of characteristic zero where the configuration space integral is available. If we consider the  homology, it is also diffcult to find a nice bounding chain of the relations defining Poisson operad.  
In the $\CECHC$ech complex, we can realize monomials of the generators by very simple maps and commutativity is managed by Eilenberg-Zilber shuffle product which is strictly symmetric in the normalized chain (while  roles of the two products are not exactly the same) and we can take a very simple bounding chain of the 3-term relation without going into the details of the definition of a chain level operation  (see section \ref{Snon-triviality}). \\
\indent In comparison with the method in \cite{salvatore}, the proof in that paper is concise and  comprehensible, while the present author feels that it might be difficult to find a systematic way to construct  a bounding chain in the surjection operad such as `$\gamma$'  in the proof of Lemma 5.1 of \cite{salvatore}. In our method, we can construct  bounding chains and confirm non-triviality of cycles with graphical and topological intuition. \\
\indent Our method can be also used to describe higher differential for the case of $d\geq 3$ but elements with possibly non-trivial differentials   consist of graphs with five edges or more in these dimensions, so we feel it would be better to begin with writing down the computation of relatively simple elements in the case $d=2$, see Remark \ref{Rdiscussion} for more detailed explanation.\\
\indent The outline of the paper is as follows. In section \ref{SPK}, we introduce two functors which are closely related to Sinha's sequence. These functors come from the embedding calculus and not essentially new, but we make minor adjustment to their definition so as to suit the later constructions. We also prove an equivalence between Sinha's cosimplicial model and our functors. In section \ref{SThom}, we introduce a functor consisting of Thom spaces of  fat diagonals. To make the structure maps of the functor preserve the diagonals on the point-set level, we take care about parameters such as the radius  of a tubular neighborhood. We prove stable and chain level equivalences between the functors in section \ref{SPK} and the functor of Thom spaces. The contents of this section are similar to those of section 3 of \cite{moriya1}. The proofs are minor variations of those given there, but we give the full proofs for the reader's convenience. In section \ref{Sspec_seq}, we define spectral sequences using the functor defined in the previous section and explain the relation between these sequences and Sinha's one. In section \ref{Scondensed} we introduce a class of maps used in the computation and prove their properties. In sections \ref{Sch2}, \ref{Snon-triviality} and \ref{Sch3}, we prove parts 1,2 and 3 of Theorem \ref{Tmain}, respectively, using results in sections \ref{SThom}-\ref{Scondensed}. We construct bounding chains and prove vanishing of some  terms of differentials of the chains and then, prove cancellation or non-triviality of the remaining terms.  In section \ref{Sabsolute}, we prove Corollary \ref{Cnon-formality_ch2}.  The $(i,\pm)$-contractions defined in section \ref{Scondensed} are only used in section \ref{Sch3} so one can read sections \ref{Sch2}, \ref{Snon-triviality} and \ref{Sabsolute} skipping the contents of section \ref{Scondensed} after Lemma \ref{Lcondensed_collapse} in order to understand the outline of the computation quickly. 
\subsection{Correction to the published version}
 The author found errors after a version of this paper was  published  (Quarterly Journal of Mathematics {\bf 75} (2024) no. 3, 1073–1121). In the present version, the errors are corrected as follows. (The main results are still valid.)
\begin{enumerate}
\item We have modified the definition of the number $c_{\alpha\beta}$ given in Definition \ref{Dpunctured}. Under the former definition, the proof of Lemma \ref{Ldiagonal_incl} included an error on the evaluation $|\bar x_{\alpha'}-\bar x_{\beta'}|$. This correction influences the claim of Lemma \ref{Lcondensed_collapse0} and the proofs of Lemmas \ref{Lcondensed_collapse} and \ref{Li-cont2}, which we modified. The modifications are minor changes. 
\item The other corrections are made on Definitions \ref{DPK}, \ref{Dpunctured}(5),(6), \ref{Dspectrafunctor} (6),(7) and the proof of Theorem \ref{TAtiyahdual}. In the published version, the space $\mathcal{C}(P)$ given in Definition \ref{Dpunctured} may not have the expected homotopy type since it includes the points in the boundary of a disc. These (minor) corrections are made to fix this.
\end{enumerate} 
\subsection{Notation and Terminology}\label{SSNT}

\begin{enumerate}
\item $\CG$ denotes the category of unpointed topological spaces and continuous maps and $\CG_*$ the corresponding pointed category.

\item For an unpointed (resp. pointed) topological space $X$, $C_*(X)$ (resp. $\bar C_*(X)$) denotes the  normalized singular chain complex (resp. the reduced normalized singular chain complex) with coefficients in a fixed field $\kk$. For chains $a\in \bar C_*(X)$, $b\in \bar C_*(Y)$, $a\wedge b\in \bar C_*(X\wedge Y)$ denotes the Eilenberg-Zilber shuffle product. For the transposition map $T:X\wedge Y\to Y\wedge X$, we have the following identity at chain level:
\[
T_*(a\wedge b)=(-1)^{|a||b|}b\wedge a, 
\]
where $|a|$ denotes the degree of $a$. This product satisfies the usual Leibniz rule for the singular differential: $d(a\wedge b)=da\wedge b+(-1)^{|a|}a\wedge db$. Throughout the paper,  for a pointed map $f:X\to Y$, the subscript $*$ for the pushforward on homology is omitted if no confusion occurs.  So $f_*(a)$ is denoted by $f(a)$. For two maps $f,g:X\to Y$, we denote $f(a)\pm g(a)$ by $(f\pm g)(a)$. When $1/2\in \kk$, for two maps $f^+, f^-:X\to Y$ with the superscripts $\pm$ on the same symbol, $f^{\pm}(a)$ denotes the average $(f^+(a)+f^-(a))/2$. We will use  combinations of these abbreviations. For example, $(f^\pm+g^\pm)(a)$ denotes $(f^+(a)+f^-(a)+g^+(a)+g^-(a))/2$.

\item Let $X$ be an unpointed space. We denote  $X$ with disjoint basepoint by $X_+$ and the one-point compactification of $X$, which is also regarded as a pointed space,  by $X^*$. We set  $S^k=(\RR^k)^*$ and $[0,\infty]=[0,\infty)^*$.  We denote the interval $[0,1]$ (regarded as an unpointed space) by $I$. We fix a fundamental cycle $w_{S^2}\in \bar C_2(S^2)$, chains $w_{\infty}\in \bar C_1([0,\infty])$ and $w_I\in \bar C_1(I_+)$ such that $dw_\infty=\{0\}$ and $dw_I=\{1\}-\{0\}$, where $\{0\}, \{1\}$ are cycles represented by $0,1\in [0,1]\subset [0,\infty]$. For  $k\geq 0,\  l>0$, we set 
\[
\begin{split}
w_k&=(w_{S^2})^{\wedge 2}\wedge (w_\infty)^{\wedge k} \qquad \in \bar C_{4+k}(S^4\wedge [0,\infty]^{\wedge k}),\\
w_{k l}&=w_k\wedge (w_I)^{\wedge l}\qquad\qquad\quad \in \bar C_{4+k+l}(S^4\wedge [0,\infty]^{\wedge k}\wedge (I_+)^{\wedge l}).
\end{split}
\]
\item $|-|$ denotes the standard Euclidean norm. We define elements $u,v\in \RR^d$ by
\[
u=(1,0,\dots, 0), \text{ \ and \ } v=(0,1,0,\dots, 0).
\]
\item We denote by $=_1$, $<_1$,$\leq_1, \dots$ etc., the relations between the first coordinates of elements of $\RR^d$. For example,  for two elements $x=(x_1,\dots, x_d), y=(y_1,\dots, y_d)\in \RR^d$ and a number $t\in \RR$, $x<_1y$ and $x=_1y$ mean $x_1<y_1$ and $x_1=y_1$, respectively, and  $x<_1 t$ means $x_1< t$.

\end{enumerate}
\section{Punctured knot model $\PK$ and  configuration space model $\mathcal{C}$ }\label{SPK}
For technical reasons, we mainly deal with a version of the punctured knot model defined below instead of the cosimplicial model.  Throughout the paper, $n$ denotes a positive integer.
\begin{defi}\label{Dpartition0}
A {\em partition $P$ of } $[n+1]=\{0,1,\dots, n+1\}$ is a set of subsets of $[n+1]$ satisfying the following conditions.
\begin{enumerate}
\item $\cup_{\alpha \in P}\alpha =[n+1]$.
\item Each element of  $P$ is non-empty.
\item If $\alpha, \beta \in P$, either of $\alpha=\beta$ or $\alpha\cap \beta=\emptyset$ holds.
\item For each element  $\alpha \in P$, if numbers $i,j,k$ satisfy $i<j<k$ and $i,k\in \alpha$, $j$ also belongs to $\alpha$. 
\item $\# P\geq 2$, in other words the set consisting of the single element $[n+1]$ is not a partition. 
\end{enumerate} 
We call an element of  $P$ a {\em piece of} $P$. We regard a partition as a totally ordered set via the order induced by $[n+1]$. A partition $Q$ is said to be a {\em subdivision of} $P$ if $Q\not=P$ and each piece of $Q$ is contained in some piece of $P$. We let $\PP_n$ denote the category (or poset) of partitions of $[n+1]$. Its objects are the partitions of $[n+1]$. 
A unique non-identity morphism $P\to Q$ exists if and only if $Q$ is a subdivision of $P$. By abuse of notation, we let  $[n+1]$ represent the partition $\{\{0\},\{1\},\dots, \{n+1\}\}$ consisting  of singletons.  
\end{defi}
\begin{exa}\label{Epartition}
The following sets are examples of objects of $\PP_4$:
\[
P=\{\{0\},\{12\},\{345\}\}, \qquad Q=\{\{0\},\{12\}, \{3\}, \{45\}\}.
\]
We omit commas in pieces, so the piece $\{12\}$ denotes $\{1,2\}$.  We have $\{0\}<\{12\}<\{345\}$ for the pieces of $P$. We see that $\#P=3$, $\#Q=4$, and  $Q$ is a subdivision of $P$.  
\end{exa}

\begin{defi}\label{DPK}
\begin{enumerate}
\item Throughout the paper, we fix positive numbers $\rho, \eps$ and $c_0,\dots, c_{n+1}$ satisfying 
\[
c_0+\cdots +c_{n+1}=1, \quad \rho<1,\quad 100\eps/\rho<c_0, \quad 100\bigl(\eps/\rho+\sum_{j<i}c_j\bigr)<c_i \quad (1\leq i\leq n+1).
\] (The last two inequalities are not used until section \ref{Scondensed}.)
\item We define a functor $\PK:\PP_n\to \CG$ as follows.
Set $b_i=c_0+\cdots +c_{i-1}$ for $1\leq i\leq n+1$. For a partition $P=\{\alpha_0<\cdots <\alpha_{p+1}\}$,  $S_P\subset \{1,\dots, n+1\}$ denotes the set of minimum elements in each of $\alpha_1,\dots, \alpha_{p+1}$ (so $\# S_P=p+1$).  Let $D^d\subset \RR^d$ be the $d$-dimensional unit closed disk. The space $\PK(P)$ is the space of embeddings 
\[
f:[0,1]-\bigcup_{i\in S_P} \Bigl(b_i-\frac{c_{i-1}}{4}, b_i+\frac{c_i}{4}\Bigr)\to D^d
\]   such that
\begin{enumerate}
\item $f(0)=-u$ and $f(1)=u$, 
\item $f(t)\in Int(D^d)$ for $t\in (0,1)$, where $Int(D^d)$ is the interior of $D^d$, and 
\item in each connected component of the domain, $f(t)=x+at u$ for some constant elements $x\in D^d$ and $a>0$. 
\end{enumerate}
For a subdivision $Q$ of $P$, the map $\PK(P)\to \PK(Q)$ is the  restriction induced by the inclusion $S_P\subset S_Q$.
\end{enumerate}
\end{defi}
Let $\Delta_n$ be the category whose objects are $[k]=\{0,1,\dots, k\}$ ($0\leq k\leq n$) and whose morphisms are the weakly order preserving maps. 
\begin{defi}\label{DKontsevich}
Let $\KK_d$ denote the $d$-dimensional Kontsevich operad defined in \cite{sinha1}. Its $p$-th term $\KK_d(p)$ is a version of Fulton-Macpherson compactification of the ordered configuration space $Conf_p(\RR^d)$ of $p$ points in $\RR^d$. In $\KK_d(p)$, some of the points in a configuration are allowed to collide in a manner that the direction of collision is recorded in the topology. The operad $\KK_d$ is equipped with a map $\ASS\to \KK_d$ from the associative operad as in \cite{sinha1}. This map induces a cosimplicial space $\KK_d^\bullet$, which we call {\em Sinha's cosimplicial space}, via the framework of McClure-Smith. A coface map of $\KK_d^\bullet$ is given by replacing a point in a configuration with the two points colliding to each other at the point along the vector $u$. We denote the restriction of $\KK^\bullet_d$  to $\Delta_n$ by $\KK^{\leq n}_d$. (Since we will not use  the structure of $\KK_d^\bullet$ in this paper except for the proof of the following lemma which is substantially included in \cite{sinha, sinha1}, we omit details of the definition. ) 
\end{defi}
The homotopy limit of $\KK_d^{\leq n}$ is weakly homotopy equvalent to the $n$-th stage of Taylor tower associated to the space $\Embbar_c(\RR,\RR^d)$. We define a functor $\mathcal{F}:\PP_n\to \Delta_n$ by $P\mapsto [\# S_P-1]$ and $(P\to Q)\mapsto ([\# S_P-1]\cong S_P\subset S_Q\cong [\# S_Q-1])$, where $\cong$ denotes the order preserving bijection.
Let $C$ be a category. We say a natural transfomation  $(F\rightarrow G):C \to \CG$ between two functors  is a {\em termwise homotopy equivalence} if it induces a homotopy equivalence $F(c)\to G(c)$ for each object $c\in C$. Two functors are termwise homotopy equivalent if they are connected by a zigzag of termwise homotopy equivalences.
\begin{lem}\label{LPK}
The two functors $\PK$ and $\mathcal{F}^*\KK_d^\bullet :\PP_n\to \CG$ are termwise homotopy equivalent. 

\end{lem}
\begin{proof}
This is almost implicit in \cite{sinha, sinha1} but we shall write down some details of the proof. 
Let  $\Emb_c(I, D^d)$ be the space of  embeddings $I\to D^d$  with fixed endpoints in $\partial D^d$ and fixed tangent vectors on them with $C^\infty$-topology.  The punctured knot model, substantially described in \cite{weiss}, is a functor inducing a stage of Taylor tower of the knot space $\Emb_c(I,D^d)$ on the homotopy limit. It is defined as follows. Let $\{J_i\}_{1\leq i\leq n+1}$ be a set of pairwise disjoint closed subintervals (or one-point sets) of $I$ whose order of  labels $i$ is consistent with the usual order of $I$. Let $P'_n$ be the poset of non-empty subsets of $\{1,\dots, n+1\}$. $\PP_n$ is isomorphic to $P'_n$ as a poset by $P\mapsto S_P$. We regard $\PK$ as a functor $:P'_n\to \CG$ via this isomorphism. The punctured knot model is the functor $P'_n\to \CG$ which sends $S\in P'_n$ to the space of embeddings $I-(\cup_{i\in S}J_i)\to D^d$ with the same endpoint condition as $\Emb_c(I,D^d)$. (Its homotopy limit is the $n$-th stage of the Taylor tower.) 
Let $\tilde \PK$ be the punctured knot model for $J_i=\{b_i\}$. In \cite{sinha}, the termwise homotopy equivalences
\begin{equation}
\DD_n\langle[D^d]\rangle\to \EEE_n\langle[D^d]\rangle\leftarrow \EEE_n(D^d) \label{EQPK}
\end{equation}
between functors  $P'_n\to \CG$ are constructed. Here, $\EEE_n(D^d)$ is the punctured knot model for some choice of $J_i$, and $\DD_n\langle[D^d]\rangle$ is a functor which sends $S\in P'_n$ to a compactified ordered configuration space of $(\#S+1)$ points in $D^d$ with a unit tangent vector attached to each point, whose first and last points and their vectors are fixed at the boundary, and an inclusion $S\subset T$ to a map replacing a point in a configuration with a set of points which are infinitesimally close and arranged in the direction of the vector  on the point, and $\EEE_n\langle[D^d]\rangle$ is a functor which sends $S$ to the union of the spaces of $\DD_n\langle[D^d]\rangle$ and $\EEE_n(D^d)$ associated to $S$, which is topologized so that points in the former space are limits of shrinking punctured knots in the latter space until they become tangent vectors. (See subsections 5.3 and 5.4 of \cite{sinha} for the precise definitions. Precisely speaking, in \cite{sinha}, the labeling of $\{J_i\}$ is reverse to ours and only constant speed embeddings are considered but these differences do not cause any serious problem.) Let $G_n(D^d)$ be the  cubical model of the space of immersions $\mathrm{Imm}_c(I,D^d)$ satisfying the same endpoint condition as $\Emb_c(I,D^d)$  defined by sending $S$  to $(S^{d-1})^{\#S-1}$ and an inclusion $S\subset T$ to a diagonal map modeled by cutting off components of a  punctured knot  (see Definition 5.14 of \cite{sinha}, where the functor is denoted by $\mathcal{G}^m_k$). A termwise homotopy equivalence $\EEE_n(D^d)\to \tilde \PK$ is given by shrinking sub-intervals. This equivalence and those in the diagram (\ref{EQPK})  fit into the following commutative diagram
\[
\xymatrix{\tilde \PK\ar[d] &\EEE_n(D^d)\ar[r]\ar[l]\ar[d]&\EEE_n\langle[D^d]\rangle\ar[d] &\EEE_n(D^d)\ar[l]\ar[d] \\
G_n(D^d) &\ar[r]_{=}\ar[l]^{=}G_n(D^d) &G_n(D^d) &G_n(D^d)\ar[l]^{=},}
\]
where the leftmost vertical map sends an embedding to the collection of unit tangent vectors at the middle points of components other than the two components including endpoints, and the rightmost vertical map is the projection to the attached tangent vectors other than those on the first and last points, and the other vertical maps are defined similarly. Since all the vertical ones are fibrations, their fibers taken by the termwise manner are also termwisely homotopy equivalent. The fiber of the leftmost one admits a natural inclusion from $\PK$, which is a termwise homotopy equivalence. It is proved in \cite{sinha1} that the fiber of the rightmost map is termwisely homotopy equivalent to $\mathcal{F}^*\KK_d^{\leq n}$ ($\rho^m_k$ defined in the paragraph before Proposition 5.16 of  \cite{sinha1} is the same as the rightmost map). Thus, we have proved the claim.
\end{proof}

\begin{defi}\label{Dpunctured}
\begin{enumerate}
\item For $P\in \PP_n$, we define a positive number $\eps_P$ by
\[
\eps_P=\frac{\eps}{8^{n-p}},\qquad \text{where}\ \ p=|P|-2.
\]
Here, $\eps$ is the number fixed in Definition \ref{DPK}.

\item 
For $P\in \PP_n$ and $\alpha, \beta \in P$ with $\alpha<\beta$, we set 
\[
\begin{split}
c_\alpha:=& \ \sum_{i\in \alpha}c_i, \\
c_{\leq \alpha}:= & \ c_\alpha/2+\sum_{\gamma\in P, \gamma<\alpha}c_\gamma, \\
c_{\geq \alpha}:= & \ c_\alpha/2+\sum_{\gamma\in P, \gamma>\alpha}c_\gamma, \\
c_{\alpha \beta}:= & \ (c_\alpha+c_\beta)/2, \\
c_{\alpha\to \beta}:= & \ c_{\alpha\beta}+ \sum_{\gamma\in P, \alpha <\gamma< \beta}c_\gamma. 
\end{split}
\]
In fact, this definition does not depend on the pieces of $P$ other than $\alpha$ (and $\beta$). We abbreviate $c_{\leq \{i\}}$ (resp. $c_{\geq \{i\}}$, $c_{\{i\} \{j\} }$, $c_{\{i\}\to\{j\}}$) as $c_{\leq i}$ (resp. $c_{\geq i}$, $c_{ij}$, $c_{i\to j}$).
\item Let $Q$ be a subdivision of $P$ and write $P=\{\alpha_0<\cdots <\alpha_{p+1}\}$ and $Q=\{\beta_0<\cdots <\beta_{q+1}\}$. We define an affine monomorphism $e_{P,Q}:\RR^{dp}\to \RR^{dq}$ as follows. Let $(x_i)_{1\leq i\leq p}\in (\RR^d)^p$ be an element. For convenience, we set 
\[
x_0=(-1+\rho c_{\alpha_0}/2)u,\quad x_{p+1}=(1-\rho c_{\alpha_{p+1}}/2)u.
\] For $0\leq i\leq p+1$, suppose that $\alpha_i$ includes exactly $k$-pieces of $Q$, say $\beta_l,\dots, \beta_{l+k-1}$. We create the line segment which is centered at $x_i$, parallel to $u$, and of length $\rho c_{\alpha_i}$, and divide this segment into the $k$ little segments of length $\rho c_{\beta_l},\dots ,\rho c_{\beta_{l+k-1}}$ arranged from left to right ($-u$ to $u$). Let $y_{l+j-1}$ be the center of the $j$-th little segment. We set $e_{P,Q}((x_i)_{i})=(y_{m})_{1\leq m\leq q}$. For $Q=[n+1]$  we write $e_{P,Q}=e_P$ (see Figure \ref{Flinear_emb}). It is clear that $e_{Q,R}\circ e_{P,Q}=e_{P,R}$ for a subdivision $R$ of $Q$. 
\item For $P\in \PP_n$, let $P^\circ\subset P$ be the subset of pieces which are neither the minimum nor the maximum. Put $p=\# P-2$. We often label the $i$-th component of an element of $\RR^{dp}=(\RR^d)^p$ with the $i$-th piece of $P^\circ$, so an element of $\RR^{dp}$ is expressed  like $(x_\alpha)_{\alpha\in P^\circ}$.
\item We define a functor $\mathcal{C}:\PP_n\to \CG$ as follows.
\begin{enumerate}
\item  For $P\in \PP_n$, $\mathcal{C}(P)$ is the subspace of $\RR^{dp}$ consisting of elements $(x_\alpha)_{\alpha\in P^\circ}$ satisfying the following inequalities for each $\alpha$ :
\[
\begin{split}
|x_\alpha|& < 1-\rho c_\alpha/2,\  \text{and} \\
-1+\rho c_{\leq \alpha}-\frac{\eps_P}{8}\  &<_1 x_\alpha  <_1\  1-\rho c_{\geq \alpha}+\frac{\eps_P}{8},\ \text{and} \\
|x_\alpha&-x_\beta|  > \rho c_{\alpha\beta}-\frac{\eps_P}{8}.
\end{split}
\]
\item If $Q$ is a subdivision of $P$,  the corresponding map $ \mathcal{C}(P)\to \mathcal{C}(Q)$ is given by the map $e_{P,Q}$ defined above. This is clearly well-defined.

\end{enumerate} 
\item We can define an inclusion $\mathcal{C}(P)\to \PK (P)$ by creating the line segment centered at $x_\alpha$ of length $\rho (c_{\alpha}-(c_{i_0}+c_{i_1})/4)$ for $(x_\alpha)_{\alpha\in P^\circ}$, where $i_0=\min \alpha$ and $i_1=\max\alpha$. These inclusions form a natural transformation $\mathcal{C}\to \PK$. 
\end{enumerate}

\end{defi}

\begin{figure}
\begin{center}
\input{LinearEmb.TEX}
\end{center}
\caption{the map $e_{P}$ for $P=\{\{01\},\{23\},\{4\},\{5\}\}$} \label{Flinear_emb}
\end{figure}
\begin{lem}\label{Lnathomologyiso}
For each $P\in \PP_n$, the inclusion $\mathcal{C}(P)\to \PK(P)$ is a homology isomorphism.
\end{lem}
\begin{proof}
Write $P=\{\alpha_0<\cdots <\alpha_{p+1}\}$. Let $\Conf_p(D^d)$ be the ordered configuration space of $p$ points in $D^d$. Since the map $\PK(P)\to \Conf_p(D^d)$ which takes a collection of line segments $(l_{\alpha_i})$ to the configuration of centers of $l_{\alpha_i}$ ($i\not=0, p+1$), is a homotopy equivalence, we only have to prove the composition $\mathcal{C}(P)\to \PK(P)\to \Conf_p(D^d)$ is a homology isomorphism. This is a codimension $0$ embedding and fits into the following commutative diagram.
\[
\xymatrix{H_*(\mathcal{C}(P))\ar[r]\ar[d] & H_*(\Conf_p(D^d)) \ar[d] \\
H^*(D^{dp}, \Delta'_{\fat}\cup \partial'D^{dp})\ar[r] & H^*(D^{dp}, \Delta_{\fat}\cup \partial D^{dp})}
\]
Here, $D^{dp}=(D^d)^p$, and $\partial' D^{dp}\subset D^{dp}$ is the subspace of elements which do not satisfy at least one of the first two inequalities in (4),(a) of Definition  \ref{Dpunctured} for some $\alpha \in P$ ( so this is a collar of $\partial D^{dp}$). $\Delta'_{\fat}$ is the subspace of elements which do not satisfy the third inequality of the same definition for some pair $\alpha,\beta \in P$. $\Delta_{\fat}$ is the fat diagonal of $(D^d)^p$.   The vertical arrows are Poincar\'e-Lefschetz duality isomorophisms, and the bottom horizontal arrow is induced by the identity. We consider the $\CECHC$ech spectral sequences of pairs $(\Delta'_{\fat}, \Delta'_{\fat}\cap \partial'D^{dp})$ and $(\Delta_{\fat}, \Delta_{\fat}\cap \partial D^{dp})$ with respect to the coverings $\{\Delta'_{\alpha\beta}\}_{\alpha,\beta}$ and $\{\Delta_{\alpha \beta}\}_{\alpha,\beta}$, where $\Delta'_{\alpha\beta}$ is the subspace of elements which do not satisfy the third inequality in (4)(a) of the definition for $\alpha,\beta$ and $\Delta_{\alpha\beta}$ is the subspace of elements whose $\alpha$- and $\beta$- components are the same. The inclusion $\Delta_{\alpha\beta}\to \Delta'_{\alpha\beta}$ is clearly a homotopy equivalence. The inclusion $\Delta_{\alpha\beta}\cap \partial D^{dp}\to \Delta'_{\alpha\beta}\cap \partial' D^{dp}$ is also a homotopy equivalence since its homotopy inverse is given by the orthogonal projection to $\Delta_{\alpha\beta}$ followed by the projection to $\partial(D^{dp})$ from the light source $0$. So, the inclusion induces an isomorphism between the relative homology of the pairs. We also see that the inclusion  induces an isomorphism between the relative homology of intersections similarly, so by the  spectral sequence, we see that the bottom arrow of the square  is an isomorphism.
\end{proof}

\section{Thom space model $\TT$}\label{SThom}
In this section, we define a functor $\TT:\PP_n^{op}\to \CG_*$ consisting of Thom spaces  of fat diagonals. We prove that the chain of this functor is  equivalent to the cochain of  the punctured knot model $\PK$. 

\begin{defi}\label{Dpartition} Let $P\in \PP_n$ be a partition and set $p=\# P-2$.
\begin{enumerate}

\item Let $\nu_{P}$ be the $\eps_P$-neighborhood of $e_P(\RR^{dp})$ i.e.
\[
\nu_P=\{y\in \RR^{dn}\mid |y-e_P(x)|<\eps_P \text{ for some } x\in \RR^{dp}\}.
\]  For  $\alpha, \beta\in P^\circ$,  we define a   subspace $D_{\alpha\beta}$ of $\RR^{dp}=(\RR^d)^p$ by
\[
D_{\alpha\beta}=D_{\alpha\beta}(P)=\{(x_\gamma)_{\gamma \in P^\circ} \mid |x_\alpha-x_\beta|\leq d_{\alpha\beta}(P)\}, 
\]
where 
\[
d_{\alpha\beta}(P)=\rho c_{\alpha \beta}-\eps_P.
\]
We denote by $E_{\alpha}=E_\alpha(P)$  the subspace of  elements $(x_\gamma)_\gamma$ satisfying the following condition:
\[
\begin{split}
|x_\alpha|& \geq 1-\rho c_\alpha/2+\eps_P,\  \text{or} \\
  x_\alpha&\leq_1\, -1+\rho c_{\leq \alpha}-\eps_P,\  \text{or} \\  x_\alpha&\geq_1\,  1-\rho c_{\geq \alpha}+\eps_P.
\end{split}
\] 
Set 
\[
E_P=\bigcup_{\alpha\in P^\circ}E_\alpha.
\] 
Let $\pi_P:\RR^{dn}\to e_P(\RR^{dp})=\RR^{dp}$ be the orthogonal projection. 
We set
\[
\TT_{\emptyset_P}=\RR^{dn}/\{(\RR^{dn}-\nu_P)\cup \pi_P^{-1}(E_P)\}.
\]
 For two pieces $\alpha, \beta\in P^\circ$ with $\alpha<\beta$, let $\TT_{\alpha\beta}$ denote the subspace of $\TT_{\emptyset_P}$ consisting of the basepoint and the elements represented by  $x\in \nu_P$ satisfying $\pi_P(x)\in D_{\alpha\beta}$.

\end{enumerate}
\end{defi}

\begin{lem}\label{Ldiagonal_bound}
Let $Q$ be a subdivision of a partition $P\in \PP_n$. We have
\[
(\RR^{dn}-\nu_Q)\cup \pi_Q^{-1}(E_Q)\subset (\RR^{dn}-\nu_P)\cup \pi_P^{-1}(E_P).
\]
In particular, the identity on $\RR^{dn}$ induces the collapsing map 
$\delta'_{P,Q}:\TT_{\emptyset_Q}\to \TT_{\emptyset_P}.
$
\end{lem}
\begin{proof}
 Let $y\in \RR^{dn}$.
Write
\[
\begin{split}
\pi_{Q}(y)&=(x_\gamma)_{\gamma\in Q^\circ}, \\
\pi_{P}(y)&=(\bar x_{\gamma'})_{\gamma'\in P^\circ},\\
e_{P,Q}(\pi_P(y))&=(x'_\gamma)_{\gamma\in Q^\circ} .
\end{split}
\] Suppose $y\in \nu_P$.  Since the image of $e_P$ is contained in the image of $e_Q$ and the map $\pi_{Q}$ sends $y$ to its closest point in the image of $e_Q$, we have
\begin{equation}
|y-e_Q(x_\gamma)|\leq |y-e_Q(x'_\gamma)|=|y-e_P(\pi_P(y))|< \eps_P\ (<\eps_Q). \label{Eemb}
\end{equation}
So we see $y\in \nu_Q$. This means $\RR^{dn}-\nu_Q\subset \RR^{dn}-\nu_P$. We shall show $\nu_Q\cap \pi_Q^{-1}(E_Q)\subset (\RR^{dn}-\nu_P)\cup \pi_P^{-1}(E_P)$. Let $\alpha\in Q^\circ$ be a piece and $\alpha'$  the piece of $P$ including $\alpha$. If $y\not\in \nu_P$, we have nothing to prove, so suppose  $y\in \nu_P\cap\pi_Q^{-1}(E_Q)$ and we shall prove $y\in \pi_P^{-1}(E_P)$. Intuitively speaking, $x_\alpha$ and $\bar x_{\alpha'}$ are different in general but the difference is sufficiently small since we have chosen the radius of the tubular neighborhood $\nu_P$ sufficiently small. Since $\eps_P<\eps_Q$ and $\alpha\subset \alpha'$, the range given by the inequalities defining $E_{\alpha'}(P)$ is wider than the one given by the inequalities of $E_\alpha(Q)$ and the margin covers the difference between $x_\alpha$ and $\bar x_{\alpha'}$ so we have $y\in \pi_P^{-1}(E_P)$. We shall give a rigorous proof.  Suppose further $|x_\alpha|\geq 1-\rho c_\alpha/2+\eps_Q$. 
By the inequality (\ref{Eemb}), we have
\[
|e_Q(x'_\gamma)-e_Q(x_\gamma)|\leq |e_Q(x'_\gamma)-y|+|y-e_Q(x_\gamma)|\leq 2\eps_P.
\]
As $e_Q$ is a composition of a diagonal map with a parallel transport, we have
\begin{equation}
|(x'_\gamma)_\gamma-(x_\gamma)_\gamma |\leq 2\eps_P \quad (\text{so} \ |x'_\gamma-x_\gamma|\leq 2\eps_P\ \text{for each}\ \gamma\in Q^\circ). \label{Edifference}
\end{equation}
Let $\beta\subset \alpha'$ be the  set of elements smaller than the minimum of $\alpha$. We easily see 
\begin{equation}
\bar x_{\alpha'}-x'_\alpha=\frac{\rho}{2}(c_{\alpha'}-2c_\beta-c_\alpha)u. \label{Earrange}
\end{equation}
Putting these (in)equalities together, we see
\[
\begin{split}
|\bar x_{\alpha'}|& \geq |x_\alpha|-|x_\alpha-x'_\alpha|-|x'_\alpha-\bar x_{\alpha'}| \\
& \geq 1-\rho c_\alpha/2+\eps_Q-2\eps_P-\rho(c_{\alpha'}-c_\alpha)/2 \\
& \geq 1-\rho c_{\alpha'}/2+\eps_P
\end{split}
\]
since we have 
\[
\eps_Q-2\eps_P=\frac{1-2\cdot 8^{p-q}}{8^{n-q}}\geq \frac{1-2/8}{8^{n-q}}>\eps_P.
\]
We have shown that the first inequality in Definition \ref{Dpartition}(2) for $Q$ and the condition $y\in \nu_P$ imply the corresponding inequality for $P$. 
Similarly, suppose $x_\alpha\leq_1 -1+\rho c_{\leq \alpha}-\eps_Q$ (and $y\in \nu_P$). Let $\alpha'$ be  as above.
We see 
\[
\begin{split}
\bar x_{\alpha'} & =\bar x_{\alpha'}-x'_\alpha+x'_\alpha-x_\alpha+x_\alpha \\
& \leq_1 \frac{\rho}{2}(c_{\alpha'}-2c_\beta-c_\alpha)+2\eps_P+(-1+\rho c_{\leq \alpha}-\eps_Q) \\
&=-1+\rho c_{\leq \alpha'}-(\eps_Q-2\eps_P)  \quad (\because \ c_{\leq \alpha}+(c_{\alpha'}-2c_\beta-c_\alpha)/2=c_{\leq \alpha'}) \\
&<-1+\rho c_{\leq \alpha'}-\eps_P.
\end{split}
\]
This is  the second inequality in Definition \ref{Dpartition}(2) for $P$. 
Similarly, we see $\bar x_{\alpha'}\geq_1 1-\rho c_{\geq \alpha'}+\eps_P$ if $ x_{\alpha}\geq_1 1-\rho c_{\geq \alpha}+\eps_P$ and $y\in \nu_P$. Thus, we have shown the claimed inclusion.
\end{proof}
\begin{lem}\label{Ldiagonal_incl}
Let $Q$ be a subdivision of $P$ and $\alpha, \beta \in Q^\circ $ pieces with $\alpha<\beta$.  Let $\alpha', \beta' \in P$ be the pieces which include $\alpha$, $\beta$ respectively. Let $\delta'_{P,Q}:T_{\emptyset_Q}\to T_{\emptyset_P}$ denote the map given in Lemma \ref{Ldiagonal_bound}.
We have 
\[
\delta'_{P,Q}(\TT_{\alpha \beta})\ \subset \
\left\{
\begin{array}{cc}
\{*\} & (\text{if $\alpha'=\beta'$, or $\alpha'$ is the minimum of $P$, or $\beta'$ is the maximum of $P$}), \\
\TT_{\alpha' \beta'} & (\text{ otherwise }).
\end{array}
\right.
\] 

\end{lem}
\begin{proof}
Let $y\in \RR^{dn}$ be an element. We use the same notations $ x_\gamma, x'_\gamma$, and $\bar x_{\gamma'}$ as in the proof of Lemma  \ref{Ldiagonal_bound}.
 We shall show the claim in the case $\alpha'=\beta'$. We assume $\delta'_{P,Q}(y)\not =*$. So $y\in \nu_P$. By the definition of the map $e_{P,Q}$, the distance between $x'_\alpha$ and $x'_\beta$ is exactly $\rho c_{\alpha\to\beta}$. Intuitively speaking, since we have chosen the radius of $\nu_P$ sufficiently small, the distances between $x'_\alpha$ and $x_\alpha$, $x'_\beta$ and $x_\beta$ are sufficiently small. The difference between  $\rho c_{\alpha\to\beta}$ and the diameter $d_{\alpha\beta}(Q)$ of the tubular neighborhood $D_{\alpha\beta}$ of the diagonal is too large to be covered by the distances between the points so we see $\pi_Q(y)\not \in D_{\alpha\beta}$. We shall give a rigorous proof.
By the inequality (\ref{Edifference}) in the proof of Lemma  \ref{Ldiagonal_bound} and the definition of the map $e_{P,Q}$, we have the following inequality.
\[
\begin{split}
|x_\alpha-x_\beta| & \geq |x'_\alpha-x'_\beta|-|x_\alpha-x'_\alpha|-|x_\beta-x'_\beta| \\
                          & \geq \rho c_{\alpha\beta}-4\eps_P>d_{\alpha\beta}(Q).
\end{split}
\]
This inequality implies $\pi_{Q}(y)\not\in D_{\alpha\beta}$. \\
\indent We shall show the claim in the case that $\alpha'$ is the minimum. It is enough to show the case of $\beta=\beta'$ since general subdivisions factor through this case. In this case, by definition of $e_{P,Q}$, $x'_\alpha=(-1+\rho c_{\leq \alpha})u$. Suppose $y\in \nu_P\cap \pi_Q^{-1}(D_{\alpha\beta})$. We have 
\[
\begin{split}
x_\beta&=  x'_\alpha+(x_\beta-x_\alpha)-(x'_\alpha-x_\alpha) \\
 &\leq _1-1+\rho c_{\leq \alpha} +|x_\beta-x_\alpha|+|x'_\alpha-x_\alpha| \\
&\leq -1+\rho c_{\leq \alpha}+\rho c_{\alpha\beta}-\eps_Q+2\eps_P \\
&< -1 +\rho c_{\leq \beta}-\eps_P.
\end{split}
\]
So we have the claim. The case that $\beta'$ is the minimum is completely similar.\\
 We shall show the remaining part of the claim. 
 By definition, we have $e_{P,Q}(\bar x_{\gamma'})=(x'_\gamma)$. Suppose $y\in \nu_P\cap \pi_Q^{-1}(D_{\alpha\beta})$ again. By the inequality (\ref{Edifference}) in the proof of Lemma  \ref{Ldiagonal_bound}, we have
\[
\begin{split}
|x'_\alpha-x'_\beta|\leq  &  |x'_\alpha-x_\alpha|+|x_\alpha-x_\beta|+|x_\beta-x'_\beta| \\
\leq & 4\eps_P+d_{\alpha\beta}(Q)
=\rho c_{\alpha\beta}-\eps_Q+4\eps_P.
\end{split}
\]
 By the equality (\ref{Earrange}) in the proof, we have 
\[
|\bar x_{\alpha'}-x'_\alpha|\leq \rho (c_{\alpha'}-c_{\alpha})/2.
\]
By similar equality for $\beta$, we see
\[
\begin{split}
|\bar x_{\alpha'}-\bar x_{\beta'}|\leq  & |\bar x_{\alpha'}-x'_\alpha|+|x'_\alpha-x'_\beta|+ |\bar x_{\beta'}-x'_\beta| \\
\leq &\rho (c_{\alpha'}-c_{\alpha})/2+(\rho c_{\alpha\beta}-\eps_Q+4\eps_P)+\rho (c_{\beta'}-c_{\beta})/2\\
<&   d_{\alpha'\beta'}(P).
\end{split}
\]
Thus, we have shown the claim.
\end{proof}

\begin{defi}\label{Dspectrafunctor}
\begin{enumerate}
\item For $P\in \PP_n$, we define a space $\TT(P)\in \CG_*$ by 
\[
\TT(P)=\TT_{\emptyset_P} / \TT_{fat},\quad \text{where} \quad \TT_{fat}=\bigcup_{\alpha,\beta \in P^\circ, \alpha<\beta}\TT_{\alpha\beta}.
\]
 If $Q$ is a subdivision of $P$, the map $\delta'_{P,Q}:\TT_{\emptyset_Q}\to \TT_{\emptyset_P}$ in Lemma \ref{Ldiagonal_bound} induces the map $\TT(Q)\to \TT(P)$ by Lemma \ref{Ldiagonal_incl}.  These spaces and maps form a functor $\TT:(\PP_n)^{op}\to \CG_*$. 
\item A {\em spectrum} $X$ is a sequence of pointed spaces $X_0, X_1, \dots$ with a  structure map $S^1\wedge X_k\to X_{k+1}$ for each $k\geq 0$. A {\em morphism (or map)} $f:X\to Y$ of spectra is a sequence of pointed maps $f_0:X_0\to Y_0, f_1:X_1\to Y_1,\dots$ compatible with the structure maps. Let $\SP$ denote the category of spectra and their maps. For a spectrum $X$,  $\pi_k(X)$ denotes the colimit of the sequence $\pi_k(X_0)\to \pi_{k+1}(X_1)\to \cdots $ defined by the structure maps. A map $f:X\to Y$ is called a {\em stable homotopy equivalence} if it induces an isomorphism $\pi_k(X)\to \pi_k(Y)$ for any integer $k$.
\item For a spectrum $X$ and unpointed space $U$, We define a spectrum $\Map(U,X)$ as follows. We define  $\Map(U,X)_k$ as the space  of (unpointed) continuous maps $U\to X_k$ with the compact-open topology. The basepoint is the constant map to the basepoint of $X_k$. The structure map is the one obviously induced by that of $X$.
\item We define a functor $\TT^S: \PP_n^{op}\to \SP$ as follows. Set $\TT^S(P)_k=S^{k-dn}\wedge \TT(P)$ if $k\geq dn$,  and $\TT^S(P)_k=*$ otherwise. These spaces form a spectrum with the obvious structure map. The map corresponding to a map $P\to Q$ is also obviously induced from that of $\TT$. 
\item For a positive number $\delta$, we define a spectrum $\Sphere_\delta$ as follows. We set $\Sphere_{\delta, k}=\{ y\in \RR^k\}/\{y \mid |y|\geq \delta\}$. The structure map $S^1\wedge \Sphere_{\delta, k}\to \Sphere_{\delta, k+1}$ is the obvious collapsing map.
\item We define a functor $\mathcal{C}^\dagger:\PP_n^{op}\to \SP$ as follows.
Set $\mathcal{C}^{\dagger}(P)=\Map(\mathcal{C}(P),\Sphere_{\delta})$ where  $\delta =\eps_P/8$ (see Definition \ref{Dpartition} (1)). For a map $P\to Q$, the corresponding map is the pullback by the induced map $\mathcal{C}(P)\to \mathcal{C}(Q)$ followed by the pushforward by the collapsing map $\Sphere_{\eps_Q/8}\to \Sphere_{\eps_P/8}$.

\item We define a functor $\mathcal{C}^\vee:\PP_n^{op}\to \SP$ as follows. Let $\Sphere$ denote the sphere spectrum given by $\Sphere_k=S^k$.
Set $\mathcal{C}^{\vee}(P)=\Map(\mathcal{C}(P),\Sphere)$.   For a map $P\to Q$, the corresponding map is the pullback by the induced map.
\item We define a map $\TPhi= \TPhi_{P,k}:\RR^k\to \mathcal{C}^\dagger (P)_k$ by
\[
\RR^k\ni y\longmapsto \{(x_\gamma) \mapsto (y-(0,e_P(x_\gamma))\}\in\mathcal{C}^\dagger(P)_k
\]
$\TT^S_k$ is naturally identified with Thom space associated to the  tubular neighborhood  $\RR^{k-dn}\times \nu_P$ of the embedding $0\times e_P:\RR^{dp}\to \RR^k$ (with some extra collapsed points). $\TPhi_{P,k}$ factors through $\TT^S(P)_k$ as in Theorem  \ref{TAtiyahdual}, and  these maps form a natural transformation $\Phi : \TT^S\to \mathcal{C}^\dagger$. We see that this is well-defined below.
\item A natural transformation $p_* :\mathcal{C}^\vee\to \mathcal{C}^\dagger$ is defined by the pushforward by the obvious collapsing map $p :\Sphere\to \Sphere_{\delta}$. 
\end{enumerate}
\end{defi}
The following equivalence is a variation of the one given in \cite{moriya1} which is based on the construction in \cite{cohen}. If it is projected to the stable homotopy category, it is a special case of Atiyah duality which states an equivalence between the Spanier-Whitehead dual of a manifold and Thom spectrum of its normal bundle. We need point-set level compatibility so we have been taking care about parameters. 
\begin{thm}\label{TAtiyahdual}
Under the notations of Definition \ref{Dspectrafunctor},
the map $\Phi$ is well-defined, and the two natural transformations $\Phi$ and $p_*$ are termwise stable homotopy equivalences (i.e. they induce a stable homotopy equivalence at each object).
\end{thm}
\begin{proof}
We shall show the map $\TPhi$ factors through $\TT^S(P)_k$. For notational simplicity, we consider the case of $k=dn$. The other cases will follow completely similarly. It is clear that $\TPhi(\RR^{k}-\nu_P)=\{*\}$. Let $y\in \RR^k$ be an element with $\TPhi(y)\not= *$.  There exists an element $(x_\gamma)\in \mathcal{C}(P)$ such that  $| y- e_P(x_\gamma)|<\eps_P/8$ holds. So we have $|y- e_P(\pi_{P}y)|<\eps_P/8$ and
\[
|\pi_P y-(x_\gamma)|\leq|e_P(\pi_P y)-e_P(x_\gamma)|\leq | e_P(\pi_P y)-y|+|y-e_P(x_\gamma)|<\eps_P/4.
\]
 If we write  $\pi_P(y)=(\bar x_\gamma)$, it follows that $|\bar x_\alpha-x_\alpha|<\eps_P/4$ for each  $\alpha\in P^\circ$. 
We see
\[
\begin{split}
|\bar x_\alpha| & \leq |x_\alpha|+|\bar x_\alpha-x_\alpha| \\
& \leq 1-\rho c_\alpha/2+\eps_P/4<1-\rho c_\alpha/2+\eps_P, \\
\bar x_\alpha  & =x_\alpha+(\bar x_\alpha-x_\alpha) \\
&\geq_1-1+\rho c_{\leq \alpha}-|\bar x_\alpha-x_\alpha|  \\
& \geq -1+\rho c_{\leq \alpha}-\eps_P/4>-1+\rho c_{\leq \alpha}-\eps_P, \\
\bar x_\alpha & = x_\alpha+(\bar x_\alpha-x_\alpha)  \\
 & \leq_1 1-\rho c_{\geq \alpha} +|\bar x_\alpha-x_\alpha|  \\
& \leq 1-\rho c_{\geq \alpha}+\eps_P/4<1-\rho c_{\geq \alpha}+\eps_P. \\
\end{split}
\]
These inequalities imply $\pi_P(y)\not\in E_\alpha$ and so $\Phi(\pi_P^{-1}(E_\alpha))=*$ in the notations of Definition  \ref{Dpartition}. We also see
\[
\begin{split}
|\bar x_\alpha -\bar x_\beta| & \geq  |x_\alpha-x_\beta|-|x_\alpha-\bar x_\alpha |-|x_\beta-\bar x_\beta| \\
 & > \rho c_{\alpha\beta}-\eps_P/8-\eps_P/2  
  > d_{\alpha\beta}(P).
\end{split}
\]
This implies $\TPhi(\pi_P^{-1}(D_{\alpha\beta}))=*$. Thus, $\TPhi$ factors through $\TT^S$. Now the claim of the theorem follows from the classical Atiyah duality  (see \cite{browder} for example).

\end{proof}
\begin{rem}\label{RMalin}
In \cite{malin}, Malin  proved homotopy invariance of the homogeneous layers of stable embedding calculus tower using  a duality similar to Theorem \ref{TAtiyahdual} or Theorem 1.1 of \cite{moriya1}.
\end{rem}
\begin{defi}\label{Dchain_spectra}
\begin{enumerate}
\item Let $\CH_{\kk}$ be the category of chain complexes and chain maps over  $\kk$. We mainly use cohomological grading denoted by a superscript.  Homological grading, which is denoted by a subscript, is regarded as  cohomological grading by negation.
\item For a chain complex $C_*$, $C_*[k]$ is the chain complex given by $C_l[k]=C_{k+l}$ with the same differential as $C_*$ (without extra sign).
\item The functors $C^*(\PK),\ \bar C_*(\TT):\PP_n^{op}\to \CH_{\kk}$ are given by taking cochains and reduced chains of $\PK$ and $\TT$ in the termwise manner respectively. $\bar C_*(\TT)[-dn]:\PP_n^{op}\to \CH_{\kk}$ is given by taking the shift of $\bar C_*(\TT)$ in the termwise manner.  
\end{enumerate}
\end{defi}
\begin{prop}\label{Phom_dual}
The functors $C^*(\PK)$ and $\bar C_*(\TT)[-dn]$ are connected by  a zigzag of termwise quasi-isomorphisms (i.e. natural transformations which induce a quasi-isomorphism at each ojbect).  
\end{prop}
\begin{proof}
In \cite{moriya1}, a chain functor $C_*$ for symmetric spectra is defined. The same definition works for our category of spectra as it is, so we adopt this functor in this proof.  It preserves stable equivalence between semistable spectra, and the spectra involved here are semistable.  By Lemma 5.3 of \cite{moriya1}, Lemma \ref{Lnathomologyiso}, and Theorem \ref{TAtiyahdual}, we have the following chain of termwise quasi-isomorphisms
\[
C^*(\PK)\simeq C^*(\mathcal{C})\simeq C_*(\mathcal{C}^\vee)\simeq C_*(\TT^S)\simeq \bar C_*(\TT)[-dn],
\]
where the last morphism is the canonical one in view of definition of the chain functor.
\end{proof}

\section{Spectral sequences}\label{Sspec_seq}
\begin{defi}\label{Dgraph}
\begin{enumerate}
\item  For a partition $P\in \PP_n$, a {\em graph $G$ on   $P$} consists of the set of vertices $V(G)=P$, a finite set of edges $E(G)$ and a map $\phi_G:E(G)\to P_{1,2}(P)$ called the {\em incidence map}, where $P_{1,2}(P)=\{S\subset P\mid \# S=1 \text{ or } 2\}$. So, the vertices of $G$ are the pieces of $P$. We say that an element of $\phi_G(e)$ is {\em incident} with  $e$, or $e$ is incident with elements of $\phi_G(e)$.  An edge $e$ is called a {\em loop} if $\#\phi_G(e)=1$. Two edges $e,e'$ are called {\em double edges} if $\phi_G(e)=\phi_G(e')$ and $\#\phi_G(e)=2$. (Other edges may have the same set of incident  vertices  as double edges.) A graph is called a {\em forest} if each connected component of its geometric realization (in the usual sence) is contractible. Let $\TGG(P)$ denote the set of all graphs on  $P$.  
$\GG(P)\subset \TGG(P)$ denotes the subset of graphs with an edge set $E(G)\subset \{(\alpha,\beta)\mid \alpha, \beta \in P^\circ,\ \alpha<\beta\}$ and the natural incidence map $(\alpha,\beta)\mapsto \{\alpha,\beta\}$. For $G\in \GG(P)$, $E(G)$  is regarded as a totally ordered set by the lexicographical order. Let $\emptyset_P\in \GG(P)$ denote the graph with  the empty edge set. We  sometimes denote a graph in $\GG(P)$ by a formal product of edges (see Example \ref{Egraph} below).   Let $e$ be the $i$-th edge of $G\in \GG(P)$. $\partial_eG$ (or $\partial_iG$) $\in \GG(P)$ denotes the subgraph of $G$ with $E(\partial_eG)=E(G)-\{e\}$. Similarly, if $e'$ is the $j$-th edge, $\partial_{ee'}G$ (or $\partial_{ij}G$) denotes the subgraph made by removing $e$ and $e'$. For two vertices $\alpha, \beta$ of a graph $G$, we write $\alpha\sim_G\beta$ when $\alpha$ and $\beta$ belong to the same connected component of $G$. By abusing notations, for $i\in [n+1]$, we write $i\sim_G\beta$ if there is a (unique) piece $\alpha$ satisfying $i\in \alpha$ and $\alpha\sim_G\beta$. Similarly, if $i$ belongs to a piece $\alpha$ which belongs to a connected component $S$ of $G$, we write $i\in S$ instead of $i\in \alpha\in S$. For $i,j \in [n+1]$, $i\sim_Gj$ is similarly understood.

\item For a map $P\to Q$ of partitions,  $\delta_{P,Q}:Q\to P$ denotes the map of sets sending  $\alpha\in Q$ to the piece of $P$ containing $\alpha$. This map induces a map $\delta_{P,Q}:\TGG(Q)\to \TGG(P)$. For a graph $G\in \TGG(Q)$, the  graph $\delta_{P,Q}(G)$ has the same edge set as $G$  and the incidence map  given by the composition
\[
E(G)\stackrel{\phi_G}{\to}P_{1,2}(Q)\stackrel{(\delta_{P,Q})_*}{\to}P_{1,2}(P).
\] While $E(\delta_{P,Q}(G))=E(G)$ by definition, we sometimes refer to this identification  the {\em standard bijection} for clarity. If $\delta_{P,Q}(G)$ has neither loops nor double edges and the minimum and maximum pieces of $P$ are discrete i.e. not incident with any edge of $\delta_{P,Q}(G)$, we always identify $\delta_{P,Q}(G)$ with a unique graph in $\GG(P)$ which has the same image of the incidence map, and we write $\delta_{P,Q}(G)\in \GG(P)$ so we use this expression even if $E(G)\not \subset \{(\alpha,\beta)\mid \alpha, \beta \in P^\circ ,\ \alpha<\beta\}$. Under this identification, we also identify an element of $E(G)$ which is  incident with $\alpha<\beta$, with the edge $(\alpha,\beta)$. The composition of this identification with the standard bijection is also called the standard bijection. Let $H\in \TGG(Q)$ be another graph with $\delta_{P,Q}(H)\in \GG(P)$. The equation $\delta_{P,Q}(G)=\delta_{P,Q}(H)$ means the equality of the corresponding graphs in $\GG(P)$, so the equation may hold even if $E(G)\not=E(H)$ (see Example \ref{Egraph} below). If $P$ is made by unifying the $i+1$-th and $i+2$-th pieces of $Q$, we denote $P$ and the map $\delta_{P,Q}$ between the partitions or the sets of graphs  by $\delta_iQ$ and $\delta_i$, respectively. For numbers $i<j$, we set $\delta_{ij}A=\delta_i\delta_jA$ for a partition or graph $A$.  Similarly, we set  $\delta_{ijk}=\delta_i\delta_j\delta_k$ for $i<j<k$.

\end{enumerate}
\end{defi}
\begin{exa}\label{Egraph}
Let $G_1, G_2 \in \GG([5])$ be the graphs   given by
\[
G_1=(1,4)(2,3),\quad G_2=(1,3)(2,4).
\]
These graphs are drawn  in Figure \ref{Fgraphs_ch2}. We have $E(\partial_1G_1)=\{(2,3)\}$ and $E(\partial_2G_1)=\{(1,4)\}$. $G_1$ has the four connected components $\{\{0\}\}$, $\{\{1\},\{4\}\}$, $\{\{2\},\{3\}\}$, and $\{\{5\}\}$. By definition, we have
\[
\delta_1[5]=\{\{0\},\{12\},\{3\},\{4\},\{5\}\}, 
\]
and $\delta_1G_1=\delta_1G_2$ since $E(\delta_1G_1)=E(\delta_1G_2)=\{(\{12\},\{3\}),(\{12\},\{4\})\}$. Similarly, we have $\delta_3G_1=\delta_3G_2$. The graph $\delta_1G_1$ has the three connected components $\{\{0\}\}$, $\{\{12\},\{3\}, \{4\}\}$, and $\{\{5\}\}$. 
\end{exa}
\begin{figure}
\begin{center}
{\unitlength 0.1in%
\begin{picture}(22.4000,4.0700)(0.4000,-4.3700)%
%
\special{pn 8}%
\special{pa 75 190}%
\special{pa 675 190}%
\special{fp}%
%
\special{pn 8}%
\special{pa 120 195}%
\special{pa 174 143}%
\special{pa 201 119}%
\special{pa 228 96}%
\special{pa 254 76}%
\special{pa 281 59}%
\special{pa 308 45}%
\special{pa 335 35}%
\special{pa 362 30}%
\special{pa 389 31}%
\special{pa 415 37}%
\special{pa 442 47}%
\special{pa 469 61}%
\special{pa 496 79}%
\special{pa 523 100}%
\special{pa 549 123}%
\special{pa 603 173}%
\special{pa 620 190}%
\special{fp}%
%
\special{pn 8}%
\special{pa 275 195}%
\special{pa 301 167}%
\special{pa 326 142}%
\special{pa 352 125}%
\special{pa 377 120}%
\special{pa 403 128}%
\special{pa 428 147}%
\special{pa 452 173}%
\special{pa 475 200}%
\special{fp}%
%
\special{pn 8}%
\special{pa 865 190}%
\special{pa 1470 190}%
\special{fp}%
%
\special{pn 8}%
\special{pa 930 190}%
\special{pa 956 163}%
\special{pa 981 136}%
\special{pa 1007 113}%
\special{pa 1033 93}%
\special{pa 1059 79}%
\special{pa 1085 71}%
\special{pa 1111 71}%
\special{pa 1138 78}%
\special{pa 1164 92}%
\special{pa 1191 111}%
\special{pa 1218 134}%
\special{pa 1245 159}%
\special{pa 1271 186}%
\special{pa 1275 190}%
\special{fp}%
%
\special{pn 8}%
\special{pa 1075 190}%
\special{pa 1101 163}%
\special{pa 1126 136}%
\special{pa 1152 113}%
\special{pa 1178 93}%
\special{pa 1204 79}%
\special{pa 1230 71}%
\special{pa 1256 71}%
\special{pa 1283 78}%
\special{pa 1309 92}%
\special{pa 1336 111}%
\special{pa 1363 134}%
\special{pa 1390 159}%
\special{pa 1416 186}%
\special{pa 1420 190}%
\special{fp}%
%
\special{pn 8}%
\special{pa 1675 190}%
\special{pa 2280 190}%
\special{fp}%
%
\special{pn 8}%
\special{pa 1730 190}%
\special{pa 1754 158}%
\special{pa 1778 130}%
\special{pa 1802 109}%
\special{pa 1825 100}%
\special{pa 1848 106}%
\special{pa 1871 124}%
\special{pa 1894 152}%
\special{pa 1916 184}%
\special{pa 1920 190}%
\special{fp}%
%
\special{pn 8}%
\special{pa 2025 190}%
\special{pa 2047 157}%
\special{pa 2070 129}%
\special{pa 2093 108}%
\special{pa 2116 100}%
\special{pa 2139 107}%
\special{pa 2163 125}%
\special{pa 2187 152}%
\special{pa 2211 184}%
\special{pa 2215 190}%
\special{fp}%
\put(3.7500,-2.9500){\makebox(0,0){$G_1$}}%
\put(11.7500,-2.9000){\makebox(0,0){$G_2$}}%
\put(19.7500,-2.8500){\makebox(0,0){$G_3$}}%
\put(1.3000,-4.9000){\makebox(0,0){$x$}}%
\put(2.9500,-4.9000){\makebox(0,0){$y$}}%
\put(4.5500,-4.9000){\makebox(0,0){$y$}}%
\put(6.1000,-4.9000){\makebox(0,0){$x$}}%
\put(9.3500,-4.9200){\makebox(0,0){$x$}}%
\put(11.0000,-4.9200){\makebox(0,0){$y$}}%
\put(12.6000,-4.9200){\makebox(0,0){$x$}}%
\put(14.1500,-4.9200){\makebox(0,0){$y$}}%
\put(17.5000,-5.0200){\makebox(0,0){$x$}}%
\put(19.1500,-5.0200){\makebox(0,0){$x$}}%
\put(20.7500,-5.0200){\makebox(0,0){$y$}}%
\put(22.3000,-5.0200){\makebox(0,0){$y$}}%
\end{picture}}%
\end{center}
\caption{graphs in sections \ref{Sspec_seq}, \ref{Scondensed} and \ref{Sch2}  and corresponding maps $f_G$ (see section \ref{Scondensed}) :  The straight line is not a part of data, and  the chords denote the edges. The vertices are the intersections of the line and chords. They are labeled  by $\{1\},\dots, \{4\}$ in the order from left to right. The minimum  and maximum vertices $\{0\}, \{5\}$ are usually omitted.  We use similar notations throughout the paper.} \label{Fgraphs_ch2}
\end{figure}

\begin{defi}\label{Dtriple}
 For a graph $G\in \GG(P)$, set
$\TT_G=\bigcap_{(\alpha,\beta)\in E(G)} \TT_{\alpha\beta}\subset T_{\emptyset_P}$. \\
\indent We define a triple complex $\TTH_{\bullet * \star}$ as follows. 
As a module, we set 
\[
\TTH_{p q s}=\bigoplus_{P, G}\bar C_q(\TT_G)
\]
where $P$ runs through the partitions of $\PP_{n}$ such that $\# P=p+2$ and $G$ runs through
the graphs of $\GG(P)$ such that $\#E(G)=s$. $\TTH$ has three differentials $\delta, d, \partial$ of degree $(1,0,0)$, $(0,1,0)$, $(0,0,1)$, respectively. $d$ denotes the differential on singular chains, $\partial$ is the $\CECHC$ech differential given by
\[
\partial=\sum_{1\leq k\leq s}(-1)^{k-1}\partial_k,
\]
where $\partial_k:\bar C_q(\TT_G)\to \bar C_q(\TT_{\partial_kG})$ is the pushforward by the inclusion. To define $\delta$, we shall define a map $\delta_i$ of degree $(1,0,0)$. 
Put $P=\{\alpha_0<\cdots <\alpha_{p+1}\}$.  If $\delta_iG\not\in \GG(\delta_iP)$,   we set $\delta_i=0$ on $\bar C_*(\TT_G)$. Suppose otherwise.
 The set of edges of $G$ whose smaller incident vertex  is $\alpha_i$ or $\alpha_{i+1}$ and the set of edges of $\delta_i G$ whose smaller incident vertex is $\alpha_i\cup \alpha_{i+1}$ are one-to-one correspondence via the standard bijection. This correspondence induces the permutation $\sigma_{G,i}$ of the lexicographical order of the edges. We set 
\[
\delta_i=sgn(\sigma_{G,i})(\delta'_{i})_*. 
\]Here, $\delta'_i:\TT_G \to \TT_{\delta_iG}$ is the collapsing map for subdivision $P$ of $\delta_iP$ given in Lemma \ref{Ldiagonal_incl}. (By definition of $\TT_G$ and the same lemma, the image of $\delta'_i$ is contained in $\TT_{\delta_iG}$.) We set 
$\delta=\sum_{i=0}^{p}(-1)^i\delta_i $.
\end{defi}
\begin{lem}\label{Lwelldeftriple}
The triple complex $\TTH$ is well defined i.e. the three differentials are commutative without signs.
\end{lem}
\begin{proof}
The only non-trivial commutativity is that of $\partial$ with $\delta$. Let $e$ be the $k$-th edge among those whose smaller incident vertex is $\alpha_i$. Suppose exactly $m$ edges pass through $e$ by $\sigma_{G,i}$. On $G$, the sign in removing $e$ is $(-1)^k$ and on $\delta_iG$, the sign in removing the corresponding edge is $(-1)^{k+m}$. On the other hand, we have $sgn(\sigma_{G,i})=(-1)^msgn(\sigma_{\partial_eG, i})$. So we have the commutativity.
\end{proof}

\begin{defi}\label{Dspecseq}
\begin{enumerate}
\item Let $\tot(\TTH)$ be the total complex of $\TTH$. Its homogeneous part of total cohomological degree $-k$ is  given by 
\[
\tot_k(\TTH)=\bigoplus_{p+q+s=k}\TTH_{pqs}
\]
and the total differential is given by $\tilde D=\delta+(-1)^pd+(-1)^{p+q}\partial$ on $\TTH_{pqs}$. (As in Definition \ref{Dchain_spectra}, cohomological degree is given by negation of the subscript.) Let $\{F_l\}_l$ be the filtration of $\tot(\TTH)$ given by $F_l=\oplus_{p\leq l}\TTH_{pqs}$. This filtration induces a spectral sequence $\{\CEE_r\}_r$. We give the grading on $\CEE_r$ such that $\CEE_r^{-p,q}$ is the part of the original induced cohomological degree $(-p,q-dn)$. 
 We also consider a truncated version $tr_m\tot(\TTH)$ which is the direct sum of  $\TTH_{pq s}$'s with  $s\geq m$. The differential is given by the same formula as $\tilde D$ but the $\CECHC$ech differentials on the graph with exactly $m$ edges are understood as zero. Considering the filtration on $p$, we form a spectral sequence $\{tr_m\CEE_r\}_r$. A  map $\tot(\TTH)\to tr_m\tot(\TTH)$ is given by the identity on $\TTH_{pqs}$ with $s\geq m$ and zero on the other summands.  This map commutes with differentials and filtrations so induces the map of spectral sequences 
$\CEE_r\to tr_m\CEE_r$.
\item For a functor $X: \Delta_n^{op}\to \CH_\kk$, the normalization $NX$ is the double complex given by
\[
N_pX=\left\{
\begin{array}{cc}
X_p/(\sum_{0\leq i\leq p-1}s_i(X_{p-1}))&\text{ if } p\leq n, \vspace{2mm}\\
0 & \text{ if } p>n,
\end{array}\right.
\]
where $s_i$ denotes the degeneracy map, with one of the differentials given by the signed sum of the face maps and the other by the original differential of $X$. Let $C^*(\KK^{\leq n}_d)$ be the functor  defined by taking cochain of the functor $\KK^{\leq n}_d$ in the termwise manner (see Definition \ref{DKontsevich}). We give the total complex of $NC^*(\KK^{\leq n})$ the filtration by cosimplicial degree $p$ and denoted by $\{\bar\EE_r\}_r$ the resulting spectral sequence. This is the version of  Sinha's  spectral sequence restricted to the part of cosimplicial degree $\leq n$.  The (full) Sinha's spectral sequence $\{\EE_r\}_r$ is defined similarly using the usual normalization of  the simplicial complex $C^*(\KK^\bullet_d)$.  
\end{enumerate}
Note that $\CEE_r$, $tr_m\CEE_r$ and $\bar\EE_r$ depend on the number $n$ in $\PP_n$ and $\Delta_n$. We always let the letter $n$ denote the subscript of the two categories used in   the spectral sequences.
\end{defi}

\begin{lem}\label{Lspecseq}
\begin{enumerate}

\item The spectral sequence $\bar \EE_r$ is isomorphic to $\CEE_r$ after $E_1$-page (not including $E_1$). 

\item The  inclusion  $NC^*(\KK_d^{\leq n})\to NC^*(\KK_d^\bullet)$ induces a map of spectral sequence $\bar \EE^{-p,q}_r\to \EE^{-p,q}_r$ which is bijective for $p\leq n-1, r=2$  and surjective for $p=n, r=2$.

\end{enumerate}
\end{lem}
\begin{proof}
We shall prove part 1. Let $P'_n$ be the poset defined in the proof of Lemma \ref{LPK}. As we did there, we identify $P'_n$ with $\PP_n$. Let $\Delta=\cup_{l\geq 0}\Delta_l$ be the  category of simplexes. For two categories $C, D$, let $Fun(C, D)$ be the category of functors $C\to D$ and natural transformations between them.  Let $\CECHF: Fun({P'_n}^{op},\CH_\kk)\to Fun(\Delta_n^{op}, \CH_\kk)$ be the left Kan extension along $\mathcal{F}:{P'_n}^{op}\to \Delta_n^{op}$ i.e. the left adjoint of the pullback by $\mathcal{F}$.  Concretely speaking,  it associates to a functor $Y:{P'_n}^{op}\to \CH_\kk$ a functor $\CECHF(Y):\Delta_n^{op}\to \CH_\kk$ which sends $[k]$ to $\bigoplus_{f}Y(f([k]))$ where $f:[k]\to \{1,\dots, n+1\}$ runs through the weakly order preserving maps.  Clearly, the normalization  $N\CECHF(Y)$ consists of the summands labeled by monomorphisms $f$. We have the following maps of double complexes
\begin{equation}
NC^*(\KK^{\leq n}_d)\stackrel{\varphi_1}{\longleftarrow} N\CECHF(C^*(\mathcal{F}^*\KK^{\leq n}_d))\simeq N(\CECHF(\bar C_*(\TT)[-dn])\leftarrow \TTH_{\bullet * \star}[-dn]. \label{EQzigzag}
\end{equation}
Here, $\TTH_{\bullet * \star}$ is regarded as the double complex with differentials $(\delta, d+(-1)^*\partial)$ and the right map is defined by sending $\bar C_*(\TT_{\emptyset_P})$ to the summand $\bar C_*(\TT(P))$ labeled by the unique monomorphism $f:[\# P-1]\cong P\subset  \{1,\dots, n+1\}$, using the collapsing map $\TT_{\emptyset_P}\to \TT(P)$,  and $\bar C_*(\TT_G)$ to $0$ for each  $G\not=\emptyset_P$. The  map $\varphi_1$
is defined by the forgetting labels $f$. The symbol $\simeq$ in the middle denotes the zigzag  induced by Lemma \ref{LPK} and Proposition \ref{Phom_dual}. Clearly, the right map and middle zigzag induce isomorphisms of spectral sequences. We shall prove the map $\varphi_1$ induces a quasi-isomorphism on $(E_1,d_1)$.  It is easy to see that the total complex of the normalization of a functor $X:\Delta_n^{op}\to \CH_\kk$ is naturally isomorphic to the realization of $X$ restricted to $\leq n$, defined using the usual cellular chain complexes of the standard simplexes, which is in turn, naturally quasi-isomorphic to the homotopy colimit of $X$ (see subsection 18.6 of \cite{hirschhorn}).
By Theorem 6.7 of \cite{sinha}, the functor $\mathcal{F}$  induces a quasi-isomorphism between the homotopy limit of a functor from $\Delta_n$ and that of its pullback by $\mathcal{F}$. By considering the opposite model category $\CH_\kk^{op}$, we see that the obvious map $ \varphi_2:\underset{{P'_n}^{op}}{\hocolim}\,\mathcal{F}^*X\to \underset{\Delta_n^{op}}{\hocolim}\, X$ is a quasi-isomorphism for a contravariant functor $X$ from $\Delta_n$, where $\hocolim$ denotes the homotopy colimit. 
The map $\varphi_2$ is decomposed as follows.
\[
\underset{{P_n'}^{op}}{\hocolim}\,\mathcal{F}^*X\stackrel{\varphi_3}{\longrightarrow} \underset{{P_n'}^{op}}{\hocolim}\,\mathcal{F}^*\CECHF\mathcal{F}^*X\stackrel{\varphi_4}{\longrightarrow} \underset{{\Delta_n}^{op}}{\hocolim}\,\CECHF\mathcal{F}^*X\stackrel{\varphi_5}{\longrightarrow} \underset{{\Delta_n}^{op}}{\hocolim}\,X,
\]
where $\varphi_3$ and $\varphi_5$ are induced by the units of the adjoint pair $(\CECHF, \mathcal{F}^*)$ and $\varphi_4$ is the map $\varphi_2$ applied to $\CECHF\mathcal{F}^*X$ instead of $X$. The composition $\varphi_4\circ\varphi_3$ is a quasi-isomorphism  by Lemma 2.2 of \cite{moriya1}. By two out of three, $\varphi_5$ is a quasi-isomorphism. Applying $\varphi_5$ to $X=H^*(\KK_d^{\leq n})$, the functor taking cohomology regarded as a complex with zero differential termwisely, with the above quasi-isomorphism between the homotopy colimit and normalization,  we see that the map $\varphi_1$ induces a quasi-isomorphism between the $E_1$-pages, and we have proved part 1. Part 2 is obvious.
\end{proof}
\indent The following lemma is well-known.
\begin{lem}[See e.g. \cite{BT}]\label{Lbasic_spec}
Let $(K,D_1,D_2)$ be a first or  second quadrant double complex and $E^{p,q}_r$ the associated spectral sequence. Let $\omega_0$ be a cycle in $(E_0,d_0=D_1)$.  There exists a sequence $\omega_1,\dots, \omega_r$ of  elements of $K$ such that $D_2\omega_i=D_1\omega_{i+1}$ for $0\leq i\leq r-1$ if and only if $\omega_0$ persists to the $E_r$-page. For any such  sequence, we have $d_{r+1}([\omega_0])=\pm [\omega_r]$. If $\omega_0$ persists to infinity, the sum $\omega_0\pm \omega_1\pm\cdots$ with signs depending on the degree of $\omega_0$ is a cycle in the total complex of $K$, which projects to the class of $\omega_0$ in the associated graded. \hfill \qedsymbol
\end{lem}
Let $E_r(p)$ be the spectral sequence associated to the double complex $(\TT_{p * \star},d,\partial)$ (for a fixed $p$). $E_r(p)$ converges  to 
\[
\CEE^{-p,*}_1=\bigoplus_{P\in\PP_n, \ \#P=p+2}\bar H_*(\TT(P)),
\]
and its $E_1$ page is the direct sum of the groups $\bar H_*(\TT_G)$ with $G\in \GG(P)$ for $\#P=p+2$. 
We easily observe that the space $\TT_G$ is homotopy equivalent to the sphere of dimension ${d(n+2-\#\pi_0(G))}$, see the proof of Lemma \ref{Lfundcyc} (2). Also, the homology $\bar H_*(\TT(P))$ is isomorphic to the cohomology $H^*(Conf_p(\RR^d))$ of the configuration space of ordered $p$ points in $\RR^d$, up to some degree shift.
With these observations, we can show that $E_r(p)$ degenerates at $E_2$-page as in the proof of a description of Bendersky-Gitler spectral sequence given in \cite{FT}.  (The 3-term relation is given by the $d_1$- (or $\CECHC$ech) differential  $(i,j)(j,k)(k,i)\mapsto \pm(j,k)(k,i)\pm (i,j)(k,i)\pm (i,j)(j,k)$.) So for a forest $G\in \GG(P)$,  the elements of $\bar H_*(\TT_G)\subset E_1(p)$ persist to infinity. We call the submodule of $\CEE_1$ consisting of the elements obtained  from  elements of $\bar H_*(\TT_G)$ by the procedure of Lemma \ref{Lbasic_spec} the {\em submodule represented by  $G$}. (The induced elements are unique since the non-zero elements of $E_2(p)$ with a single total degree  are concentrated to a singule horizontal degree (i.e. number of edges).) By the standard presentation of $H^*(Conf_p(\RR^d))$ (see the paragraph after Lemma \ref{Lspec_seq2} below), 
the module $\bar H_*(\TT(P))\subset \CEE_1$ is spanned by  the submodules represented by   forests $G\in \GG(P)$. We have a similar description for $tr_m\CEE_1$, but for the graphs with exactly $m$ edges,  not only forests, any graphs give (possibly non-trivial) elements of the $E_1$-page. We do not give a complete description of $tr_m\CEE_1$.  The following lemma, which is obvious by a spectral sequence argument, is sufficient for our purpose. 
\begin{lem}\label{Lspec_seq2}
Under the above terminology, the map $\CEE_1^{pq}\to tr_m\CEE_1^{pq}$  in Definition \ref{Dspecseq} (1) induces an isomorphism between the sum of  submodules represented by  forests $G$ with $m+1$  edges or more and a monomorphism between the sum of  submodules represented by graphs with exactly $m$ edges.\hfill \qedsymbol
\end{lem}
We have obtained the zigzag of maps of spectral sequences
\[
\EE_r\leftarrow \bar \EE_r\leftarrow \CEE_r\rightarrow tr_m\CEE_r.
\]
By Lemmas \ref{Lspecseq} and \ref{Lspec_seq2}, to compute a differential of an element in $\EE_r^{-n,q}$ is equivalent to computation of the  differential of an element in $\bar\EE_r^{-n,q}$ which is sent to the element. This is in turn, equivalent to computation of the corresponding differential in $tr_{m}\CEE_r$ for  $r\geq 2$ and some $m$. For $d=2$, $d_r$ decreases the number of edges by $r-1$ so we can take $m$ as the number of edges of the  element minus $r-1$. \\ 
\indent We shall recall the well-known description of Sinha's sequence. The group $\EE^{-p,*}_1$ is isomorphic to the cohomology $H^*(Conf_p(\RR^d))$, which is generated by elements $g_{ij}$ of degree $d-1$ with $1\leq i, j\leq p$, $i\not=j$ as an algebra modulo the relations $(g_{ij})^2=0$, $g_{ji}=(-1)^dg_{ij}$, and the 3-term relation $g_{ij}g_{jk}+g_{jk}g_{ki}+g_{ki}g_{ij}=0$. The $d_1$-differential $:\EE^{-p,q}_1\to \EE^{-p+1,q}_1$ is given by $\sum_{0\leq i\leq p}(-1)^i\delta_i$ where $\delta_i$  changes the labels of the generators by the order preserving surjection $[p+1]\to [p],\ i, i+1\mapsto i$. ($g_{0i}$, $g_{ip}$ and  $g_{ii}$ are regarded as zero in the target.) For example, if we consider $g_{13}g_{24}\in \EE^{-4,2d-2}_1$, we have
\[
\begin{split}
d_1(g_{13}g_{24}) & =g_{02}g_{13}-g_{12}g_{13}+g_{12}g_{23}-g_{13}g_{23}+g_{13}g_{24} \\
                      & =0+g_{31}g_{12}+g_{12}g_{23}+g_{23}g_{31}+0=0.
\end{split}
\]  
The $d_1$-differential of $\bar \EE^{-p,q}_1$ is given by the same formula if $p\leq n$. \\
\indent We now turn to a description of $\CEE_1$. 
We re-label the generators of $H^*(Conf_p(\RR^d))$ with  elements of $P^\circ$ instead of $1,\dots, p$. An isomorphism between $H^*(Conf_p(\RR^d))$ and $\bar H_*(\TT(P))$ are given by sending a monomial $g_{\alpha_1,\beta_1}\cdots g_{\alpha_r,\beta_r}$ to a generator of $\bar H_{n(d-r)}(\TT_G)$ with $G=(\alpha_1,\beta_1)\cdots (\alpha_r,\beta_r)$. Under this identification,  the $d_1$-differential of $\CEE$ is similar to $\EE$ but the change of labels is induced by the natural surjection $\delta_i:P\to \delta_iP$. For example, if $p=n=4$, we have
\[
d_1(g_{13}g_{24})  =-g_{\{12\}3}g_{\{12\}4}+g_{1\{23\}}g_{\{23\}4}-g_{1\{34\}}g_{2\{34\}}\ \in\ \bigoplus_{1\leq i\leq 3}\bar H_*(\TT(\delta_i[4])).
\]   

\section{Condensed maps}\label{Scondensed}
In this section, we define a class of maps used to define chains in the  complex $\TTH$ and prove their properties. 
\begin{defi}\label{Dcondensed}

\begin{enumerate}
\item Let $X$ be an unpointed topological space, $P\in \PP_{n+1}$ a partition,  $G\in \GG(P)$  a graph, and  $f=(f_1,\dots, f_n):X\to \RR^{dn}$ a map. For a connected component $S$ of $G$ contained in $P^\circ$,   a vertex $\alpha\in S$ is called {\em a base of $S$ (for $f$, or $(G,f)$)}, which satisfies  the following conditions.
\begin{enumerate}
\item For each $\beta\in S$, there exists an element $y_i(x)$ of the convex hull of $\{f_j(x)\mid j\in \alpha\}$ satisfying $f_i(x)=_1y_i(x)$   for any $i\in \beta$ and $x\in X$, and there is at least one $i\in \beta$ such that for any $x$, $f_i(x)$ belongs to the convex hull.
\item there is no edge in $G$ which both of the vertices incident with  are in $S$ and strictly smaller than $\alpha$.
\end{enumerate}
(A discrete vertex in $P^\circ$ is regarded as a base.)
We say $f$ is {\em $G$-condensed} if the following conditions hold.
\begin{enumerate}
\item $f$ is proper. 
\item Each connected component $S\subset P^\circ$ of $G$ has at least one base for $f$.
\end{enumerate}
\item Let $G\in \GG([n+1])$ be a graph with exactly $2$ connected components contained in $\{1,\dots, n\}$. Define a map  $f_G=(f_1,\dots f_n):\RR^{2d}\to \RR^{dn}$ by
\[
f_i(x,y)=\left\{\begin{array}{cc}
x & (i\sim_G1)\\
y & (\text{otherwise}).
\end{array}\right.
\]
Clearly, $f_G$ is $G$-condensed, and a vertex is a base if it satisfies the condition (b) on a base. At least, the minimum and the second to minimum vertices of each connected component $\subset P^\circ$ are  bases.
\item For a graph $G\in \GG(P)$, we define subsets $U_G$ and $U^1_G$ of $\RR^{dn}$ by
\[
\begin{split}
U_G=&(\RR^{dn}-\nu_{P})\cup \pi_P^{-1}(\cap_{(\alpha,\beta)\in E(G)}D_{\alpha\beta}), \\
U_G^1=&(\RR^{dn}-p_1^{-1}p_1(\nu_P))\cup p_1^{-1} p_1\pi_P^{-1}(\cap_{(\alpha,\beta)\in E(G)}D_{\alpha\beta}).
\end{split}
\]
Here, $p_1:\RR^{dn}=(\RR^d)^n\to \RR^n$ is  the product of the projections to the first coordinate $\RR^d\to \RR$.
\end{enumerate}
\end{defi}

\begin{lem}\label{Lcondensed}
Let $f:X\to \RR^{dn}$ be a $G$-condensed map for a graph $G\in \GG(P)$. We have $f(X) \subset U_G\cap U^1_G$. In particular, $f$ induces a map $f:X^*\to \TT_G$.
\end{lem}
\begin{proof}
Let $S$ be a connected component of $G$,  $\alpha$ a base of $S$, and  $\beta$  a vertex in $S$.  For $x\in X$, suppose  $f(x)\in \nu_P$ and write $\pi_P(f(x))=(y_\gamma)$.   For a vertex $\gamma$, let $j_\gamma$ denote the minimum number in $\gamma$. By definitions of $e_P$ and $\nu_P$, for $i\in \beta$, we have
\[
|f_i(x)-y_\beta|\leq \rho (c_\beta-c_{j_\beta})/2+\eps_P.
\]
Since $f_i(x)$ belongs to the convex hull of $\{f_r(x)\}_{r\in \alpha}$ for some $i\in \beta$, we have 
$
|f_i(x)-y_\alpha|\leq \rho (c_\alpha-c_{j_\alpha})/2+\eps_P.
$
So, we have $|y_\alpha-y_\beta|<\rho (c_\alpha+c_\beta-c_{j_\alpha}-c_{j_\beta})/2+2\eps_P$. Let $\bar \beta$ be a vertex of $G$ such that there is an edge incident with $\beta$ and $\bar\beta$. By definition of the base, $\alpha$ is not the maximum of $\alpha, \beta$ and $\bar \beta$. By the above argument, 
we see 
\[
\begin{split}
|y_{\bar \beta}-y_\beta|&<\frac{\rho}{2}(c_\beta+c_{\bar \beta}+2c_\alpha-c_{j_\beta}-c_{j_{\bar \beta}}-2c_{j_\alpha})+4\eps_P\\
&<d_{\bar\beta\beta}(P).
\end{split}
\] The second inequality follows from  the condition on $c_r$ imposed in Definition \ref{DPK} (1). We have shown $f(X)\subset U_G$. Since $f$ is proper, $f$ induces the map from the one-point compactification. The proof of the inclusion to $U^1_G$ is completely similar. 
\end{proof}

\begin{defi}\label{Dcondensed2}
Let $P\in \PP_n$ and $G\in\GG(P)$.
\begin{enumerate}
\item An edge $e$ of  $G$ is called a {\em bridge} if $\#\pi_0(\partial_eG)>\#\pi_0(G)$.
\item Let $f:X\to \RR^{dn}$ be a $G$-condensed map. For a bridge $e$ of $G$, there is a unique connected component of $G$ which splits into two components by removing $e$.  We say   $e$  is {\em admissible} for $f$ if each of the two  new components has a base  for $(\partial_eG, f)$. (This condition is satisfied for example,  if for each  of the two components $S$, $S$  contains a base of a component of $G$ for $(G,f)$, or is discrete, or for any two numbers  $i,j \in S$,  $f_i(x)=f_j(x)$ holds.)  Similarly, for  a pair $(e,e')$ of distinct bridges of $G$, there are exactly two components of $G$ each of which splits into two components, or there is a unique component of $G$ which splits into three components by removing $e$ and $e'$.   
We say  $(e,e')$  is {\em admissible} for $f$ if each of
the new components has a base for $(\partial_{ee'}G, f)$. 
(This condition is satisfied for example, if for each of the new components $S$ of $\partial_{ee'}G$, $S$ contains a base for $(G,f)$ or is discrete, or  for any two numbers  $i,j \in S$,  $f_i(x)=f_j(x)$.) 
\item Let  $f:X\to \RR^{dn}$ be a $G$-condensed map. Let $k$ be a  number  with $1\leq k\leq \# P-3$, and $\alpha_k$ and  $\alpha_{k+1}$ the $k+1$-th and $k+2$-th pieces of $P$, respectively. We say $k$  is {\em admissible} for $(G,f)$ if $\delta_k(G)\in \GG(\delta_kP)$ and either of the following conditions holds.
\begin{enumerate}
\item Both of $\alpha_k$ and $\alpha_{k+1}$ are bases for $(G,f)$.
\item  For exactly one  of $l=k,k+1$,  $\alpha_{l}$ is  discrete in $G$, and if  $m$ denotes the other number, for some base $\beta\sim_G\alpha_{m}$,  there is a point $y_i(x)$ in the convex hull of $\{f_j(x)\mid j\in \beta\}$ satisfying $f_i(x)=_1y_i(x)$ for each $i\in\alpha_{l}$ and $x\in X$. 
\end{enumerate}
\item Let $f:X\to \RR^{dn}$ be a  map, and $e=(\alpha,\beta)$  a bridge of  $G$. For $1\leq i\leq n$ and $s\in [0,\infty)$, we set 
\[
A_{e}^i(s)=\left\{\begin{array}{cl}
s & \text{if $i\sim _{\partial_eG} \alpha$} \\
-s & \text{if $i\sim _{\partial_eG} \beta$} \\
0 & \text{otherwise}
\end{array}
\right.
\]
and define $A_{e}:[0,\infty)\to \RR^{dn}$
\[
A_{e}(s)=(A_{e}^i(s)v)_{1\leq i\leq n},
\]
where $v=(0,1,0,\dots, 0)$ as in subsection \ref{SSNT}. The {\em contraction $F:X\times [0,\infty)\to \RR^{dn}$ of $f$  in removing $e$ from $G$} (in short, {\em $e$-contraction of $f$ for $G$}) is   defined by
\[
F(x,s)=f(x)+A_e(s).
\]
Let $e'$ be another bridge of $G$. The {\em $(e,e')$-contraction  $F':X\times [0,\infty)^2\to \RR^{dn}$ of $f$ for $G$}  is defined by
\[
F'(x,s_1,s_2)=f(x)+A_e(s_1)+A_{e'}(s_2).
\]

\end{enumerate}
\end{defi}

\begin{lem}\label{Lcontraction}
Let $Q\in \PP_n$, $G\in \GG(Q)$, and  $f:X\to \RR^{dn}$  a $G$-condensed map.
\begin{enumerate}
\item  If $e\in E(G)$ is an admissible bridge for $f$, the $e$-contraction $F$ of $f$ for $G$ is $\partial_eG$-condensed. 
\item  If $(e, e')$ is an admissible pair of   bridges of $G$ for $f$, the $(e,e')$-contraction $F'$ of $f$ for $G$ is $\partial_{ee'}G$-condensed.
\item  If $k$ is admissible for $(G,f)$, $f$ is $\delta_kG$-condensed.
\end{enumerate}
\end{lem}
\begin{proof}
For part 1, we only have to prove that $F$ is proper. The other conditions are clear by definition. Let $\beta$ and $\gamma$ be the vertices incident with $e$, $\alpha$  a base of the connected component $S$ of $G$ including $e$.  Say $\alpha\not\sim_{\partial_eG}\beta$ and $\gamma<\beta$. Suppose $F(x)\in D^{dn}_R=(D^d_R)^n$, the product of disks of radius $R$, centered at $0$. For $i\in \beta$ and $j\in \alpha$, we have
\[
2R\geq |F_i(x)-F_j(x)|=|f_i(x)- sv-(f_j(x)+ sv)|\geq ||f_i(x)-f_j(x)|-2s|.
\]  
This implies $s\leq R+|f_i(x)-f_j(x)|/2$. Since the convex hulls of  $\{f_j(x)+sv\}_{j\in \alpha}$ and $\{f_j(x)\}_{j\in \alpha}$ are congruent and the diameter of former one is  $2R$ or less, we have $|f_i(x)-f_j(x)|\leq 2R$. By these inequalities,  we  have $s\leq 2R$, which implies $|f_k(x)|\leq 3R$ for any number $k\in S$. Thus, we have 
$F^{-1}(D^{dn}_R)\subset f^{-1}(D^{dn}_{3R})$ and conclude that $F$ is proper.  Part 2 is similar and  part 3 is obvious. 
\end{proof}
\begin{exa}\label{Ef_G}
Let $G_1$ be the graph in Example \ref{Egraph}.
The map $f=f_{G_1}$ in Definition \ref{Dcondensed} is given by $(x,y)\mapsto (x,y,y,x)$. The edges $e=(1,4)$ and $e'=(2,3)$ are adimissible bridges and also admissble as a pair. The $e$-contraction $F$ of $f$ (for $G_1$) is given by $(x,y,s)\mapsto (x+sv,y,y, x-sv)$ and the $(e,e')$-contraction $F'$ of $f$ for $G_1$ is given by $(x,y,s_1,s_2)\mapsto (x+s_1v, y+s_2v, y-s_2v, x-s_1v)$. By Lemma \ref{Lcontraction}, $F$  is  $\partial_eG_1$-condensed and $F'$ is $\partial_{ee'}G_1$-condensed.  
Since $3\in [5]$ is an admissible number for $(G_1,f)$, $(\partial_eG_1,F)$ and $(\partial_{ee'}G_1,F')$, the maps $f$, $F$ and $F'$ are also $\delta_3G_1$-, $\delta_3\partial_eG_1$- and  $\delta_3\partial_{ee'}G_1$-condensed respectively.  Similar claim holds for the $e'$-contraction. The number $1$ is admissible for the pairs $(G_1,f), (\partial_eG_1,F), (\partial_{ee'}G_1, F'), (\delta_3\partial_eG_1,F),$ and $(\delta_3\partial_{ee'}G_1,F')$,  and not admissible for $(\delta_3G_1,f)$, for example ($ (\delta_3\partial_eG_1,F)$ is a case where the condition (b) of Definition \ref{Dcondensed2} (3) holds). Also, the number 2 is not admissible for $(G_1,f)$ and $(\partial_eG_1,F)$.  The vertex $\{1,2\}$ is a base of a connected component of $(\delta_1G_1,f)$. The other vertices  are not bases.

\end{exa}
For the maps in later sections, necessary admissibility of bridges or numbers is also seen obviously and we will not mention it in each case.  
\begin{defi}\label{Dstraight}
 For two  maps $f,g:X\to \RR^{dn}$, the 
{\em straight homotopy $h:X\times I\to \RR^{dn}$ from $f$ to $g$} is defined by
$h(x,t)=(1-t)f(x)+tg(x)$.
\end{defi}
\begin{lem}\label{Lstraight_homotopy} Let $G\in\GG(P)$ be a graph.
\begin{enumerate} 
\item Let $f,g:X\to \RR^{dn}$ be  two $G$-condensed maps such that for each  component $S\subset P^\circ$, there is a common base $\alpha$ of $S$ for $f$ and $g$   satisfying the following two conditions: For each $i\in \alpha$  $f_i=g_i$, and for each $j\in S$ and $x\in X$, $f_j(x)$ and $g_j(x)$ belong to the convex hull of $\{f_i(x)\}_{i\in \alpha}$.
Then, the straight homotopy $h$ from $f$ to $g$ is $G$-condensed, so it induces a map
$h:X^*\wedge (I_+)\to \RR^{dn}$. 
\item Let $f:X\to \RR^{dn}$ be a $G$-condensed map, $k$ an admissible number for $(G,f)$ satisfying the condition (a) in Definition \ref{Dcondensed2} (3). Let $e$ be an admissible bridge of $G$ for $f$ which is also a bridge of $\delta_kG$ through the standard bijection. Suppose   $e$ belongs to the connected component of $\delta_kG$ which includes the $k+1$-th vertex of $\delta_kG$ (i.e. the unified vertex). Let $F$ (resp. $F'$) be the $e$-contraction of  $f$ for $G$ (resp. $\delta_kG$). The straight homotopy $H$ from $F$ to $F'$ is $\delta_k\partial_eG$-condensed. 
\item Let $f:X\to \RR^{dn}$ be a $G$-condensed map, $k$ an admissible number for $(G,f)$ satisfying the condition (a) in Definition \ref{Dcondensed2} (3). Let $(e,e') $ be an admissible pair of bridges of $G$ for $f$ such that $e$ and $e'$ are also bridges of $\delta_kG$. Suppose $e$ and $e'$ belong to the  connected component of $\delta_kG$ which includes $k+1$-th vertex of $\delta_kG$.  Let $F$ (resp. $F'$) be the $(e,e')$-contraction of  $f$ for $G$ (resp. $\delta_kG$). The straight homotopy from $F$ to $F'$ is $\delta_k\partial_{ee'}G$-condensed.

\end{enumerate}
\end{lem}
\begin{proof}
Part 1 is clear. We shall prove part 2. Let $\alpha_k$, $\alpha_{k+1}$ be the $k+1$-th and $k+2$-th pieces of $P$, respectively. Put $e=(\beta, \gamma)$ and $H=(H_1,\dots, H_n)$. We consider the case of $\beta\sim_{\partial_eG}\alpha_k\not\sim_G\alpha_{k+1}$.  By the assumption, 
\[
H_i(x,s,t)=
\left\{
\begin{array}{cc}
f_i(x)+sv & \text{ if }i\sim_{\partial_eG}\beta,\\
f_i(x)-sv & \text{ if }i\sim_{\partial_eG}\gamma,\\
f_i(x)+stv & \text{ if }i\sim_G\alpha_{k+1},\\
f_i(x) & \text{ otherwise. }
\end{array}
\right.
\]
By this formula, we easily see $H$ is $\delta_k\partial_eG$-condensed as in Lemma \ref{Lcontraction}. The other cases and part 3 are similar.
\end{proof}
\begin{exa}\label{Estraight}
Let $G_1, G_2$ be the graphs in Example \ref{Egraph}. Put $f=f_{G_1}$, $f'=f_{G_2}$, and $e_1=(1,4)\in E(G)$. Let $T:(\RR^2)^2\to (\RR^2)^2$ be the transposition $(x,y)\mapsto (y,x)$. We see that $\delta_3G_1=\delta_3G_2$ and the straight homotopy $\psi$ from $f$ to $f'\circ T$ is given by $(x,y,t)\mapsto ((1-t)x+ty,tx+(1-t)y,y,x)$ and $\psi$ is  $\delta_3G_1$-condensed by Lemma \ref{Lstraight_homotopy}. The unique base  is $\{1,2\}$. The straight homotopy $\lambda_1$ from the $e_1$-contraction of $f$ for $G$ to the $e_1$-contraction of $f$ for $\delta_3G$ is given by $(x,y,s,t)\mapsto (x+sv,y-stv, y-stv,x-sv)$ and $\delta_3\partial_1G$-condensed (see Figure \ref{Fpsi}). Note that the straight homotopy from $f$ to $f'$ is not $\delta_3G_1$-condensed.
\end{exa}
\begin{figure}
\begin{center}
{\unitlength 0.1in%
\begin{picture}(51.3000,4.9800)(-0.5300,-5.0800)%
%
\special{pn 8}%
\special{pa 2990 237}%
\special{pa 3887 237}%
\special{fp}%
%
\special{pn 8}%
\special{pn 8}%
\special{pa 3078 237}%
\special{pa 3084 231}%
\special{fp}%
\special{pa 3111 206}%
\special{pa 3117 200}%
\special{fp}%
\special{pa 3144 175}%
\special{pa 3150 169}%
\special{fp}%
\special{pa 3178 145}%
\special{pa 3184 140}%
\special{fp}%
\special{pa 3213 115}%
\special{pa 3219 110}%
\special{fp}%
\special{pa 3248 88}%
\special{pa 3255 83}%
\special{fp}%
\special{pa 3286 62}%
\special{pa 3293 58}%
\special{fp}%
\special{pa 3326 40}%
\special{pa 3333 37}%
\special{fp}%
\special{pa 3367 23}%
\special{pa 3375 20}%
\special{fp}%
\special{pa 3411 12}%
\special{pa 3419 11}%
\special{fp}%
\special{pa 3457 11}%
\special{pa 3465 12}%
\special{fp}%
\special{pa 3501 20}%
\special{pa 3509 22}%
\special{fp}%
\special{pa 3543 37}%
\special{pa 3550 40}%
\special{fp}%
\special{pa 3583 59}%
\special{pa 3590 63}%
\special{fp}%
\special{pa 3621 83}%
\special{pa 3627 88}%
\special{fp}%
\special{pa 3657 111}%
\special{pa 3663 116}%
\special{fp}%
\special{pa 3692 140}%
\special{pa 3698 145}%
\special{fp}%
\special{pa 3726 169}%
\special{pa 3732 175}%
\special{fp}%
\special{pa 3759 200}%
\special{pa 3765 206}%
\special{fp}%
\special{pa 3792 231}%
\special{pa 3798 237}%
\special{fp}%
%
\special{pn 8}%
\special{pa 4181 237}%
\special{pa 5077 237}%
\special{fp}%
%
\special{pn 8}%
\special{pn 8}%
\special{pa 4269 237}%
\special{pa 4275 231}%
\special{fp}%
\special{pa 4302 206}%
\special{pa 4308 200}%
\special{fp}%
\special{pa 4335 175}%
\special{pa 4341 169}%
\special{fp}%
\special{pa 4369 145}%
\special{pa 4375 140}%
\special{fp}%
\special{pa 4404 115}%
\special{pa 4410 110}%
\special{fp}%
\special{pa 4439 88}%
\special{pa 4446 83}%
\special{fp}%
\special{pa 4477 62}%
\special{pa 4484 58}%
\special{fp}%
\special{pa 4517 40}%
\special{pa 4524 37}%
\special{fp}%
\special{pa 4558 23}%
\special{pa 4566 20}%
\special{fp}%
\special{pa 4602 12}%
\special{pa 4610 11}%
\special{fp}%
\special{pa 4648 11}%
\special{pa 4656 12}%
\special{fp}%
\special{pa 4692 20}%
\special{pa 4700 22}%
\special{fp}%
\special{pa 4734 37}%
\special{pa 4741 40}%
\special{fp}%
\special{pa 4774 59}%
\special{pa 4781 63}%
\special{fp}%
\special{pa 4812 83}%
\special{pa 4818 88}%
\special{fp}%
\special{pa 4848 111}%
\special{pa 4854 116}%
\special{fp}%
\special{pa 4883 140}%
\special{pa 4889 145}%
\special{fp}%
\special{pa 4917 169}%
\special{pa 4923 175}%
\special{fp}%
\special{pa 4950 200}%
\special{pa 4956 206}%
\special{fp}%
\special{pa 4983 231}%
\special{pa 4989 237}%
\special{fp}%
\put(30.7700,-3.7700){\makebox(0,0){$+$}}%
\put(37.9500,-3.8100){\makebox(0,0){$-$}}%
\put(42.6400,-3.8000){\makebox(0,0){$+$}}%
\put(50.0200,-3.8200){\makebox(0,0){$-$}}%
\put(48.4800,-3.8200){\makebox(0,0){$-$}}%
\put(44.3200,-3.7700){\makebox(0,0){$-$}}%
\put(40.2600,-1.7200){\makebox(0,0){$\simeq$}}%
\put(33.2500,-5.7300){\makebox(0,0){$e_1$-cont. of $f$ for $G_1$}}%
\put(48.7500,-5.6300){\makebox(0,0){$e_1$-cont. of $f$ for $\delta_3G_1$}}%
%
\special{pn 8}%
\special{pa 3238 237}%
\special{pa 3265 212}%
\special{pa 3292 188}%
\special{pa 3319 166}%
\special{pa 3346 147}%
\special{pa 3374 132}%
\special{pa 3401 121}%
\special{pa 3428 115}%
\special{pa 3456 116}%
\special{pa 3483 123}%
\special{pa 3511 135}%
\special{pa 3538 152}%
\special{pa 3566 171}%
\special{pa 3622 217}%
\special{pa 3644 237}%
\special{fp}%
%
\special{pn 8}%
\special{pa 4428 237}%
\special{pa 4455 212}%
\special{pa 4482 188}%
\special{pa 4509 166}%
\special{pa 4536 147}%
\special{pa 4564 132}%
\special{pa 4591 121}%
\special{pa 4618 115}%
\special{pa 4646 116}%
\special{pa 4673 123}%
\special{pa 4701 135}%
\special{pa 4728 152}%
\special{pa 4756 171}%
\special{pa 4812 217}%
\special{pa 4834 237}%
\special{fp}%
%
\special{pn 8}%
\special{pa 3606 188}%
\special{pa 3833 188}%
\special{pa 3833 290}%
\special{pa 3606 290}%
\special{pa 3606 188}%
\special{pa 3833 188}%
\special{fp}%
%
\special{pn 8}%
\special{pa 4792 188}%
\special{pa 5020 188}%
\special{pa 5020 290}%
\special{pa 4792 290}%
\special{pa 4792 188}%
\special{pa 5020 188}%
\special{fp}%
%
\special{pn 8}%
\special{pa 261 231}%
\special{pa 1093 231}%
\special{fp}%
%
\special{pn 8}%
\special{pa 323 238}%
\special{pa 377 186}%
\special{pa 431 136}%
\special{pa 458 113}%
\special{pa 485 92}%
\special{pa 511 72}%
\special{pa 538 55}%
\special{pa 565 40}%
\special{pa 592 27}%
\special{pa 619 18}%
\special{pa 646 12}%
\special{pa 673 10}%
\special{pa 700 12}%
\special{pa 726 17}%
\special{pa 753 26}%
\special{pa 780 38}%
\special{pa 807 53}%
\special{pa 834 71}%
\special{pa 861 90}%
\special{pa 887 112}%
\special{pa 914 135}%
\special{pa 941 159}%
\special{pa 968 184}%
\special{pa 994 210}%
\special{pa 1016 231}%
\special{fp}%
%
\special{pn 8}%
\special{pa 538 238}%
\special{pa 564 210}%
\special{pa 590 183}%
\special{pa 615 161}%
\special{pa 641 144}%
\special{pa 667 136}%
\special{pa 692 137}%
\special{pa 717 148}%
\special{pa 742 166}%
\special{pa 767 190}%
\special{pa 792 218}%
\special{pa 815 245}%
\special{fp}%
%
\special{pn 8}%
\special{pa 1563 227}%
\special{pa 2402 227}%
\special{fp}%
%
\special{pn 8}%
\special{pa 1654 227}%
\special{pa 1680 199}%
\special{pa 1705 172}%
\special{pa 1731 147}%
\special{pa 1757 123}%
\special{pa 1782 102}%
\special{pa 1808 85}%
\special{pa 1834 71}%
\special{pa 1860 63}%
\special{pa 1886 60}%
\special{pa 1913 63}%
\special{pa 1939 71}%
\special{pa 1966 84}%
\special{pa 1992 100}%
\special{pa 2019 120}%
\special{pa 2045 143}%
\special{pa 2072 168}%
\special{pa 2099 194}%
\special{pa 2126 222}%
\special{pa 2131 227}%
\special{fp}%
%
\special{pn 8}%
\special{pa 1854 227}%
\special{pa 1880 199}%
\special{pa 1905 172}%
\special{pa 1931 147}%
\special{pa 1957 123}%
\special{pa 1983 103}%
\special{pa 2009 85}%
\special{pa 2035 72}%
\special{pa 2061 63}%
\special{pa 2087 60}%
\special{pa 2113 63}%
\special{pa 2140 70}%
\special{pa 2166 83}%
\special{pa 2193 100}%
\special{pa 2219 119}%
\special{pa 2246 142}%
\special{pa 2273 167}%
\special{pa 2300 193}%
\special{pa 2326 220}%
\special{pa 2333 227}%
\special{fp}%
\put(3.5300,-4.1500){\makebox(0,0){$\underline{x}$}}%
\put(5.8200,-4.1500){\makebox(0,0){$\underline{y}$}}%
\put(8.0300,-4.1500){\makebox(0,0){$y$}}%
\put(10.1800,-4.1500){\makebox(0,0){$x$}}%
\put(16.7600,-4.1200){\makebox(0,0){$\underline{y}$}}%
\put(19.0600,-4.1200){\makebox(0,0){$\underline{x}$}}%
\put(21.2700,-4.1200){\makebox(0,0){$y$}}%
\put(23.4100,-4.1200){\makebox(0,0){$x$}}%
\put(13.3200,-1.5800){\makebox(0,0){$\simeq$}}%
%
\special{pn 8}%
\special{pa 773 178}%
\special{pa 1043 178}%
\special{pa 1043 286}%
\special{pa 773 286}%
\special{pa 773 178}%
\special{pa 1043 178}%
\special{fp}%
%
\special{pn 8}%
\special{pa 2100 178}%
\special{pa 2369 178}%
\special{pa 2369 286}%
\special{pa 2100 286}%
\special{pa 2100 178}%
\special{pa 2369 178}%
\special{fp}%
\put(1.2200,-1.6400){\makebox(0,0){$\psi:$}}%
\put(28.5000,-1.7100){\makebox(0,0){$\lambda_1:$}}%
\end{picture}}%
\end{center}
\caption{$\psi$ and $\lambda_1$ 
: The underlines indicate the components moved by the straight homotopy,  the rectangles do the unified vertices, the dotted chords do removed edges, and the sign $\pm$ under each vertex does that of $\pm sv$ added to the corresponding component. We use similar notations throughout the paper.}\label{Fpsi}
\end{figure}
\noindent {\bf Notation and terminology : }\ As written in Lemmas  \ref{Lcondensed}, \ref{Lcontraction}, and \ref{Lstraight_homotopy}, we use the same symbol for the induced map between the pointed spaces as the original unpointed map. We also use the terms such as  `contraction' and `straight homotopy' for the induced map.\\

The proof of the following lemma is similar to Lemma \ref{Ldiagonal_incl}.

\begin{lem}\label{Lcondensed_collapse0}
Let $f=(f_1,\dots, f_n):X\to \RR^{dn}$ be a proper map, $x\in X$, and $P\in \PP_n$. Put $f_0(x)=(-1+\rho c_0/2)u$ and $f_{n+1}(x)=(1-\rho c_{n+1}/2)u$. For the induced map $f:X^*\to \TT_{\emptyset_P}$, if $f(x)\not=*$,  the following inequalities hold for any $\alpha\in P$,  numbers $1\leq k\leq n$, and $i, j\in \alpha$ with $i<j$:
\[
\begin{split}
-1+\rho c_{\leq k}-2\eps_P&<_1f_k(x)<_11-\rho c_{\geq k}+2\eps_P, \\
\rho c_{i\to j}-2\eps_P&<_1f_j(x)-f_i(x)<_1\rho c_{i\to j}+2\eps_P.
\end{split}
\]

In particular if $f_{j}(x)\leq_1 f_{i}(x)$ for some $i<j \in \alpha$, $f(x)=*$.\hfill\qedsymbol
\end{lem}
\begin{lem}\label{Lcondensed_collapse1}
Put $Q=[n+1]\in \PP_n$. Let $G\in\GG(Q)$ be a forest having exactly two components contained in $\{1,\dots,n\}$ and $P\in \PP_n$  a partition such that $\delta_{P,Q}G$ is not a forest. Put $f=f_G=(f_1,\dots, f_n)$, see Definition \ref{Dcondensed}. Let $g=(g_1,\dots, g_n):X\to \RR^{dn}$ be a proper map and $x\in X$ a point. If there is a point $x'\in \RR^{2d}$ such that for each $1\leq l\leq n$ the equation $g_l(x)=_1f_l(x')$ holds, we have $g(x)=*$ for the induced map $g:X^*\to \TT_{\emptyset_P}$.
\end{lem}
\begin{proof}
Under the assumption, at least one of the following claims holds.
\begin{itemize}
\item There exist a piece $\alpha\in P$ and numbers $i,j\in \alpha $ such that $i\not=j$ and  $i\sim_Gj$.
\item There exist   $\alpha,\beta\in P$, $i_1,j_1\in \alpha$, $i_2,j_2\in \beta$ such that $\alpha\not=\beta$, $i_1\not=j_1$, $i_1\sim_Gi_2$, and $j_1\sim_Gj_2$.
\end{itemize}
This observation, the assumptions on $g$ and on $c_r$ in Defintion \ref{DPK} (1), and Lemma \ref{Lcondensed_collapse0} easily imply the claim.
\end{proof}
\begin{lem}\label{Lcondensed_collapse}

 Let $P\to Q\in \PP_n$ be a subdivision, $G\in\GG(Q)$ a graph.

\begin{enumerate}
\item 

Let  $f:X\to \RR^{dn}$ be a proper map whose image is contained in $U^1_G$ and $e=(\alpha,\beta)$ an edge of $G$.  Suppose  either of the following two conditions holds.
\begin{enumerate}
 \item $\alpha\cup \beta$ is included in a  piece of $P$.
\item At least one of $\alpha$, $\beta$ is included in either of the minimum or maximum of $P$.
\end{enumerate} Then, the induced map $f :X^*\to \TT_{\emptyset_P}$ is $*$ (the constant map to the basepoint). In particular, if $f$ is the $e$-contraction of some $G$-condensed map, or more generally, $f_l=_1g_l$  ($1\leq l\leq n$)  for some $G$-condensed map $g$, the induced map to $\TT_{\emptyset_P}$ is $*$.
\item Let $g:X\to \RR^{dn}$ be a $G$-condensed map. Let $\alpha<\beta<\gamma$ be three  pieces of  $Q$ , Suppose that $\alpha$ is a base for $g$, $\alpha\sim_G\beta\sim_G\gamma$,  and $\beta\cup\gamma$ is included in a piece of $P$. Let $f:X\to \RR^{dn}$ be a proper map satisfying $f_l=_1g_l$ ($1\leq l\leq n$). Then, the map $X^*\to \TT_{\emptyset_P}$ induced by $f$  is  $*$.
\end{enumerate}
\end{lem}
\begin{proof}
We shall show part 1 for the condition (a).  Let $i\in \alpha$ and $j\in \beta$. Suppose $f(x)\in \nu_P$ for some $x\in X$. By definitions of $e_P$ and $\nu_P$, we have
$f_{j}(x)-f_{i}(x)>_1 \rho c_{i\to j}/2-2\eps_P$. Since $f(x)\in U^1_G\cap \nu_P$, by an argment similar to the proof of Lemma \ref{Ldiagonal_bound}, we have
\[
 f_j(x)-f_i(x)<_1\rho c_{ij}/2-\eps_Q+2\eps_P<_1 \rho c_{ij}/2-2\eps_P.
\]
 These inequalities contradict with each other so $f(x)\not\in \nu_P$. The proofs for (b) is similar. We shall show part 2. Suppose $f(x)\in \nu_P$. Since $f_l=_1g_l$, the diameter of  the convex hull of the first coordinates of $f_l(x)$ with $l\in \alpha$, taken in $\RR$, is smaller than $\rho c_{\alpha}+\eps_P$. By the assumption on $c_r$ imposed in Definition \ref{DPK}, for any elements $i\in \beta$ and $j\in \gamma$, we have
\[
f_j(x)-f_i(x)=_1g_j(x)-g_i(x)<_1\rho c_\alpha+\eps_P<\rho c_{ij}-2\eps_P.
\]
So by Lemma \ref{Lcondensed_collapse0}, we have the claim.  
\end{proof}
\begin{defi}\label{Di-cont}
Let $f=(f_1,\dots, f_n):X\to\RR^{dn}$ be a map.
For $\eps=+$ or $-$, we define a map $f^{i\eps}:X\times [0,\infty)\to \RR^{dn}$ called the {\em $(i,\eps)$-contraction of $f$} as follows.
\[
f^{i\eps}_k(x)=
\left\{
\begin{array}{cl}
f_i(x)+\eps s u &(k=i), \\
f_{i+1}(x)-\eps su &(k=i+1), \\
f_k(x) & (\text{otherwise}).
\end{array}\right.
\]
On the right hand side,  $\eps=\pm$ are regarded as $\pm 1$ respectively.

\end{defi}
While the $(i,\eps)$-contraction is not necessarily $G$-condensed, it induces a map to the Thom space.
\begin{lem}\label{Li-cont}
Let $f=(f_1,\dots, f_n):X\to \RR^{dn}$ be a $G$-condensed map for a graph $G\in\GG(P)$,  and $i$ a number such that  $i, i+1$ are contained in a single piece $\alpha$ of $P^\circ$.  Suppose either of the following conditions holds: 
\begin{itemize}

\item  There exists a base $\beta$ of the component containing $\alpha$, satisfying $\beta\not=\alpha$.
\item  $\alpha$ is discrete and there exists a vertex $\beta\not=\alpha$ such that for both of  $i'=i,i+1$ and $x\in X$,  there is a point $y_{i'}(x)$ in the convex hull of $\{f_j(x)\}_{j\in\beta}$ satisfying  $f_{i'}(x)=_1y_{i'}(x)$.  
\end{itemize}
Then, the $(i,\eps)$-contraction $f^{i\eps}$ of $f$ is proper and its image is contained in  $U_{G}\cap U^1_{G}$. In particular, it induces a map $f^{i\eps}:X^*\wedge[0,\infty]\to \TT_G$.

\end{lem}
\begin{proof}
Put  $f^{i\eps}=f'=(f'_1,\dots, f'_n)$. Similarly to the proof of Lemma \ref{Lcondensed}, we can see that $f'$ is proper. (The condition in the case of discrete $\alpha$ is used here.)  We shall show $f'(X\times [0,\infty))\subset U_{G}$ in the case of non-discrete $\alpha$. In the following proof, we denote by $j_\gamma\in \gamma$ the minimum of a piece $\gamma$.  Suppose $f'(\tilde x)\in\nu_{P}$ for $\tilde x=(x,s)\in X\times [0,\infty)$. Put $\pi_P(f'(\tilde x))=(y_\gamma)_\gamma$. Let   $\gamma\in P$ be a piece such that there is an edge incident with $\alpha$ and $\gamma$.  By definition, we have
\[
|y_\alpha-\frac{1}{2}(f'_i(\tilde x)+f'_{i+1}(\tilde x))|<\frac{\rho (c_\alpha-c_{j_\alpha})}{2}+\eps_P.
\]
 The first coordinate of $\frac{1}{2} (f_i(x)+f_{i+1}(x))$ is in the image of projection of the convex hull of $\{f_j(x)\}_{j\in \beta}$ to the first coordinate,  so 
\[
|y_{\beta 1}-\frac{1}{2}(f_{i1}(x)+f_{i+1,1}(x))|<\frac{\rho (c_\beta-c_{j_\beta})}{2}+\eps_P,
\]
where the subscripts $1$ indicate the first coordinates and we also use similar notations below. Clearly, $\frac{1}{2} (f'_i(\tilde x)+f'_{i+1}(\tilde x))=\frac{1}{2} (f_i(x)+f_{i+1}(x))$, and
 we have $|y_{\alpha 1}-y_{\beta 1}|<\rho (c_\alpha+c_\beta-c_{j_\alpha}-c_{j_\beta})/2+2\eps_P$. Similarly, we have $|y_{\gamma 1}-y_{\beta 1}|<\rho (c_\gamma+c_\beta-c_{j_\gamma}-c_{j_\beta})/2+2\eps_P$. Since $\beta$ is smaller than ( one of ) $\alpha, \gamma$,
we have
$|y_{\alpha 1}-y_{\gamma 1}|<d_{\alpha\gamma}(P)-\rho c_{j_\beta}$ by the assumption on $c_r$ in Definition \ref{DPK}.  We also have $f'_{j2}(x)=f_{j 2}(x)$ for any $1\leq j\leq n$ by definition, where the extra subscripts $2$ mean the $n-1$-tuples from  the second to $n$-th coordinate. Since $f'(\tilde x)\in \nu_{P}$ and the map $e_P$ arranges the points in a common piece along the direction of $u$, for any pair $j,j'\in \beta$ we have $|f_{j2}(\tilde x)-f_{j',2}(\tilde x)|<2\eps_{P}$. This observation implies  $|y_{\alpha 2}-y_{\gamma 2}|<4\eps_{P}$. Thus, we have 
\[
|y_\alpha-y_\gamma|\leq |y_{\alpha1}-y_{\gamma1}|+|y_{\alpha2}-y_{\gamma2}|<d_{\alpha\gamma}(P)-\rho c_{j_\beta}+4\eps_{P}<d_{\alpha\gamma}(P).
\]
For other components, we see the analogous inequality as in Lemma \ref{Lcontraction}, and we have proved $f'(X\times [0,\infty))\subset U_{G}$. The inclusion to $U^1_G$ is completely similar. 
\end{proof}
\begin{exa}\label{Ei-cont}
Put $G=(1,4)(2,4)(3,5)\in \GG([6])$ and $e=(1,4)$. ($G$ is the same as $G_2$ in Figure \ref{Fgraphs_ch3}.) Let $F$ be the $e$-contraction of $f_G$. The $(1,\pm)$-contraction of $F$ is given by $(x,y,s_1,s_2)\mapsto (x+s_1v\pm s_2u, x-s_1v\mp s_2u, y, x-s_1v, y)$ and its image is contained in $U_{\delta_1\partial_eG}\cap U^1_{\delta_1\partial_eG}$ by Lemma \ref{Li-cont}.
\end{exa}
Other examples of application of  Lemma \ref{Li-cont} is given in Definitions \ref{Dchain_cycle} and \ref{D2nd_chain_2graphs}. The following lemma is also used in Definition \ref{D2nd_chain_2graphs}.
\begin{lem}\label{Li-cont2}
Let $f=(f_1,\dots, f_n):X\to \RR^{dn}$ be a $G$-condensed map for a graph $G\in\GG(P)$,  $i$ a number such that the set $\{i, i+1, i+2\}$ is a base for $(G,f)$. 
\begin{enumerate}
\item Suppose there exists a point $y(x)$ on the segment between $f_{i+1}(x)$ and $f_{i+2}(x)$ satisfying $f_i(x)=_1y(x)$ for each $x\in X$.
The image of $(i,+)$-contraction $f^{i+}$ of $f$ is contained in $\RR^{dn}-\nu_P$ so it induces the constant map to $*$ on the pointed spaces.
 The $(i,-)$-contraction $f^{i-}$ of $f$ is proper and its image is contained in  $U_{G}\cap U^1_{G}$. In particular, it induces a map $f^{i-}:X^*\wedge[0,\infty]\to \TT_{G}$.
\item Suppose there exists a point $y(x)$ on the segment between $f_{i}(x)$ and $f_{i+1}(x)$ satisfying $f_{i+2}(x)=_1y(x)$ for each $x\in X$.
 The image of $(i+1,+)$-contraction $f^{i+1,+}$ of $f$ is contained in $\RR^{dn}-\nu_P$ so it induces the constant map to $*$ on the pointed spaces.
 The $(i+1,-)$-contraction $f^{i+1,-}$ of $f$ is proper and its image is contained in  $U_{G}\cap U^1_{G}$. In particular, it induces a map $f^{i+1, -}:X^*\wedge[0,\infty]\to \TT_{G}$.

\end{enumerate}

\end{lem}
\begin{proof} We shall prove part 1. We put $f'=f^{i\eps}$ and let $\tilde x=(x,s)\in X\times [0,\infty)$. First set $\eps=+$. If $f_{i+1}(x)\leq_1f_{i+2}(x)$, we have $f'_{i+1}(\tilde x)=f_{i+1}(x)-su\leq _1f_i(x)+su=f'_i(\tilde x)$. By Lemma \ref{Lcondensed_collapse0}, we have $f'(\tilde x)\not\in \nu_P$. If $f_{i+1}(x)>_1f_{i+2}(x)$,  we have $f'_{i+2}(\tilde x)\leq _1f'_i(\tilde x)$ by the assumption, and $f'(\tilde x)\not\in \nu_P$. Next, we set $\eps=-$.
Suppose $f'(\tilde x)\in (D^d_R)^n$. We have $f_{i+2}(x)=f'_{i+2}(\tilde x)\in D^d_R$ and $(f_{i}(x)+f_{i+1}(x))/2=(f'_i(\tilde x)+f'_{i+1}(\tilde x))/2\in D^d_R$. By these relations, we have $f_{i+1}(x)\in D^d_{3R}$. With this relation, an argument similar to the proof of Lemma \ref{Lcontraction} shows $f'$ is proper. Suppose $f'(\tilde x)\in \nu_{P}$ and put $\pi_P(f'(\tilde x))=(y_\gamma)_\gamma$. 
In this case, we have $f_{i+1}(x)<_1f_{i+2}(x)$ and $s<\rho(c_i+c_{i+1})/4+\eps_{P}$. So we have
\[
\begin{split}
\rho(c_{i+1}+c_{i+2})/2-2\eps_{P}\leq_1 f_{i+2}(x)-f_{i+1}(x)&\leq_1\rho(c_{i+1}+c_{i+2})/2+2\eps_{P}+\rho (c_i+c_{i+1})/4+\eps_{P}\\
&=\rho(c_i+3c_{i+1}+2c_{i+2})/4+3\eps_{P}.
\end{split}
\]
Put $\alpha=\{i,i+1,i+2\}$. $y_\alpha$ is approximately $f_{i+2}(x)-\rho c_{i \to i+2} u/2$ with an error of norm $<\eps_{P}$, and we have 
\[
\rho(c_i+3c_{i+1}+2c_{i+2})/4+3\eps_{P}-\rho c_{i\to i+2}/2<\rho c_{i,i+2}/2.
\]
So the larger of distances between $y_\alpha$ and $f_{i'}(x)$ for $i'=i+1,i+2$ is smaller than $\rho c_{i,i+2}/2+\eps_P$. Also, the distance between $y_\alpha$ and the segment between $f_{i+1}(x)$ and $f_{i+2}(x)$ is smaller than $\eps_P$.
Let $(\beta, \bar \beta)$ be an edge of $G$ with  $\beta\sim_{G}\alpha$. By an argument similar to the proof of Lemma \ref{Li-cont}, together with the above argment, we see  
 $|y_\alpha-y_\beta|, |y_\alpha-y_{\bar \beta}|<\rho c_{i,i+2}/2+4\eps_P$. This implies $|y_\beta-y_{\bar\beta}|<d_{\beta,\bar\beta}(P)$.
\end{proof}
Part 1 of the following lemma is the reason why we need both signs in $(i,\pm)$-contractions. This lemma is used in Definitions \ref{D2nd_chain_2contractions} and \ref{D3rd_chain_2graphs}.

\begin{lem}\label{Li-cont_homotopy}
Let $f=(f_1,\dots, f_n):X\to \RR^{dn}$ be a $G$-condensed map for a graph $G\in \GG(P)$, and  $i, j$ two numbers such that $i<j$ and $\{i,i+1\}$ and $\{j,j+1\}$ are included in some pieces $\alpha$ and $\beta$ of $P^\circ$, respectively. Let $f^{i\eps}$ and $f^{j\eps}$ denote the $(i,\eps)$- and $(j,\eps)$-contractions of $f$ for $\eps=\pm$.
\begin{enumerate}
\item Suppose that $\alpha\not=\beta$, $\beta$ is a base,  and if $\alpha$ is not discrete, $\alpha$ and $\beta$ belong to the same component.  Furthermore, suppose  either of the following conditions holds.
\begin{enumerate}
\item $f_{i}=_1f_{j}$ and $f_{i+1}=_1f_{j+1}$, or 
\item $f_{i}=_1f_{j+1}$ and $f_{i+1}=_1f_{j} $ 
\end{enumerate}
We put $\eps'=-\eps$ (resp. $\eps$) in the case (a) (resp. (b)). The straight homotopy $h$ from  $f^{i\eps}$ to $f^{j\eps'}$  is proper and its image is contained in  $U_{G}\cap U^1_{G}$. So, $h$ induces a map $h:X^*\wedge[0,\infty]\wedge(I_+) \to \TT_{G}$. Furthermore, if $\eps=-$ (resp. $+$)  in the case (a) (resp. (b)), the induced map is $*$.

\item Suppose that $i+1=j$ (so $\alpha=\beta$), and there is a base $\gamma$ which does not contain any of $i,i+1,i+2$ and satisfy either of the following two conditions.
\begin{itemize}
\item  $\alpha\sim_G\gamma$.
\item  $\alpha$ is discrete and there exists a point $y_{i'}(x)$ in the convex hull of $\{f_k(x)\}_{k\in \gamma}$ satisfying $y_{i'}(x)=_1f_{i'}(x)$ for any $i'=i,i+1,i+2$ and $x\in X$.
\end{itemize}
For any pair $\eps, \eps'=\pm$, the straight homotopy $h$ from  $f^{i\eps}$ to $f^{j\eps'}$  is proper and its image is contained in $U_{G}\cap U^1_{G}$ . So, $h$ induces a map $h:X^*\wedge[0,\infty]\wedge(I_+) \to \TT_{G}$.
\end{enumerate}
\end{lem}
\begin{proof}
The proof is similar to that of Lemma \ref{Li-cont}. The choice of $\eps'$ in part 1 ensures $h$ is proper. 
\end{proof}

\section{Computation of a differential in characteristic 2}\label{Sch2}
In this section, we will prove part 1 of Theorem \ref{Tmain}. Throughout this section, we set $n=4$ and $d=2$ and assume that the base field $\kk$ is  of characteristic $2$. 
 By a straightforward computation, we see that the element 
\begin{equation}
g_{14}g_{23}+g_{13}g_{24}+g_{12}g_{34} \label{EQch2}
\end{equation} in $\EE^{-4,2}_1$ is a cycle for the $d_1$-differential  in characteristic $2$. We also see that the  element in $\CEE^{-4,2}_1$ given by the same formula (\ref{EQch2}) is also a $d_1$-cycle  (see the  paragraphs after Lemma \ref{Lspec_seq2}). Actually it is enough to compute the corresponding differential of the projection of the element  to the truncated  sequence $tr_1\CEE$. The computation is based on  the description of the differential given in Lemma \ref{Lbasic_spec}. We will apply this lemma to the double complex $(\TTH, \delta, d+(-1)^*\partial )$. We have $d+(-1)^*\partial =d+\partial$ since the characteristic is $2$.
\\ 
\indent We define three graphs in $\GG([5])$ as follows:
\[
G_1=(1,4)(2,3),\qquad G_2=(1,3)(2,4),\qquad G_3=(1,2)(3,4).
\]
See Figure  \ref{Fgraphs_ch2}. Throughout this section, $G_i$ denotes one of these graphs (not those in sections \ref{Snon-triviality} and \ref{Sch3}).

\begin{defi}\label{Dchain_ch2}
For $G=G_1,G_2,$ and $G_3$, put $f=f_G$ and $E(G)=\{e_1<e_2\}$ (see Definition \ref{Dcondensed}). For $j=1, 2$, let $f_j$ be the $e_k$-contraction of $f$ for $G$. 
Set
\[
c(G)=f(w_0)+f_{1}(w_1)+f_{2}(w_1)\quad \in \bar C_4(\TT_{G})\oplus \bar C_5(\TT_{\partial_1G})\oplus \bar C_5(\TT_{\partial_2G})\quad \subset tr_1(\tot_{10}\TTH).
\]
Here, by our convention, $f_j(w_1)$ denotes the pushforward of $w_1$ by the induced map $f_j:S^4\wedge[0,\infty]\to \TT_{\partial_jG}$ and $f(w_0)$ is a similar abbreviation. For well-definedness, see Lemma \ref{L1st_cyc_ch2} below.
\end{defi}

\begin{figure}
\begin{center}
{\unitlength 0.1in%
\begin{picture}(22.0000,1.9000)(0.3000,-2.3000)%
%
\special{pn 8}%
\special{pa 30 190}%
\special{pa 630 190}%
\special{fp}%
%
\special{pn 8}%
\special{pa 80 190}%
\special{pa 107 165}%
\special{pa 135 141}%
\special{pa 162 119}%
\special{pa 190 98}%
\special{pa 217 79}%
\special{pa 245 63}%
\special{pa 272 51}%
\special{pa 300 43}%
\special{pa 327 40}%
\special{pa 354 42}%
\special{pa 382 49}%
\special{pa 409 60}%
\special{pa 437 75}%
\special{pa 464 93}%
\special{pa 492 114}%
\special{pa 519 136}%
\special{pa 546 160}%
\special{pa 574 185}%
\special{pa 580 190}%
\special{fp}%
%
\special{pn 8}%
\special{pa 230 190}%
\special{pa 257 166}%
\special{pa 285 146}%
\special{pa 312 133}%
\special{pa 340 131}%
\special{pa 367 141}%
\special{pa 395 160}%
\special{pa 422 183}%
\special{pa 430 190}%
\special{fp}%
%
\special{pn 8}%
\special{pa 830 190}%
\special{pa 1430 190}%
\special{fp}%
%
\special{pn 8}%
\special{pn 8}%
\special{pa 880 190}%
\special{pa 886 185}%
\special{fp}%
\special{pa 913 160}%
\special{pa 919 154}%
\special{fp}%
\special{pa 948 130}%
\special{pa 954 125}%
\special{fp}%
\special{pa 984 103}%
\special{pa 990 98}%
\special{fp}%
\special{pa 1021 77}%
\special{pa 1028 73}%
\special{fp}%
\special{pa 1061 56}%
\special{pa 1068 53}%
\special{fp}%
\special{pa 1104 43}%
\special{pa 1112 42}%
\special{fp}%
\special{pa 1149 42}%
\special{pa 1156 43}%
\special{fp}%
\special{pa 1192 53}%
\special{pa 1199 56}%
\special{fp}%
\special{pa 1233 73}%
\special{pa 1240 77}%
\special{fp}%
\special{pa 1270 98}%
\special{pa 1277 102}%
\special{fp}%
\special{pa 1306 125}%
\special{pa 1312 130}%
\special{fp}%
\special{pa 1340 155}%
\special{pa 1346 160}%
\special{fp}%
\special{pa 1374 185}%
\special{pa 1380 190}%
\special{fp}%
%
\special{pn 8}%
\special{pa 1030 190}%
\special{pa 1057 166}%
\special{pa 1085 146}%
\special{pa 1112 133}%
\special{pa 1140 131}%
\special{pa 1167 141}%
\special{pa 1195 160}%
\special{pa 1222 183}%
\special{pa 1230 190}%
\special{fp}%
%
\special{pn 8}%
\special{pa 1630 190}%
\special{pa 2230 190}%
\special{fp}%
%
\special{pn 8}%
\special{pa 1680 190}%
\special{pa 1707 165}%
\special{pa 1735 141}%
\special{pa 1762 119}%
\special{pa 1790 98}%
\special{pa 1817 79}%
\special{pa 1845 63}%
\special{pa 1872 51}%
\special{pa 1900 43}%
\special{pa 1927 40}%
\special{pa 1954 42}%
\special{pa 1982 49}%
\special{pa 2009 60}%
\special{pa 2037 75}%
\special{pa 2064 93}%
\special{pa 2092 114}%
\special{pa 2119 136}%
\special{pa 2146 160}%
\special{pa 2174 185}%
\special{pa 2180 190}%
\special{fp}%
%
\special{pn 8}%
\special{pn 8}%
\special{pa 1830 190}%
\special{pa 1836 185}%
\special{fp}%
\special{pa 1865 160}%
\special{pa 1872 156}%
\special{fp}%
\special{pa 1904 137}%
\special{pa 1911 133}%
\special{fp}%
\special{pa 1948 134}%
\special{pa 1956 137}%
\special{fp}%
\special{pa 1988 156}%
\special{pa 1995 160}%
\special{fp}%
\special{pa 2024 185}%
\special{pa 2030 190}%
\special{fp}%
\put(8.9000,-2.9000){\makebox(0,0){$+$}}%
\put(13.8000,-2.9000){\makebox(0,0){$-$}}%
\put(18.3500,-2.9500){\makebox(0,0){$+$}}%
\put(20.3000,-2.9500){\makebox(0,0){$-$}}%
\put(7.3000,-1.4500){\makebox(0,0){$+$}}%
\put(15.3000,-1.4000){\makebox(0,0){$+$}}%
\end{picture}}%
\end{center}
\caption{$c(G_1)$ : We use the notations in Figures \ref{Fgraphs_ch2} and \ref{Fpsi}. }\label{Fc(G_1)}
\end{figure}
\begin{exa}
Let $G=G_1$. Under the notations of Definition \ref{Dchain_ch2}, see Example.\ref{Ef_G} for the concrete formulas of $f$ and $f_1$. The map $f_2$ is given by $(x,y,s)\mapsto (x,y+sv,y-sv,x)$. See Figure \ref{Fc(G_1)} for a graphical expression of $c(G_1)$.
\end{exa}

Set $D=d+\partial$. 
\begin{lem}\label{L1st_cyc_ch2}
For $G=G_1,G_2,$ and $G_3$, $c(G)$ is a cycle in $(tr_1(\tot\TTH), D)=(tr_1\CEE_0,d_0)$.
\end{lem}
\begin{proof}
Under the notations of Definition \ref{Dchain_ch2},  $f_j$ is $\partial_jG$-condensed by Lemma \ref{Lcontraction}. So by Lemma \ref{Lcondensed}, each pushfoward in the definition of $c(G)$ is well-defined. Clearly, we have $df(w_0)=0$ and $\partial_jf(w_0)=df_j(w_1)$. We also have $\partial_kf_j(w_1)=0$ in the truncated complex. These equalities imply the claim. 
\end{proof}
By Lemma \ref{Lbasic_spec}, we easily see that  $c(G_i)$ represents the projection of the $i$-th term of the element (\ref{EQch2}) (but we do not use this fact below).  Let us compute the differential of the element represented by $\sum_{1\leq i\leq 3}c(G_i)$.

\begin{defi}\label{Dchainhomotopy}
For $(G,H,i)=(G_1,G_2,3), (G_1,G_2,1)$ and $(G_2,G_3,2)$, we shall define a bounding chain of $\delta_i(c(G)+c(H))$. Set $f=f_G$ and $f'=f_H$. If the $i$-th components of $f$ and $f'$ are identical, $\psi$ denotes the straight homotopy from $f$ to $f'$. Otherwise, $\psi$ is the straight homotopy from $f$ to $f'\circ T$, where $T:\RR^4\to \RR^4$ is the transposition $T(x,y)=(y,x)$.
Put $E(G)=\{e_1<e_2\}$ and $E(H)=\{e'_1<e'_2\}$. We regard $E(\delta_iG)=E(G)$ and $E(\delta_iH)=E(H)$ by the standard bijection (see Definition \ref{Dgraph}). Clearly, we have $\delta_iG=\delta_iH$ (as elements of $\GG(\delta_i[5])$). Let 
\begin{enumerate}
\item $\lambda_j$ be the straight homotopy from the $e_j$-contraction of $f$ for $G$ to the  $e_j$-contraction of $f$ for $\delta_iG$,
\item $\lambda'_j$  the straight homotopy from the $e'_j$-contraction of $f'$ for $H$ to the  $e'_j$-contraction of $f'$ for $\delta_iH$,
\item $\psi_j$ be the $e_j$-contraction of $\psi$ for $\delta_iG$.
\end{enumerate} 
$\psi$, $\psi_j$, $\lambda_j$ and $\lambda'_j$ induce the maps $\psi:S^4\wedge(I_+)\to \TT_{\delta_iG}$, $\psi_j:S^4\wedge(I_+)\wedge[0,\infty]\to \TT_{\delta_i\partial_jG}$, $\lambda_j:S^4\wedge[0,\infty]\wedge (I_+)\to \TT_{\delta_i\partial_jG}$, and $\lambda'_j:S^4\wedge[0,\infty]\wedge (I_+)\to \TT_{\delta_i\partial_{j'}G}$, respectively, where $j'$ is a unique number with $\delta_i\partial_{j'}G=\delta_i\partial_{j}H$ (see the proof of Lemma \ref{Lh_i_welldef} below). 
Set
\[
\begin{split}
c(G,H,i)=\psi(w_{01})+\sum_{j=1,2}&(\lambda_j+\lambda'_j+\psi_j)(w_{11})\\
&\in \bar C_5(\TT_{\delta_iG})\oplus \bar C_6(\TT_{\delta_i\partial_1G})\oplus \bar C_6(\TT_{\delta_i\partial_2G})\quad \subset tr_1(\tot_{10}\TTH).
\end{split}
\]
Here, by the convention in subsection \ref{SSNT}, $(\lambda_j+\lambda'_j+\psi_j)(w_{11})$ denotes $\lambda_j(w_{11})+\lambda'_j(w_{11})+\psi_j(w_{11})$, and we compose $\psi_j$ with the transposition of $[0,\infty]\wedge (I_+)$ implicitly since definition of $\psi_j$ puts $[0,\infty]$ at the rightmost component. See Figure \ref{Fpsi} for $(G_1,G_2,3)$.

\end{defi}
\begin{exa}\label{Emaps_ch2}
For $(G, H, i)=(G_1,G_2,3)$, we have
\[
\begin{split}
\lambda_1(x,y,s,t)&=(x+sv,y-stv, y-stv,x-sv)\\
\lambda_2(x,y,s,t)&=(x-stv,y+sv,y-sv,x-stv),\\
\lambda'_1(x,y,s,t)&=(x+sv,y-stv,x-sv,y-stv),\\
\lambda'_2(x,y,s,t)&=(x-stv,y+sv,x-stv,y-sv),\\
\psi(x,y,t)&=((1-t)x+ty, tx+(1-t)y, y, x), \\
\psi_1(x,y,t,s)&=((1-t)x+ty+sv, tx+(1-t)y-sv, y-sv, x-sv), \\
\psi_2(x,y,t,s)&=((1-t)x+ty-sv, tx+(1-t)y+sv, y-sv, x-sv),
\end{split}
\]
\end{exa}

\begin{lem}\label{Lh_i_welldef}
For $(G,H,i)=(G_1,G_2,3), (G_1,G_2,1)$ and $(G_2,G_3,2)$, the chain $c(G,H,i)$ is well-defined and satisfies 
\[
Dc(G,H,i)=\delta_i(c(G)+c(H)).
\] If we set $C=c(G_1,G_2,3)+c(G_1,G_2,1)+c(G_2,G_3,2)$, we have
\[
D(C)=\delta(c(G_1)+c(G_2)+c(G_3)).
\]
(In particular, $c(G_1)+c(G_2)+c(G_3)$ represents a cycle in $(tr_1\CEE^{-4,2}_1,d_1)$.)
\end{lem}
\begin{proof}
We use the notations in Definition \ref{Dchainhomotopy}.  The equation $\delta_iG=\delta_iH$, together with Lemmas \ref{Lcontraction} (3) and   \ref{Lstraight_homotopy} implies that $\psi$ is $\delta_iG$-condensed. The maps $\psi_j$ and  $\lambda_j$ are $\delta_i\partial_jG$-condensed and  $\lambda'_j$ is $\delta_i\partial_jH$-condensed similarly. (The admissibility of $i$ and the edges are easily confirmed.) By Lemma \ref{Lcondensed}, each pushforward in the definition of $c(G,H,i)$ is well-defined. 
Roughly speaking, the first equation in the claim holds since concatenation of $\psi_j$, $\lambda_j$, and $\lambda'_{j'}$ gives a homotopy between the $e_j$-contraction of $f$ for $G$ and $e'_{j'}$-contraction of $f'$ or $f'\circ T$ for $H$, where $j'=j$ if the composition of standard bijections $E(G)\cong E(\delta_iG)=E(\delta_iH)\cong E(H)$  preserve the order of edges, and $j'=3-j$ otherwise. Let us look at the case $(G,H,i)=(G_1,G_2,1)$ more closely. We have
\[
\begin{split}
d\psi(w_{01})&=\delta_i'\circ f(w_0)+\delta_i'\circ f'(w_0),\\
d\psi_j(w_{11})&=\psi(w_{01})+\psi_j|_{t=0}(w_1)+\psi_j|_{t=1}(w_1), \\
d\lambda_j(w_{11})&=\lambda_j|_{s=0}(w_{01})+\lambda_j|_{t=0}(w_1)+\lambda_j|_{t=1}(w_1),
\end{split}
\]
and similar equation for  $\lambda'_{j'}$, where $s\in [0,\infty)$ and $t\in I$. The map $\delta_i'\circ f$ is the composition of  $f:S^4\to \TT_G$ with the natural collapsing map $\delta_i':\TT_G\to\TT_{\delta_iG}$, see Definition \ref{Dtriple}. (Since $f$ is also $\delta_iG$-condensed, $f$ itself can represent the map to $\TT_{\delta_iG}$, but we use this notation to clarify, and use similar notations in the rest of the paper.) The maps $\psi_1|_{t=0}$ and $\psi_1|_{t=1}$ are equal to $\lambda_1|_{t=1}$ and $\lambda'_2|_{t=1}$, respectively, as in the following figure.
\begin{figure}[H]
\begin{center}
{\unitlength 0.1in%
\begin{picture}(25.1500,3.2800)(-1.1300,-3.3800)%
%
\special{pn 8}%
\special{pa 261 231}%
\special{pa 1093 231}%
\special{fp}%
%
\special{pn 8}%
\special{pn 8}%
\special{pa 323 238}%
\special{pa 329 232}%
\special{fp}%
\special{pa 356 206}%
\special{pa 362 200}%
\special{fp}%
\special{pa 390 174}%
\special{pa 396 168}%
\special{fp}%
\special{pa 424 142}%
\special{pa 430 137}%
\special{fp}%
\special{pa 459 112}%
\special{pa 466 107}%
\special{fp}%
\special{pa 496 83}%
\special{pa 503 78}%
\special{fp}%
\special{pa 534 57}%
\special{pa 542 53}%
\special{fp}%
\special{pa 575 35}%
\special{pa 583 31}%
\special{fp}%
\special{pa 618 18}%
\special{pa 626 16}%
\special{fp}%
\special{pa 664 11}%
\special{pa 672 10}%
\special{fp}%
\special{pa 710 14}%
\special{pa 718 16}%
\special{fp}%
\special{pa 755 27}%
\special{pa 762 30}%
\special{fp}%
\special{pa 797 47}%
\special{pa 804 51}%
\special{fp}%
\special{pa 836 72}%
\special{pa 842 77}%
\special{fp}%
\special{pa 873 100}%
\special{pa 879 105}%
\special{fp}%
\special{pa 908 130}%
\special{pa 915 135}%
\special{fp}%
\special{pa 943 161}%
\special{pa 949 167}%
\special{fp}%
\special{pa 977 193}%
\special{pa 983 199}%
\special{fp}%
\special{pa 1010 225}%
\special{pa 1016 231}%
\special{fp}%
%
\special{pn 8}%
\special{pa 538 238}%
\special{pa 564 210}%
\special{pa 590 183}%
\special{pa 615 161}%
\special{pa 641 144}%
\special{pa 667 136}%
\special{pa 692 137}%
\special{pa 717 148}%
\special{pa 742 166}%
\special{pa 767 190}%
\special{pa 792 218}%
\special{pa 815 245}%
\special{fp}%
%
\special{pn 8}%
\special{pa 1563 227}%
\special{pa 2402 227}%
\special{fp}%
%
\special{pn 8}%
\special{pa 1654 227}%
\special{pa 1680 199}%
\special{pa 1705 172}%
\special{pa 1731 147}%
\special{pa 1757 123}%
\special{pa 1782 102}%
\special{pa 1808 85}%
\special{pa 1834 71}%
\special{pa 1860 63}%
\special{pa 1886 60}%
\special{pa 1913 63}%
\special{pa 1939 71}%
\special{pa 1966 84}%
\special{pa 1992 100}%
\special{pa 2019 120}%
\special{pa 2045 143}%
\special{pa 2072 168}%
\special{pa 2099 194}%
\special{pa 2126 222}%
\special{pa 2131 227}%
\special{fp}%
%
\special{pn 8}%
\special{pn 8}%
\special{pa 1854 227}%
\special{pa 1859 221}%
\special{fp}%
\special{pa 1885 194}%
\special{pa 1890 188}%
\special{fp}%
\special{pa 1916 161}%
\special{pa 1922 155}%
\special{fp}%
\special{pa 1949 130}%
\special{pa 1955 125}%
\special{fp}%
\special{pa 1985 102}%
\special{pa 1991 97}%
\special{fp}%
\special{pa 2023 78}%
\special{pa 2031 74}%
\special{fp}%
\special{pa 2066 62}%
\special{pa 2074 62}%
\special{fp}%
\special{pa 2111 63}%
\special{pa 2119 64}%
\special{fp}%
\special{pa 2154 77}%
\special{pa 2161 80}%
\special{fp}%
\special{pa 2193 100}%
\special{pa 2199 105}%
\special{fp}%
\special{pa 2229 127}%
\special{pa 2235 133}%
\special{fp}%
\special{pa 2263 158}%
\special{pa 2269 163}%
\special{fp}%
\special{pa 2296 189}%
\special{pa 2301 194}%
\special{fp}%
\special{pa 2327 221}%
\special{pa 2333 227}%
\special{fp}%
\put(13.3200,-1.5800){\makebox(0,0){$\simeq$}}%
%
\special{pn 8}%
\special{pa 305 175}%
\special{pa 575 175}%
\special{pa 575 283}%
\special{pa 305 283}%
\special{pa 305 175}%
\special{pa 575 175}%
\special{fp}%
%
\special{pn 8}%
\special{pa 1630 175}%
\special{pa 1899 175}%
\special{pa 1899 283}%
\special{pa 1630 283}%
\special{pa 1630 175}%
\special{pa 1899 175}%
\special{fp}%
\put(1.2200,-1.6400){\makebox(0,0){$\psi_1:$}}%
\put(3.4000,-3.9600){\makebox(0,0){$+$}}%
\put(10.1000,-3.9500){\makebox(0,0){$-$}}%
\put(8.0000,-3.9500){\makebox(0,0){$+$}}%
\put(5.4500,-3.9500){\makebox(0,0){$+$}}%
\put(16.6000,-4.0300){\makebox(0,0){$+$}}%
\put(23.3000,-4.0200){\makebox(0,0){$-$}}%
\put(21.2000,-4.0200){\makebox(0,0){$+$}}%
\put(18.6500,-4.0200){\makebox(0,0){$+$}}%
\end{picture}}%
\phantom{aaaaaaa}
\end{center}
\vspace{4mm}
\caption{$\psi_1$ for $(G_1,G_2,1)$}\label{Fpsi1edge}
\vspace{-7mm}
\end{figure}
Similarly, the maps $\psi_2|_{t=0}$ and $\psi_2|_{t=1}$ are equal to $\lambda_2|_{t=1}$ and $\lambda'_1|_{t=1}$, respectively.
By definition, $\lambda_j|_{t=0}=\delta_i'\circ f_j$ in the notation of Definition \ref{Dchain_ch2}.  The map $\lambda_j|_{s=0}$ is constant for the variable $t$ since $t$ is only multiplied on $A_e(s)$ in Definition \ref{Dcondensed2}.  So in the normalized singular complex, we have $\lambda_j|_{s=0}(w_{01})=0$. The term $\psi(w_{01})$ in the equation for $d\psi_j(w_{11})$ cancels with $\partial_j \psi(w_{01})$. The transposition $T$ does not matter since $T(w_0)=w_0$. By these observation, we have $Dc(G,H,i)=\delta_i(c(G)+c(H))$. 
We now turn to the proof of  the equation for $D(C)$ in the claim.  We claim that the chains $\delta_jc(G_k)$ which are not included in $Dc(G,H,i)$ are zero by Lemma \ref{Lcondensed_collapse} (1). For any $k$,  we have $\delta_0c(G_k)=\delta_4c(G_k)=0$ as follows. For example, we consider the case of $G_1$. Put $f_j=(f_{j,1},\dots, f_{j,4})$ in the notation of Definition \ref{Dchain_ch2} for $G=G_1$. While the vertex $\{1\}$ of $\partial_1G_1$ is discrete, we have $f_{1,l}=_1f_{2,l}$ and the vertex $\{1\}$ of  $\partial_2G_1$ is not discrete. So by Lemmas \ref{Lcondensed} and \ref{Lcondensed_collapse} (1)(b), we have $\delta_0f_1(w_1)=0$.  Similarly,   by Lemmas \ref{Lcondensed} and \ref{Lcondensed_collapse} (1)(a), we have $\delta_2c(G_1)=\delta_1c(G_3)=\delta_3c(G_3)=0$ as $\delta_kG_j$ has a loop in these cases.  The equation for $D(C)$ follows from the  equation for $Dc(G, H,i)$ and these observations.  
\end{proof}

\begin{lem}\label{Lfundcyc} Let $C$ be the chain defined in Lemma \ref{Lh_i_welldef}, and    $\psi_j$ the map given in Definition \ref{Dchainhomotopy} for $(G,H,i)=(G_1, G_2,3)$.  Put $G_0=\delta_{13}\partial_1G_1$, see Definition \ref{Dgraph}. $G_0$ is the graph with only one edge $(\{12\},\{34\})$. 
\begin{enumerate}
\item Each pushforward which appears as  the terms of $c(G,H,i)$ are annihilated by $\delta_k$ unless $(G,H,i,k)=(G_1,G_2,3,1)$. Moreover, 
the terms of $c(G_1,G_2,3)$, except for $\psi_j(w_{11})$, are annihilated by $\delta_1$.

\item  The space $\TT_{G_0}$ contains $S^6$ as a deformation retract. For a retraction $\tilde r:\TT_{G_0}\to S^6$,\\
$\tilde r(\delta C)=\tilde r(\delta_1(\psi_1(w_{11})+\psi_2(w_{11})))$ is a fundamental cycle of $S^6$.
\end{enumerate}
\end{lem}
\begin{proof}
For part 1, we consider the case of $(G,H,i)=(G_1,G_2,1)$. So for a while, $f, \psi, \psi_j$ and $\lambda_j$ denote the maps given in Definition \ref{Dchainhomotopy} for this triple. We have $\delta_1'\circ \psi=*$ by Lemma \ref{Lcondensed_collapse}(1)(a) as $\delta_1(\delta_1G)$ has a loop. Since the first coordinate of each component of $\psi_j(x,y,s,t)$ is equal to that of the corresponding component of $\psi(x,y,t)$, by the same lemma, we have $\delta_1'\circ \psi_j=*$. We also see  $\delta'_1\circ\lambda_{2}=*$  as the first coordinate of each component of $\lambda_2(x,y,s,t)$ is equal to that  of $f(x,y)$. We see $\delta'_2\circ \lambda_j=*$ by Lemma \ref{Lcondensed_collapse1}. We also have $\delta'_2\circ\psi_j=*$ by Lemma \ref{Lcondensed_collapse} (2). The other terms vanish similarly.\\
\indent We shall show part 2. In the rest of proof, $\psi_j$ denotes the map for $(G,H,i)=(G_1,G_2,3)$.
 Using  the notations of Definition \ref{Dpartition}, put $N=D_{\alpha\beta}\cap (\overline{\RR^4-E_P})$, where $\alpha=\{12\}, \beta=\{34\}$ and $P=\{\{0\},\alpha, \beta,\{5\}\}$, and the overline indicates the closure (taken in $\RR^4$). We regard the space $\nu_P|_N:=\nu_P\cap\pi_P^{-1}(N)$ as a disk bundle over $N$ with the projection $\pi_P$.  By definition, $\TT_{G_0}$ is naturally identified with 
the Thom space associated to $\nu_P|_N$ defined  by collapsing the boundary of each fiber and the preimage of $N\cap \partial E_P\subset \partial N$ (and not collapsing the preimage of the rest of $\partial N$). Since $\overline{\RR^4-E_P}$ is  a product of two disks  whose neighborhoods of  north and south poles  are cut off,  and $N$ is a tubular neighborhood of the diagonal $\Delta=\{(a,a)\in \overline{\RR^4-E_P}\  \}$, 
$\TT_{G_0}$ is homeomorphic to $S^4\wedge S^2\wedge (D^2_+)$, where $D^2_+$ is the disk  $D^2$ with the disjoint basepoint, and  $S^4$ and $S^2\times\{0\}\subset S^2\wedge (D^2_+)$ correspond to a fiber and $\Delta$,  respectively. \\
\indent We will consider a retract to $S^4\wedge (S^2\times \{0\})\cong Th(\nu_P|_\Delta)$ and show that the composition of $\psi_j$ with this retract is of degree one. Here,  $Th(\nu_P|_\Delta)\subset \TT_{G_0}$ is the subspace  of points represented by points in $\nu_P|_\Delta=\nu_P\cap\pi_P^{-1}(\Delta)$.  Write $\nu=\nu_P|_N$. Let  $\tilde r: \nu\to \nu|_{\Delta}$ be the bundle map which covers the orthogonal projection $r:N\to \Delta$ and  restricts to the parallel transport in $\RR^8$ taking the  center to the center on each fiber. This map induces a map $\tilde r:\TT_{G_0}\to Th(\nu|_{\Delta})$.  We consider $\nu$ and $\nu|_\Delta$ as  subspaces of $\RR^8$. Put 
\[
F_j:=\tilde r\circ \psi_j:(\psi_j)^{-1}(\nu)\to \nu|_{\Delta}\qquad \text{for}\qquad j=1,2.
\]
We shall write down $F_j$ concretely. By definition, $\nu$ is the tubular neighborhood of the map
\[
e_P:(a,b)\mapsto \left(a-\frac{\rho}{2}c_2u,\ a+\frac{\rho}{2}c_1u,\ b-\frac{\rho}{2}c_4u,\ b+\frac{\rho}{2}c_3u\right).
\]
The projection $\pi_P$  sends $(c,d,e,f)\in \RR^8$ to the point $(a,b)$ which minimize the distance $|(c,d,e,f)-e_P(a,b)|$. By elementary calculation, this point is given by
\[
(a,b)=\left(\frac{c+d}{2}+\frac{\rho}{4}(c_2-c_1)u,\ \frac{e+f}{2}+\frac{\rho}{4}(c_4-c_3)u\right).
\]
Similarly, we see that $r:N\to \Delta$ is given by $r(a,b)=(a+b)/2,$ regarding $\Delta\subset \RR^2$ by the first component. 
 Since $\psi_1(x,y,s,t)=((1-t)x+ty+sv, tx+(1-t)y-sv, x-sv, y-sv)$, we have 
\[
r\circ \pi_P\circ \psi_1(x,y,s,t)=\frac{x+y}{2}-\frac{sv}{2}+\frac{\rho}{8}(c_2+c_4-c_1-c_3)u\ (=:w).
\] 
We denote the right hand side by $w$. For simplicity, we move the fiber of $\nu$ over $\pi_P(\psi_1(x,y,s,t))$ by the parallel transport which sends its center to $0$. By this move, $\psi_1(x,y,s,t)$ is sent to 
\[
\psi_1(x,y,s,t)-e_P(\pi_P(\psi_1(x,y,s,t))=(p^1,-p^1,q, -q)
\] 
where 

\begin{equation}
p^1=(t-1/2)(y-x)+sv+\frac{\rho}{4}(c_1+c_2)u,\ \ \ q=\frac{-1}{2}(y-x)+\frac{\rho}{4}(c_3+c_4)u.
\label{EQp^1_ch2}
\end{equation} Similarly, we see that $r\circ\pi_P\circ \psi_2(x,y,s,t)=w$ and $\psi_2(x,y,s,t)$ is sent to $(p^2,-p^2, q, -q)$ where 
\begin{equation}
p^2=(t-1/2)(y-x)-sv+\frac{\rho}{4}(c_1+c_2)u, \label{EQp^2_ch2}
\end{equation} by the similar parallel transport. Thus, $F_j$ is given by $(x,y,s,t)\mapsto (w, p^j,q)$. The fiber of $\pi_P|_\nu$ is a disk of radius $\epsilon_P$. To prove the lemma, it is enough to show  that there exists a unique $(x,y,s,t)\in \RR^4\times[0,\infty)\times I$ such that $\bar w=w$, $\bar q=q$ and ($\bar p=p^1$ or $\bar p=p^2$) for a given point $(\bar w,\bar p,\bar q)$ with $|(\bar p,\bar q)|\leq \epsilon_P$, and  $p^1=p^2$ holds if and only if $s=0$. Suppose that $(\bar p, \bar q)=(p^j,q)$. By the equations (\ref{EQp^1_ch2}) and (\ref{EQp^2_ch2}), for both of $j=1,2$ we have
\begin{equation}
t=\frac{{\bar p}_1-{\bar q}_1+\rho(c_3+c_4-c_1-c_2)/4}{-2{\bar q}_1+\rho(c_3+c_4)/2},
\label{EQt_ch2}
\end{equation}
where ${\bar p}_1$ and ${\bar q}_1$ denote the first coordinates of $\bar p$ and $\bar q$, respectively. The assumptions on $c_r$ in Definition \ref{DPK} and  on $|(\bar p,\bar q)|$ ensure $0<t<1$. By the equations (\ref{EQp^1_ch2}), (\ref{EQp^2_ch2}) and (\ref{EQt_ch2}), the values of $\bar p_1$ and $\bar q$ determine and are determined by the values of $y-x$ and $t$. Since  $p^1$ and $p^2$  are different only in the coefficients of $v$, any value of the second coordinate of $\bar p$ is realized by a unique value of $s$ in the equation (\ref{EQp^1_ch2}) or (\ref{EQp^2_ch2})  and  $p^1$ and $p^2$ take the same value only if $s=0$ (when $q$ is fixed).  Since  $x+y$ still can be set freely, $\bar w$ determines the values of $x,y$.
Thus we have proved the claim, which implies part 2.
\end{proof}
\begin{thm}\label{Tdifferential} In dimension $d=2$ and over a field of characteristic $2$, there exists an element  $g\in \EE^{-4,2}_2$ satisfying $d_2(g)\not=0$. 
\end{thm}
\begin{proof}
 By Lemma \ref{Lspec_seq2}, we see that the class $g'=[\sum_{1\leq i\leq 3}c(G_i)]\in tr_1\CEE^{-4,2}_1$ is lifted to a class $g''\in \bar \EE^{-4,2}_1$. Let $g$  denote the image in $\EE^{-4,2}_1$ of $g''$ by the map given in Lemma \ref{Lspecseq}.  By Lemmas \ref{Lh_i_welldef} and \ref{Lfundcyc}, $d_1g'=0$ and $d_2g'\not=0$.  So we have $d_1g=0$ and $d_2g\not=0$ (see the paragraph below Lemma \ref{Lspec_seq2}).
\end{proof}
\section{Computation of a differential in characteristic $\not=2$}\label{Snon-triviality}
In this section, we prove part 2 of Theorem \ref{Tmain}.  Here, we set $d=2$ and $n=4$ and assume that $\kk$ is a field of characteristic $\not=2$. The module $\EE^{-4,2}_2$ is generated by the class represented by the $d_1$-cycle $g_{13}g_{24}$. Unlike previous case, this monomial is not a $d_1$-cycle in $\CEE$. A $d_1$-cycle in $\CEE$ corresponding to the element  is given by
\begin{equation}
-g_{13}g_{24}+g_{13}g_{14}+g_{14}g_{24}\label{EQ3term}
\end{equation}
since $g_{13}g_{14}$ and $g_{14}g_{24}$ are zero in the normalized $E_1$-page of $\EE$.  We will show that the $d_2$-differential of this element is zero. Actually we will compute the corresponding differential of the projection of the element  to the truncated  sequence $tr_1\CEE$.  The notations in this section are independent of those in previous section. For example, the chain $c(G)$ in this section is different from the chain of the same notation in previous section, see Definition \ref{Dchain_3term}. \\
\indent Throughout this section, $G_i$ $(1\leq i\leq 3)$ denotes one of the following graphs which correspond to the terms of the element (\ref{EQ3term}):
\[
G_1=(1,3)(2,4),\qquad G_2=(1,3)(1,4),\qquad G_3=(1,4)(2,4)
\]
(see Figure \ref{Fgraphs_3term}). The computation is similar to the one in section \ref{Sch2}. The main difference is that we need to deal with the 3-term relation since we use it in the computation showing $d_1(g_{13}g_{24})=0$. To make the computation easier, we modify the definition of chains.

\begin{defi}\label{D(e,-)-cont}
Let $f:X\to \RR^8$ be a $G$-condensed map and $e=(\alpha,\beta)$ be a bridge of $G$.
In this section, we call the $e$-contraction given in Definition \ref{Dcondensed2} the {\em $(e,+)$-contraction}. The {\em $(e,-)$-contraction} of $f$ for $G$ is the version of $e$-contraction whose contracting direction is reversed. We add $-sv$ (resp. $sv$) to the $i$-th component if $i\sim_{\partial_eG} \alpha$ (resp. $i\sim_{\partial_eG}\beta$).  
\end{defi} 

\begin{exa}\label{E(e,-)-cont}
Put $f=f_{G_1}, e=(1,3)$. The $(e,-)$-contraction of $f$ (for $G_1$) is given by  
$(x,y,s)\mapsto (x-sv,y,x+sv,y)$.
\end{exa}
\begin{figure}
\begin{center}
{\unitlength 0.1in%
\begin{picture}(44.0000,3.3500)(0.1000,-3.4500)%
%
\special{pn 8}%
\special{pa 110 210}%
\special{pa 136 181}%
\special{pa 161 154}%
\special{pa 187 128}%
\special{pa 212 105}%
\special{pa 238 86}%
\special{pa 264 71}%
\special{pa 289 62}%
\special{pa 315 60}%
\special{pa 340 65}%
\special{pa 366 76}%
\special{pa 392 92}%
\special{pa 417 113}%
\special{pa 443 137}%
\special{pa 468 164}%
\special{pa 494 192}%
\special{pa 510 210}%
\special{fp}%
%
\special{pn 8}%
\special{pa 310 210}%
\special{pa 336 182}%
\special{pa 362 155}%
\special{pa 388 130}%
\special{pa 414 107}%
\special{pa 440 89}%
\special{pa 465 75}%
\special{pa 491 67}%
\special{pa 517 65}%
\special{pa 543 71}%
\special{pa 569 82}%
\special{pa 595 99}%
\special{pa 621 120}%
\special{pa 647 144}%
\special{pa 673 170}%
\special{pa 699 198}%
\special{pa 710 210}%
\special{fp}%
%
\special{pn 8}%
\special{pa 1210 210}%
\special{pa 2010 210}%
\special{fp}%
%
\special{pn 8}%
\special{pa 10 210}%
\special{pa 810 210}%
\special{fp}%
%
\special{pn 8}%
\special{pa 1310 210}%
\special{pa 1363 157}%
\special{pa 1390 132}%
\special{pa 1417 108}%
\special{pa 1443 86}%
\special{pa 1470 65}%
\special{pa 1496 48}%
\special{pa 1523 33}%
\special{pa 1550 21}%
\special{pa 1576 14}%
\special{pa 1603 10}%
\special{pa 1630 11}%
\special{pa 1656 17}%
\special{pa 1683 26}%
\special{pa 1709 39}%
\special{pa 1736 56}%
\special{pa 1763 74}%
\special{pa 1789 96}%
\special{pa 1816 119}%
\special{pa 1843 144}%
\special{pa 1869 169}%
\special{pa 1910 210}%
\special{fp}%
%
\special{pn 8}%
\special{pa 1310 210}%
\special{pa 1368 170}%
\special{pa 1397 152}%
\special{pa 1426 137}%
\special{pa 1455 125}%
\special{pa 1483 117}%
\special{pa 1512 115}%
\special{pa 1541 118}%
\special{pa 1570 127}%
\special{pa 1599 139}%
\special{pa 1628 155}%
\special{pa 1657 173}%
\special{pa 1686 193}%
\special{pa 1710 210}%
\special{fp}%
%
\special{pn 8}%
\special{pa 3215 210}%
\special{pa 2415 210}%
\special{fp}%
%
\special{pn 8}%
\special{pa 3115 210}%
\special{pa 3062 157}%
\special{pa 3035 132}%
\special{pa 3008 108}%
\special{pa 2982 86}%
\special{pa 2955 65}%
\special{pa 2929 48}%
\special{pa 2902 33}%
\special{pa 2875 21}%
\special{pa 2849 14}%
\special{pa 2822 10}%
\special{pa 2795 11}%
\special{pa 2769 17}%
\special{pa 2742 26}%
\special{pa 2716 39}%
\special{pa 2689 56}%
\special{pa 2662 74}%
\special{pa 2636 96}%
\special{pa 2609 119}%
\special{pa 2582 144}%
\special{pa 2556 169}%
\special{pa 2515 210}%
\special{fp}%
%
\special{pn 8}%
\special{pa 3115 210}%
\special{pa 3057 170}%
\special{pa 3028 152}%
\special{pa 2999 137}%
\special{pa 2970 125}%
\special{pa 2942 117}%
\special{pa 2913 115}%
\special{pa 2884 118}%
\special{pa 2855 127}%
\special{pa 2826 139}%
\special{pa 2797 155}%
\special{pa 2768 173}%
\special{pa 2739 193}%
\special{pa 2715 210}%
\special{fp}%
%
\special{pn 8}%
\special{pa 3610 210}%
\special{pa 4410 210}%
\special{fp}%
%
\special{pn 8}%
\special{pa 3710 210}%
\special{pa 3763 157}%
\special{pa 3790 132}%
\special{pa 3817 108}%
\special{pa 3843 86}%
\special{pa 3870 65}%
\special{pa 3896 48}%
\special{pa 3923 33}%
\special{pa 3950 21}%
\special{pa 3976 14}%
\special{pa 4003 10}%
\special{pa 4030 11}%
\special{pa 4056 17}%
\special{pa 4083 26}%
\special{pa 4109 39}%
\special{pa 4136 56}%
\special{pa 4163 74}%
\special{pa 4189 96}%
\special{pa 4216 119}%
\special{pa 4243 144}%
\special{pa 4269 169}%
\special{pa 4310 210}%
\special{fp}%
%
\special{pn 4}%
\special{sh 1}%
\special{ar 1510 210 16 16 0 6.2831853}%
\special{sh 1}%
\special{ar 1510 210 16 16 0 6.2831853}%
%
\special{pn 4}%
\special{sh 1}%
\special{ar 2915 210 16 16 0 6.2831853}%
\special{sh 1}%
\special{ar 2915 210 16 16 0 6.2831853}%
%
\special{pn 4}%
\special{sh 1}%
\special{ar 3915 210 16 16 0 6.2831853}%
\special{sh 1}%
\special{ar 4115 210 16 16 0 6.2831853}%
\special{sh 1}%
\special{ar 4115 210 16 16 0 6.2831853}%
%
\special{pn 8}%
\special{pa 3890 180}%
\special{pa 4135 180}%
\special{pa 4135 240}%
\special{pa 3890 240}%
\special{pa 3890 180}%
\special{pa 4135 180}%
\special{fp}%
%
\special{pn 8}%
\special{pa 3715 210}%
\special{pa 3741 185}%
\special{pa 3768 162}%
\special{pa 3795 142}%
\special{pa 3822 127}%
\special{pa 3850 117}%
\special{pa 3878 115}%
\special{pa 3907 121}%
\special{pa 3936 132}%
\special{pa 3966 148}%
\special{pa 3995 167}%
\special{pa 4015 180}%
\special{fp}%
%
\special{pn 8}%
\special{pa 4315 210}%
\special{pa 4289 185}%
\special{pa 4262 162}%
\special{pa 4235 142}%
\special{pa 4208 127}%
\special{pa 4180 117}%
\special{pa 4152 115}%
\special{pa 4123 121}%
\special{pa 4094 132}%
\special{pa 4064 148}%
\special{pa 4035 167}%
\special{pa 4015 180}%
\special{fp}%
\put(4.1000,-4.1000){\makebox(0,0){$G_1$}}%
\put(16.1000,-4.1000){\makebox(0,0){$G_2$}}%
\put(28.1000,-4.1000){\makebox(0,0){$G_3$}}%
\put(40.1000,-4.1000){\makebox(0,0){$G_4$}}%
\end{picture}}%
\end{center}
\caption{graphs in section \ref{Snon-triviality}: The dots in $G_2$ and $G_3$ denote the discrete vertices and the rectangle including two dots in $G_4$ does the unified vertices. } \label{Fgraphs_3term}
\end{figure}
\begin{defi}\label{Dchain_3term}
For $G=G_1,G_2,$ and $G_3$, set $f=f_G$ and $\{e_1<e_2\}=E(G)$.  Let $f^{\eps}_j$ be the $(e_j,\eps)$-contraction of $f$ for $\eps=\pm$. Set
\[
c(G)=f(w_0)-f^{\pm}_1(w_1)+f^\pm_2(w_1)\quad \in tr_1\tot_{10}(\TTH).
\]
Here, $f^{\pm}_j(w)$ denotes the average $(f^+_j(w_1)+f^-_j(w_1))/2$ as in subsection \ref{SSNT}. We consider the chain
\[
C=-c(G_1)+c(G_2)+c(G_3).
\]
\end{defi}
We set $D=d+(-1)^*\partial$ where $*$ is the singular degree. We have $D(c(G_i))=0$ as in Lemma \ref{L1st_cyc_ch2} for $1\leq i\leq 3$. We shall construct a bounding chain of $\delta C$. 
\begin{defi}\label{Dchain_2graphs_3term}
Let $(G, H,i)$ be either of $
(G_1,G_2,1)$ or  $(G_1,G_3,3)$.
Set $f=f_G$, $f'=f_H$, and $E(G)=\{e_1<e_2\}$. We identify  the edge sets through the composition of the standard bijections $E(G)\cong E(\delta_iG)= E(\delta_iH)\cong E(H)$. (In this case, the bijection preserves the order of edges.) 
 If the $i$-th components of $f$ and $f'$ are identical, $\psi$ denotes the straight homotopy from $f$ to $f'$. Otherwise, $\psi$ is the straight homotopy from $f$ to $f'\circ T$, where $T:\RR^4\to \RR^4$ is the transposition $T(x,y)=(y,x)$. For $\eps=\pm$, let
\begin{enumerate}
\item $\psi_j^\eps$  be the $(e_j,\eps)$-contraction of $\psi$ for $\delta_iG$,
\item $\lambda_j^\eps$  the straight homotopy from the $(e_j,\eps)$-contraction  of $f$ for $G$ to the $(e_j,\eps)$-contraction  of $f$ for $\delta_iG$, and
\item ${\lambda'_j}^\eps$ the straight homotopy from the $(e_j,\eps)$-contraction of $f'$ for $H$ to the $(e_j,\eps)$-contraction of $f'$ for $\delta_iH$.
\end{enumerate}
Then, we set 
\[
c(G,H,i):=\psi(w_{01})+\sum_{ j=1,2}(-1)^{j+1}(\psi^{\pm}_j+\lambda^{\pm}_j-{\lambda'_j}^\pm)(w_{11}).
\]Similarly to Definition \ref{Dchainhomotopy}, the maps $\psi^\eps_j$'s are implicitly composed with the  transposition of $(I_+)\wedge [0,\infty]$.
\end{defi}
We have 
\[
Dc(G_1,G_2,1)=\delta_1c(G_2)-\delta_1c(G_1), \qquad Dc(G_1,G_3,3)=\delta_3c(G_3)-\delta_3c(G_1)
\]
as in Lemma \ref{Lh_i_welldef}.
We shall construct a bounding chain of  
$\delta_2C$. 
\begin{defi}\label{Dbound_3term}
Put $f_i=f_{G_i}$ for $1\leq i\leq 3$. Below, we denote the two edges of all the involved graphs  by the same notation $e_1<e_2$. Since $\delta_k$ which appears here does not permute edges, this does not cause confusion.  Let $\lambda_{ij}^\eps$ be the straight homotopy from the $(e_j,\eps)$-contraction of $f_i$ for $G_i$ to the $(e_j,\eps)$-contraction of $f_i$ for $\delta_2G_i$, 
$\psi$ the straight homotopy  from $f_1$ to  $f_2$, 
$\phi$ the straight homotopy  from  $f_1$ to  $f_3$, $\psi_j^{\eps}$ the $(e_j,\eps)$-contraction of $\psi$ for $\delta_2G_2$, $\phi_j^\eps$ the $(e_j,\eps)$-contraction of $\phi$ for $\delta_2G_3$. Let $G_4\in\GG(\delta_2[5])$ be the graph with $E(G_4)=\{(1,\{23\}),(1,4),(\{23\},4)\}$, see Figure \ref{Fgraphs_3term}. We set
\[
c(G_1,G_2,G_3,2)=f_1(w_0)+(\psi+\phi)(w_{01})+\sum_{j=1,2}(-1)^{j+1}(\psi^\pm_j+\phi^\pm_j+\lambda^\pm_{1j}-\lambda^\pm_{2j}-\lambda^\pm_{3j})(w_{11}),
\]
where the pushforward by $f_1$, (resp. $\psi$, $\phi$, $\psi^\eps_j$, $\phi^\eps_j$) belongs to $\bar C_*(\TT_{G_4})$, (resp. $\bar C_*(\TT_{\delta_2G_2})$, $\bar C_*(\TT_{\delta_2G_3})$, $\bar C_*(\TT_{\delta_2\partial_jG_2})$, $\bar C_*(\TT_{\delta_2\partial_jG_3})$). The pushforwad by $\lambda_{ij}^\pm$ belongs to $\bar C_*(\TT_{\delta_2\partial_jG_i})$. The maps $\psi^\eps_j,\ \phi^\eps_j$ are implicitly composed with the  transposition on the parameter space. This is well-defined as $f_1$ is also $\delta_2G_2$- and $\delta_2G_3$-condensed. 
\end{defi}
\noindent{\bf Terminology}:\ In the rest of the paper, if a chain  $C$ contains a  term given as the pushforward by some map $g$, we sometimes call the term simply {\em the term of $g$ in $C$}.
\begin{lem}\label{L3term}
We have
$Dc(G_1,G_2,G_3,2)=\delta_2C$. 
If we set $\Gamma=c(G_1,G_2,1)+c(G_1,G_3,3)-c(G_1,G_2,G_3,2)$, we have $D\Gamma=-\delta C$.
\end{lem}
\begin{proof}
In this proof, we use the notations in Definition \ref{Dbound_3term}. We shall show the equation for $Dc(G_1,G_2,G_3,2)$ in the claim. 
 We have 
\[
\begin{split}
Df_1(w_0)&=f_1(w_0) -f_1(w_0) +f_1(w_0)\ \in \bar C_*(\TT_{\delta_2G_3})\oplus\bar C_*(\TT_{\delta_2G_1})\oplus\bar C_*(\TT_{\delta_2G_2}), \\
D\psi(w_{01})&=(-f_1(w_0)+f_2(w_0))-\psi(w_{01})+\psi(w_{01})\ \in \bar C_*(\TT_{\delta_2G_2})\oplus\bar C_*(\TT_{\delta_2\partial _1G_2})\oplus\bar C_*(\TT_{\delta_2\partial_2 G_2}), \\
D\phi(w_{01})&=(-f_1(w_0)+f_3(w_0))-\phi(w_{01})+\phi(w_{01})\ \in \bar C_*(\TT_{\delta_2G_3})\oplus\bar C_*(\TT_{\delta_2\partial _1G_3})\oplus\bar C_*(\TT_{\delta_2\partial_2 G_3}), \\
D\psi^\eps_j(w_{11})&=\psi (w_{11})+\psi^\eps_j|_{t=0}(w_1)-\psi^\eps_j|_{t=1}(w_1) \ \in\bar C_*(\TT_{\delta_2\partial_jG_2}), \\
D\phi^\eps_j(w_{11})&=\phi (w_{11})+\phi^\eps_j|_{t=0}(w_1)-\phi^\eps_j|_{t=1}(w_1)\ \in\bar C_*(\TT_{\delta_2\partial_jG_3}), \\
\end{split}
\]
where $t\in I$ is the variable for straight homotopy. 
Two of the three copies of $f_1(w_0)$ in $Df_1(w_0)$ cancel with the terms in $D\psi(w_{01})$ and $D\phi(w_{01})$, and the other one cancels with the term  $f_1(w_0)$ in $\delta_2C$.  The maps $\psi_j^+$ and $\phi_j^+$ are expressed as follows
\begin{center}
\input{graphs3termcont.TEX}\vspace{2mm}
\end{center}
and $\psi_j^-$ and $\phi_j^-$ are given by reversing the signs. Here, $\underline{x}$ (resp. $\underline{y}$) represents the homotopy from $x$ to $y$ (resp. $y$ to $x$), and the signs are those of the terms $sv$ added to the components.
As in this picture, the terms of $\psi^\eps_1$ and $\phi^\eps_2$ belong to the same graph,  and we see that $\psi^\eps_1|_{t=0}=\phi^{\eps'}_2|_{t=0}$ for $\eps\not=\eps'$. (This twist of sign is the reason why we introduce the $(e,-)$-contraction.) So the terms of  $\psi^\pm_1|_{t=0}$ and $\phi^\mp_2|_{t=0}$ cancel with each other. We denote by $f_{ij}^\eps$ the $(e_j,\eps)$-contraction of $f_i$ for $G_i$ ($i=1,2,3$). The terms of  $\psi^\eps_1|_{t=1}$ and $\phi^\eps_2|_{t=1}$ cancel with the terms of $f^\eps_{21}$ and $f^\eps_{32}$ in $\delta_2C$ (through $\lambda^\eps_{21}$ and $\lambda^\eps_{32}$) respectively. The terms of $\psi^\eps_2|_{t=0}$ and $\phi^\eps_1|_{t=0}$ cancel with  $f^\eps_{12}$ and $f^\eps_{11}$ respectively. The terms of $\psi^\eps_2|_{t=1}$ and $\phi^\eps_1|_{t=1}$ cancel with $f^\eps_{22}$ and $f^\eps_{31}$. respectively. We easily see that the remaining terms cancel with one another. The equation for $D\Gamma$ easily follows from the equation for $Dc(G_1,G_2,G_3,2)$ and the equations for $Dc(G,H,i)$ before Definition \ref{Dbound_3term}.
\end{proof}

\begin{prop}\label{Pd2_ch3}
Suppose  the characteristic of the field $\kk$ is not two. Let $\Gamma$ be the chain defined in Lemma \ref{L3term}. $\delta\Gamma$ is  null-homologous in the chain complex $(tr_1\tot\mathbb{T}, D)=(tr_1\CEE_0,d_0)$. In particular, $d_2(g_{13}g_{24})=0$ in $\EE_2$. 
\end{prop}
\begin{proof}
Among the terms of $\delta\Gamma$,  we see that only  $\delta_1(c(G_1,G_3,3))$ contains non-zero terms by Lemmas \ref{Lcondensed_collapse0} and \ref{Lcondensed_collapse} as in the proof of Lemmas \ref{Lh_i_welldef} and \ref{Lfundcyc}. The terms of  $(e,\eps)$-contractions  forms a cycle for each of $\eps=+,-$. We claim that the  classes of these cycles  cancel with each other. The proof  is similar to that of Lemma \ref{Lfundcyc}. Let $\nu, \Delta, r$ and $\tilde r$ denote the spaces and maps given in the proof, and $\psi^\eps_j$ denote the map given in Definition \ref{Dchain_2graphs_3term} for $(G_1,G_3,3)$. Put
\[
F^\eps_j=\tilde r\circ \psi_j^\eps: (\psi_j^\eps)^{-1}\nu\to \nu|_\Delta
\]
for $j=1,2$ and $\eps=\pm$. By computation similar to the proof, we have
\[
r\circ\pi_P\circ \psi^\eps_j(x,y,s,t)=\frac{1}{4}((2-t)x+(2+t)y)-\frac{1}{2}\eps sv+\frac{\rho}{8}(c_2+c_4-c_1-c_3)u\quad  (=:w^\eps),
\]
where $\eps$ in the right hand side denotes $\pm 1$ for $\eps=\pm$ respectively. The fiberwise parallel transport taking each center to $0$ sends $F_j^\eps(x,y,s,t)$ to $(p_j^\eps,-p_j^\eps,q,-q)$ where
\[
\begin{split}
p_j^\eps&=\frac{1}{2}(1-t)(x-y)+(-1)^{j+1}\eps sv+\frac{\rho}{4}(c_1+c_2)u, \\
q&=\frac{1}{2}(x-y)+\frac{\rho}{4}(c_3+c_4)u.
\end{split}
\] 
We define a homotopy $H_j$ from $F_j^+$ to $F_{3-j}^-$ by  $(x,y,s,t,\tau)\mapsto ((1-\tau)w^++\tau w^-,p_j^+,q)$ with the transported coordinate. Since $p^+_j=p^-_{3-j}$ and the value of  $(p_j^+,q)$ bounds the values of $|x-y|$ and $s$,  $H_j$ induces a well-defined homotopy between the pointed spaces. Furthermore, 
since $H_j$ sends the elements with $t=0,1$ to the outside of $\nu|_\Delta$ and $H_j|_{s=0}$ is independent of $\tau$, a signed sum of $H_1$ and $H_2$ gives a bounding chain 
of 
\[
\sum_{j=1,2}(-1)^{j+1}(\psi^+_j-\psi^-_{3-j})(w_{11})=\sum_{j=1,2}(-1)^{j+1}(\psi^+_j+\psi^-_{j})(w_{11})=2\delta_1c(G_1,G_3,3).
\] The latter statement easily follows from the former  and Lemmas \ref{Lspecseq} and \ref{Lspec_seq2}.
\end{proof}
\begin{rem}\label{R3term_ch2}
The monomial $g_{13}g_{24}$ is also a $d_1$-cycle in  characteristic 2 and Proposition \ref{Pd2_ch3} implies  $d_2(g_{13}g_{24})=0$ even in this case since the elements in  source and target come from non-torsion elements. We can also prove this  by modifying the definitions of chains so as not to include $1/2$. 
\end{rem}

\section{Computation of a differential in characteristic 3}\label{Sch3}
In this section, we prove part 3 of Theorem \ref{Tmain}. 
Throughout this section, we set $n=5$ and $d=2$ and assume that $\kk$ is a field of characteristic $3$. 
 By a straightforward computation, we see that the element 
\begin{equation}
-g_{13}g_{23}g_{45}+g_{14}g_{24}g_{35}+g_{14}g_{25}g_{34}+g_{15}g_{24}g_{34}\label{EQch3}
\end{equation} in $\EE^{-5,3}_1$ is a cycle for the $d_1$-differential  in characteristic $3$. We  see that the  element in $\CEE^{-5,3}_1$ given by the same formula (\ref{EQch3}) is also a $d_1$-cycle (see the  paragraphs after Lemma \ref{Lspec_seq2}). We will show that the $d_2$-differential of this element is zero and the $d_3$-differential is non-zero. Actually we compute the corresponding differentials of the projection of the element  to the truncated  sequence $tr_1\CEE$.   The computation is similar to the one in section \ref{Sch2}. The notations in this section are independent of those in previous two sections. For example, the chain $c(G)$ in this section is different from the chain of the same notation in previous two sections.  We define four graphs in $\GG([6])$ as follows:
\[
G_1=(1,3)(2,3)(4,5),\quad G_2=(1,4)(2,4)(3,5),\quad G_3=(1,4)(2,5)(3,4),\quad G_4=(1,5)(2,4)(3,4),
\]
see Figure \ref{Fgraphs_ch3}.
Throughout this section, $G_i$ ($1\leq i\leq 4$) denotes one of these graphs.
\begin{figure}
\begin{center}
{\unitlength 0.1in%
\begin{picture}(46.0000,3.7500)(0.3000,-4.0500)%
%
\special{pn 8}%
\special{pa 1235 306}%
\special{pa 2235 306}%
\special{fp}%
%
\special{pn 8}%
\special{pa 30 306}%
\special{pa 1030 306}%
\special{fp}%
%
\special{pn 8}%
\special{pa 130 306}%
\special{pa 155 276}%
\special{pa 180 247}%
\special{pa 205 220}%
\special{pa 230 196}%
\special{pa 255 175}%
\special{pa 280 159}%
\special{pa 305 149}%
\special{pa 330 146}%
\special{pa 355 150}%
\special{pa 380 160}%
\special{pa 404 177}%
\special{pa 429 198}%
\special{pa 454 222}%
\special{pa 478 250}%
\special{pa 503 279}%
\special{pa 525 306}%
\special{fp}%
%
\special{pn 8}%
\special{pa 730 306}%
\special{pa 755 277}%
\special{pa 780 251}%
\special{pa 805 233}%
\special{pa 830 226}%
\special{pa 855 233}%
\special{pa 880 251}%
\special{pa 905 277}%
\special{pa 930 306}%
\special{fp}%
%
\special{pn 8}%
\special{pa 2435 306}%
\special{pa 3435 306}%
\special{fp}%
%
\special{pn 8}%
\special{pa 3630 306}%
\special{pa 4630 306}%
\special{fp}%
%
\special{pn 8}%
\special{pa 1335 306}%
\special{pa 1361 278}%
\special{pa 1387 251}%
\special{pa 1413 225}%
\special{pa 1440 200}%
\special{pa 1466 176}%
\special{pa 1492 155}%
\special{pa 1518 136}%
\special{pa 1544 120}%
\special{pa 1570 108}%
\special{pa 1596 99}%
\special{pa 1622 95}%
\special{pa 1649 95}%
\special{pa 1675 99}%
\special{pa 1701 108}%
\special{pa 1727 121}%
\special{pa 1753 137}%
\special{pa 1779 156}%
\special{pa 1805 177}%
\special{pa 1832 201}%
\special{pa 1858 226}%
\special{pa 1884 252}%
\special{pa 1910 279}%
\special{pa 1935 306}%
\special{fp}%
%
\special{pn 8}%
\special{pa 1535 306}%
\special{pa 1563 284}%
\special{pa 1592 263}%
\special{pa 1620 243}%
\special{pa 1649 227}%
\special{pa 1677 214}%
\special{pa 1705 205}%
\special{pa 1734 202}%
\special{pa 1762 205}%
\special{pa 1790 214}%
\special{pa 1819 227}%
\special{pa 1847 244}%
\special{pa 1875 263}%
\special{pa 1903 285}%
\special{pa 1930 306}%
\special{fp}%
%
\special{pn 8}%
\special{pa 1735 306}%
\special{pa 1762 280}%
\special{pa 1788 255}%
\special{pa 1815 231}%
\special{pa 1842 211}%
\special{pa 1868 194}%
\special{pa 1895 182}%
\special{pa 1922 175}%
\special{pa 1949 175}%
\special{pa 1976 181}%
\special{pa 2002 193}%
\special{pa 2029 209}%
\special{pa 2056 229}%
\special{pa 2083 252}%
\special{pa 2110 277}%
\special{pa 2137 303}%
\special{pa 2140 306}%
\special{fp}%
%
\special{pn 8}%
\special{pa 2535 306}%
\special{pa 2561 278}%
\special{pa 2587 251}%
\special{pa 2613 225}%
\special{pa 2640 200}%
\special{pa 2666 176}%
\special{pa 2692 155}%
\special{pa 2718 136}%
\special{pa 2744 120}%
\special{pa 2770 108}%
\special{pa 2796 99}%
\special{pa 2822 95}%
\special{pa 2849 95}%
\special{pa 2875 99}%
\special{pa 2901 108}%
\special{pa 2927 121}%
\special{pa 2953 137}%
\special{pa 2979 156}%
\special{pa 3005 177}%
\special{pa 3032 201}%
\special{pa 3058 226}%
\special{pa 3084 252}%
\special{pa 3110 279}%
\special{pa 3135 306}%
\special{fp}%
%
\special{pn 8}%
\special{pa 2735 306}%
\special{pa 2761 279}%
\special{pa 2788 252}%
\special{pa 2814 225}%
\special{pa 2840 200}%
\special{pa 2867 177}%
\special{pa 2893 156}%
\special{pa 2919 137}%
\special{pa 2945 121}%
\special{pa 2972 108}%
\special{pa 2998 100}%
\special{pa 3024 95}%
\special{pa 3050 94}%
\special{pa 3076 99}%
\special{pa 3103 107}%
\special{pa 3129 119}%
\special{pa 3155 135}%
\special{pa 3181 154}%
\special{pa 3207 175}%
\special{pa 3233 198}%
\special{pa 3259 223}%
\special{pa 3286 249}%
\special{pa 3312 276}%
\special{pa 3338 304}%
\special{pa 3340 306}%
\special{fp}%
%
\special{pn 8}%
\special{pa 3730 306}%
\special{pa 3756 279}%
\special{pa 3783 252}%
\special{pa 3809 225}%
\special{pa 3836 200}%
\special{pa 3862 175}%
\special{pa 3889 151}%
\special{pa 3915 129}%
\special{pa 3942 109}%
\special{pa 3968 90}%
\special{pa 3995 73}%
\special{pa 4021 59}%
\special{pa 4048 47}%
\special{pa 4074 38}%
\special{pa 4100 33}%
\special{pa 4127 30}%
\special{pa 4153 31}%
\special{pa 4179 35}%
\special{pa 4206 43}%
\special{pa 4232 54}%
\special{pa 4258 67}%
\special{pa 4284 82}%
\special{pa 4311 100}%
\special{pa 4337 120}%
\special{pa 4363 142}%
\special{pa 4389 165}%
\special{pa 4415 190}%
\special{pa 4442 215}%
\special{pa 4520 296}%
\special{pa 4530 306}%
\special{fp}%
\put(5.3000,-4.7000){\makebox(0,0){$G_1$}}%
\put(17.3000,-4.7000){\makebox(0,0){$G_2$}}%
\put(29.3000,-4.6600){\makebox(0,0){$G_3$}}%
\put(41.3000,-4.6600){\makebox(0,0){$G_4$}}%
%
\special{pn 8}%
\special{pa 3935 306}%
\special{pa 3959 275}%
\special{pa 3984 246}%
\special{pa 4008 218}%
\special{pa 4033 193}%
\special{pa 4057 172}%
\special{pa 4082 156}%
\special{pa 4106 146}%
\special{pa 4131 142}%
\special{pa 4156 146}%
\special{pa 4180 156}%
\special{pa 4205 172}%
\special{pa 4230 193}%
\special{pa 4254 217}%
\special{pa 4279 245}%
\special{pa 4304 274}%
\special{pa 4329 305}%
\special{pa 4330 306}%
\special{fp}%
%
\special{pn 8}%
\special{pa 4130 306}%
\special{pa 4158 284}%
\special{pa 4187 266}%
\special{pa 4215 256}%
\special{pa 4244 255}%
\special{pa 4272 266}%
\special{pa 4300 284}%
\special{pa 4329 305}%
\special{pa 4330 306}%
\special{fp}%
%
\special{pn 8}%
\special{pa 2930 306}%
\special{pa 2958 284}%
\special{pa 2987 266}%
\special{pa 3015 256}%
\special{pa 3044 255}%
\special{pa 3072 266}%
\special{pa 3100 284}%
\special{pa 3129 305}%
\special{pa 3130 306}%
\special{fp}%
%
\special{pn 8}%
\special{pa 325 306}%
\special{pa 353 284}%
\special{pa 382 266}%
\special{pa 410 256}%
\special{pa 439 255}%
\special{pa 467 266}%
\special{pa 495 284}%
\special{pa 524 305}%
\special{pa 525 306}%
\special{fp}%
\end{picture}}%
\end{center}
\caption{graphs in section \ref{Sch3}}\label{Fgraphs_ch3}
\end{figure}
\begin{defi}\label{Dchain_ch3}
Let $G$ be one of $G_1,\dots,G_4$ . Put $E(G)=\{e_1<e_2<e_3\}$ and $f=f_G:\RR^4\to \RR^{10}$, see Definition \ref{Dcondensed}. Let $f_k$ be the $e_k$-contraction of $f$ and $f_{kl}$ the $(e_k,e_l)$-contraction of $f$ (both for $G$). We define a chain $c(G)$ by
\[
\begin{split}
c(G)=f(w_0)+\sum_{k=1}^3(-1)^kf_k(w_1) &+\sum_{1\leq k<l\leq 3}(-1)^{k+l+1}f_{kl}(w_2). \\
&\in\quad \bar C_4(\TT_G)\oplus \bigoplus_{1\leq k\leq 3}\bar C_5(\TT_{\partial_kG})\oplus\bigoplus_{1\leq k<l\leq 3}\bar C_6(\TT_{\partial_{kl}G})\quad \subset tr_1\tot_{12}(\TTH).
\end{split}
\]

\end{defi}
Set $D=d+(-1)^*\partial$, where $*$ is the singular degree. 
\begin{lem}\label{Lcycle_ch3}
$c(G)$ is a cycle in $(tr_1\tot\TTH, D)=(tr_1\CEE_0,d_0)$.
\end{lem}
\begin{proof}
This is similar to Lemma \ref{L1st_cyc_ch2}. For the new terms $f_{kl}(w_2)$ with $k<l$, we have
\[
df_{kl}(w_2)=f_{kl}|_{s_1=0}(w_1)-f_{kl}|_{s_2=0}(w_1)=f_l(w_1)-f_k(w_1),
\]
where $s_1$ and $s_2$ are the variables for the contractions in removing the $k$-th and $l$-th edges respectively. These two terms cancel with $\partial_kf_l(w_1)$ and $\partial_{l-1}f_k(w_1)$ since by definition $D=d+(-1)^5\partial$ on $f_k(w_1)$ and $f_l(w_1)$, and $\partial=\sum_i(-1)^{i-1}\partial_i$.
\end{proof}
\subsection{First bounding chain}\label{SS1st_ch3}
\begin{defi}\label{Dchain_2graphs}
Let $(G, H,i)$ be one of $
(G_1,G_2,3), (G_2,G_3,2),(G_2,G_3,4),(G_3,G_4,1),$ and $(G_4,G_3,4)$.
For these triples, we see $\delta_iG=\delta_iH$. 
Set $f=f_G$, $f'=f_H$, and $E(G)=\{e_1<e_2<e_3\}$. We identify  the edge sets through the standard bijections $E(G)\cong E(\delta_iG) = E(\delta_iH)\cong E(H)$. (By this convention, $e_k$ does not necessarily represent the $k$-th edge of $E(H)$  so may not  agree with the notation in Definition \ref{Dchain_ch3} for $H$.) 
 If the $i$-th components of $f$ and $f'$ are identical, $\psi$ denotes the straight homotopy from $f$ to $f'$. Otherwise, $\psi$ is the straight homotopy from $f$ to $f'\circ T$, where $T:\RR^4\to \RR^4$ is the transposition $T(x,y)=(y,x)$. Let
\begin{enumerate}
\item $\psi_k$ (resp. $\psi_{kl}$) be the $e_k$-contraction (resp. $(e_k,e_l)$-contraction) of $\psi$ for $\delta_iG$,
\item $\lambda_k$ (resp. $\lambda_{kl}$)  the straight homotopy from the $e_k$-contraction (resp. $(e_k,e_l)$-contraction) of $f$ for $G$ to the $e_k$-contraction (resp. $(e_k,e_l)$-contraction) of $f$ for $\delta_iG$, and
\item $\lambda'_k$ (resp. $\lambda'_{kl}$) the straight homotopy from the $e_k$-contraction  (resp. $(e_k,e_l)$-contraction) of $f'$ for $H$ to the $e_k$-contraction (resp. $(e_k,e_l)$-contraction) of $f'$ for $\delta_iH$.
\end{enumerate}
Then, we set 
\[
c(G,H,i):=\psi(w_{01})+\sum_{1\leq k\leq 3}(-1)^{k+1}(\psi_k+\lambda_k-\lambda'_k)(w_{11})+\sum_{1\leq k<l\leq 3}(-1)^{k+l+1}(\psi_{kl}+\lambda_{kl}-\lambda'_{kl})(w_{21}).
\]
Here, we compose $\psi_k$  with the transposition of $[0,\infty]\wedge (I_+)$ implicitly since the definition of $\psi_k$ puts $[0,\infty]$ at the rightmost component. $\psi_{kl}$ is also composed with the transposition $ [0,\infty]^{\wedge 2}\wedge (I_+)\cong (I_+)\wedge[0,\infty]^{\wedge 2} ,\ (s_1,s_2,t)\mapsto (t,s_1,s_2)$. 
\end{defi}

We use the terminology given after Definition \ref{Dbound_3term} in the rest of this section.
\begin{lem}\label{Lbound_ch3}
We have 
\[
\begin{split}
Dc(G_1,G_2,3)& =-\delta_3c(G_1)+\delta_3c(G_2), \\
Dc(G_2,G_3,2)&=-\delta_2c(G_2)-\delta_2c(G_3), \\
Dc(G_2,G_3,4)&=-\delta_4c(G_2)+\delta_4c(G_3), \\
Dc(G_3,G_4,1)&=-\delta_1c(G_3)-\delta_1c(G_4), \\
Dc(G_4,G_3,4)&=-\delta_4c(G_4)+\delta_4c(G_3).
\end{split}
\]

\end{lem}
\begin{proof}
The proof is similar to that of Lemma \ref{Lh_i_welldef} with some care about signs. We use the notations of Definition \ref{Dchain_2graphs}. For the triples $(G,H,i)$ in this definition,  the standard bijection $E(G)\cong E(\delta_iG)$ preserves the order of edges so the restriction of each map to $0\in [0,\infty)$ cancels with $\check{\text{C}}$ech differential of another term. Let $f_k$ and $f'_k$(resp.  $f_{kl}$ and $f'_{kl}$) denotes the $e_k$-(resp. $(e_k,e_l)$-) contractions of $f$ and $f'$ for $G$ and $H$, respectively. The concatenation of $\psi_k$, $\lambda_k$, and $\lambda'_k$ defines a homotopy from $f_k$ to $f'_k$ possibly composed with the transposition. We have a similar homotopy for $\psi_{kl}, \lambda_{kl}$, and $\lambda'_{kl}$. We also see that the terms of $\lambda_{kl}|_{s=0}$ and $\lambda_{kl}|_{s=0}$ are zero as in Lemma \ref{Lh_i_welldef}, where $s\in [0,\infty)$ is the variable for the contraction in removing an edge. For $(G,H,i)=(G_1,G_2,3)$, the bijection $E(G)\cong E(H)$ in the definition  preserves the order of edges.  Straightforward computation shows the first equation. The proofs of the  third and fifth equations are similar. For the second equation, the case of $(G,H,i)=(G_2,G_3,2)$, the bijection swaps the second and third edges i.e. $e_2$ and $e_3$ represent the third and second edges  of $H=G_3$, respectively as follows.
\begin{figure}[H]
\begin{center}
\phantom{aaaaaaaaaaaaaa}\input{AlternateEdges.TEX}
\end{center}
\vspace{-7.2cm}\caption{}\label{Falternate}
\end{figure}\vspace{-0.4cm}The restrictions of the homotopies $\lambda'_k$ and $\lambda'_{kl}$ to $t=0\in I$ are equal to  $f'_{k}$ and $f'_{kl}$, respectively,  so both of $Dc(G,H,i)$ and $\delta_ic(H)$ have the terms of $f'_k$ and $f'_{kl}$. Since the signs on the terms of the chains  depend on the order of edges, the alternation of edges gives the  signs on the terms in $Dc(G,H,i)$ which are opposite to the signs on the corresponding terms of $c(H)$, except for the terms of $f'$ and $f'_1$. For $f'_{23}$, the signs given by the order of edges are the same on both sides, but the concatenated homotopy  swaps the components of $[0,\infty]^{\wedge 2}$, that is,  we have $\lambda'_{23}|_{t=0}=f'_{23}\circ (id_{S^4}\wedge T\wedge id_I)$ where $T$ is the transposition on $[0,\infty]^{\wedge 2}$, which produces a sign.   The exceptions $f'$ and $f'_1$ are complemented by the extra signs on $\delta_i$ in permuting  edges on $c(H)$  since we have 
$sgn(\sigma_{H,i})=sgn(\sigma_{\partial_1H,i})=-1$ and $sgn(\sigma_{H',i})=1$ for the other subgraphs $H'$ of $H=G_3$ (see Definition \ref{Dtriple}).  The proof of the fourth equation is similar.
\end{proof}

\begin{defi}\label{Dchain_cycle}
Let $(G,i)$ be one of $(G_1,1), (G_2,1)$ and $(G_4,2)$. Set $f=f_G$ and $E(G)=\{e_1<e_2<e_3\}$. Note that $\delta_iG$ has double edges and $e_2$ corresponds to one of them under the standard bijection.
Let 
$f_k$ be the $e_k$-contraction of $f$, and $f_{kl}$ the $(e_k, e_l)$-contraction of $f$ (both for $G$). Let $f_k^{\eps}$ (resp.$f_{kl}^{\eps}$) be the $(i,\eps)$-contraction of $f_k$ (resp. $f_{kl}$) for $\eps=\pm$  (see Definition \ref{Di-cont}).
We set
\[
\begin{split}
c(G,i)=\sum_{1\leq k\leq 3}&(-1)^{k+1}f^\pm_k(w_2)+\sum_{1\leq k<l\leq 3}(-1)^{k+l+1}f^\pm_{kl}(w_3)\\
&\in \bar C_*(\TT_{\delta_i\partial_{2}G})\oplus\bar C_*(\TT_{\delta_i\partial_{12}G})\oplus \bar C_*(\TT_{\delta_i\partial_{23}G}).
\end{split}
\]
Here,  we use the notation given in subsection \ref{SSNT}. Also, one of  $\delta_i\partial_1G$ and $\delta_i\partial_3G$ has double edges and the other is equal to $\delta_i\partial_2G$. The term $f^\pm_k(w_2)$ corresponding to the graph with double edges is regarded as zero and the other term is regarded as a chain in $\bar C_*(\TT_{\delta_i\partial_2G})$.  The graph  $\delta_i\partial_{13}G$ is equal to exactly  one of $\delta_i\partial_{12}G$ and $\delta_i\partial_{23}G$ and the term of $f^{\pm}_{13}$ is regarded as an element of the chain complex of the one equal to $\delta_i\partial_{13}G$ .  For example, for $(G,i)=(G_1,1)$,  the term of $f^{\pm}_3$ is regarded as zero, and those of $f^{\pm}_1$ and $f^{\pm}_2$ belongs to $\bar C_*(\TT_{\delta_1\partial_2G_1})$ since $\delta_1\partial_1G_1=\delta_1\partial_2G_1$. The term of $f^{\pm}_{13}$ belongs to  $ \bar C_*(\TT_{\delta_1\partial_{23}G_1})$. The chains are well-defined by Lemma \ref{Li-cont}.
\end{defi}
\begin{lem}\label{L1st_bound_ch3}
Let $(G,i)$ be as in Definition \ref{Dchain_cycle}. We have $D(c(G,i))=0$. Set
\[
C_1=c(G_1,G_2,3)+c(G_2,G_3,2)+c(G_2,G_3,4)-c(G_3,G_4,1)+c(G_4,G_3,4)-c(G_1,1)+c(G_2,1)-c(G_4,2).
\]
We have $D(C_1)=-\delta(-c(G_1)+c(G_2)+c(G_3)+c(G_4))$.
\end{lem}
\begin{proof}
We shall show the  equation for $c(G,i)$ in the case $(G,i)=(G_1,1)$. Let $s$ and $s'$ be the variable for $e_k$-contraction and $(i,\eps)$-contraction, respectively. Since $\delta_i\partial_1G_1=\delta_i\partial_2G_1$ and  the restrictions of $f_1^\eps$ and $f_2^\eps$ to $s=0$ are the same, the corresponding terms in $df_1^\eps(w_2)$ and $df_2^\eps(w_2)$ cancel with each other. Put $f^\eps_k=({f^\eps_{k}}^1,\dots,{f^\eps_{k}}^5)$. We have ${f^\eps_k}^1|_{s'=0}=_1{f^\eps_{k}}^2|_{s'=0}$. By Lemma \ref{Lcondensed_collapse0}, the induced map is $*$. Thus, $d(\sum_k(-1)^{k+1}f_k^\pm(w_2))=0$. Similarly, the chain $df^\eps_{kl}(w_3)$  consists of the two terms corresponding to the restriction of each of the two variables in removing edges to $0$. The terms $\partial_1f^\eps_1(w_2)$ and $\partial_1f_2^\eps(w_2)$ cancel with the two terms of $df^\eps_{12}(w_3)$. Each of  $\partial_2f^\eps_1(w_2)$ and $\partial_2f_2^\eps(w_2)$ cancels with one of the terms  of each of $df^\eps_{13}(w_3)$ and $df^\eps_{23}(w_3)$, respectively. The rest of  terms of $df^\eps_{13}(w_3)$ and $df^\eps_{23}(w_3)$ cancel with each other.   Thus, $c(G_1,1)$ is a cycle. The other cases are similar. The equation for $D(C_1)$ follows from Lemma \ref{Lbound_ch3} with an argument similar to the proof of Lemma \ref{Lh_i_welldef}.
\end{proof}
Of course, we do not need $c(G_k,i)$ in $C_1$ to make the equation for $D(C_1)$ in the lemma hold but  it makes later constructions easier.
\subsection{Second bounding chain}\label{SS2nd_ch3}
\begin{defi}\label{D2nd_chain_2graphs}

 Let $(G,H,i,j)$ be one of the followng quartets
\[
\begin{split}
(G_2, G_3,2,4),\ & 
(G_2, G_3,4,2),\  
(G_4, G_3,4,2),\  
(G_2, G_3,4,1),\
(G_4, G_3,4,1),\
(G_3, G_4,1,4),\\
(G_1,G_2,3,1),\ &
(G_2, G_3,2,1),\
(G_3, G_4,1,2).\
\end{split}
\]
For the triples $(G,H,i)$ of these quartets, we use notations in Definition \ref{Dchain_2graphs}. 
 So $\psi_k, \psi_{kl}, \lambda_k, \lambda_{kl}, \lambda'_k,$ and $\lambda'_{kl}$ denotes the maps given there. The graphs $G, H$ satisfy  $\delta_iG=\delta_iH$ and $\delta_{ij}G$ has double edges one of which corresponds to $e_2\in E(G)$ (see Definition \ref{Dgraph}). In the following, the superscript $\eps$ on these maps represents $(j,\eps)$-contraction of the map for $\eps=\pm$. For example, $\psi_k^\eps$ denotes the $(j,\eps)$-contraction of $\psi_k$.
We set
\[
\begin{split}
c(G,H,i,j)=
\sum_{1\leq k\leq 3}& (-1)^{k}sgn(\sigma_{\partial_kG, j})(\psi_k^{\pm}+\lambda_k^\pm -{\lambda'_k}^\pm)(w_{21})+\sum_{1\leq k<l\leq 3}(-1)^{k+l+1}(\psi_{kl}^\pm+\lambda_{kl}^{\pm}-{\lambda'}^\pm_{kl})(w_{31})\\
&\in \bar C_*(\TT_{\delta_{ij}\partial_{2} G})\oplus \bar C_*(\TT_{\delta_{ij}\partial_{12}G})\oplus \bar C_*(\TT_{\delta_{ij}\partial_{23}G}).
\end{split}
\]

Here, we use the same convention as   Definitions \ref{Dchain_2graphs} and  \ref{Dchain_cycle}. Namely, the maps $\psi_{k}^\pm$  are implicitly composed with the transposition $[0,\infty)^2\times I\to I\times [0,\infty)^2, (s_1,s_2,t) \mapsto (t,s_1,s_2)$ and the maps $\psi_{kl}^\pm$ are also composed with a similar cyclic transposition.  The terms of a graph with double edges are regarded as zero, and the terms with the subscript {13} belong to the complex of the graph which is equal to $\delta_{ij}\partial_{13}G$.   See Definition  \ref{Dtriple} for $sgn(\sigma_{\partial_kG,j})$. Actually, for the terms without double edges, $sgn(\sigma_{\partial_kG,j})=1$ with only one exception $(G,H,i,j,k)=(G_3,G_4,1,2,1)$. The chains are well-defined by Lemma \ref{Li-cont} for the case $|i-j|\geq 2$ and by Lemma \ref{Li-cont2} for the case $|i-j|=1$.

\end{defi}

\begin{defi}\label{D2nd_chain_2contractions}

 Let $(G,i,j)$ be one of the following triples
\[
(G_2,2,4),\ 
(G_3,2,4),\ 
(G_4,4,1),\
(G_3,4,1).
\]
Put $E(G)=\{e_1<e_2<e_3\}$. For these triples, 
 $\delta_{ij}G$ has double edges one of which corresponds to $e_2$ via the standard bijection.
Let $f_k$ (resp. $f_{kl}$) be the $e_k$-contraction (resp. the $(e_k,e_l)$-contraction) of $f_G$ for $G$.
Let $\mu^{\eps}_k$ (resp. $\mu^{\eps}_{kl}$) be the straight homotopy  from the $(i,\eps)$-contraction of $f_k$ (resp. $f_{kl}$) to $(j,\eps)$-contraction of $f_k$ (resp. $f_{kl}$). 
We set
\[
\begin{split}
c(G,i,j)=\sum_{1\leq k\leq 3}&(-1)^{k}\mu^\pm_k(w_{21})+\sum_{1\leq l<k\leq 3}(-1)^{k+l+1}\mu^\pm_{kl}(w_{31}) \\
&\in \bar C_*(\TT_{\delta_{ij}\partial_{2} G})\oplus \bar C_*(\TT_{\delta_{ij}\partial_{12}G})\oplus\bar C_*(\TT_{\delta_{ij}\partial_{23}G}).
\end{split}
\]
Here we use the same convention as Definition \ref{Dchain_cycle}. The chains are well-defined by Lemma \ref{Li-cont_homotopy} (1).
\end{defi}
In what follows, we construct a bounding chain of the cycle $\delta (C_1)$.  We split the construction into the construction of a  bounding chain of  each part of the cycle belonging to a single partition in $\PP_5$.   
\begin{lem}\label{L2nd_bound_2,4}
On $\delta_{24}[6]=\{\{0\},\{1\},\{23\},\{45\},\{6\}\}$, we have
\[
\begin{split}
Dc(G_2, G_3,2,4)& =\delta_3c(G_2,G_3,2)+c'(G_2,4)+c'(G_3,4), \\
Dc(G_2, G_3,4,2)& =\delta_2c(G_2,G_3,4)+c'(G_2,2)-c'(G_3,2), \\
Dc(G_4, G_3,4,2)& =\delta_2c(G_4,G_3,4)+\delta_3c(G_4,2)-c'(G_3,2), \\
Dc(G_2,2,4)&=c'(G_2,2)-c'(G_2,4),\\
Dc(G_3,2,4)&=c'(G_3,2)-c'(G_3,4).
\end{split}
\]
Here, $c'(G_k,i)$ is the chain defined by the same formula as $c(G,i)$ for $G=G_k$.
If we set
\[
C_{21}=c(G_2,G_3,2,4)-c(G_2,G_3,4,2)-c(G_4,G_3,4,2)+c(G_2,2,4)+c(G_3,2,4),
\]
we have 
\[
DC_{21}=\delta_3(c(G_2,G_3,2)-c(G_4,2))-\delta_2(c(G_2,G_3,4)+c(G_4,G_3,4)).
\]
\hfill 
\end{lem}
\begin{proof}
The proof is similar to Lemma \ref{Lbound_ch3} and we omit the details. The following figures will provide  intuitive explanation. 
The differential  of $c(G_4,G_3,4,2)$ is illustrated as follows.
\begin{figure}[H]
\begin{center}
\input{cG4G342.TEX}\vspace{4mm}
\end{center}
\caption{}\label{Fc(G4,G3,4,2)}\vspace{-0.7cm}
\end{figure}
Here,  we actually consider subgraphs of the presented ones made by removing  edge(s), and omit $\lambda_{kl}^\pm$'s.  The two-sided arrow $\leftrightarrow$ means addition of $\pm su$ on the corresponding components in the definition of $(i,\pm)$-contraction. The terms of $Dc(G_4,G_3,4,2)$ absent from this figure cancel with one another as in the proof of Lemma \ref{L1st_bound_ch3}.
Cancellation of some parts of boundaries (i.e. singular differentials) is illustrated as follows. 
\begin{figure}[H]
\begin{center}
\phantom{aaaaaa}\input{cancel.TEX}\vspace{-3cm}
\end{center}
\caption{}\label{Fcancel}
\vspace{-0.5cm}
\end{figure}
 The two-sided arrows labeled by `b' mean cancellation between parts of boundaries. The trident means cancellation of three parts by $3=0$. Let $s\in [0,\infty)$ be the variable for contraction in removing an edge. The terms of $\mu_k^\pm|_{s=0}$ and $\mu_{kl}^\pm|_{s=0}$ cancel with one another or with the $\CECHC$ech differential of other terms as in the proof of Lemma \ref{L1st_bound_ch3}. For the variable $s'$ of $(i,\pm)$-contraction, the terms of $\mu_k^\pm|_{s'=0}$ and $\mu_{kl}^\pm|_{s'=0}$   are zero since these maps are independent of $t\in I$.  The sum of the rest of the non-zero boundaries is equal to the right hand side of the equation for $D(C_{21})$ in the claim.
\end{proof} 

\begin{lem}\label{L2nd_bound_1,4}
On $\delta_{14}[6]$, we have 
\[
\begin{split}
Dc(G_2, G_3,4,1)& =\delta_1c(G_2,G_3,4)+\delta_3c(G_2,1)-c'(G_3,1), \\
Dc(G_4, G_3,4,1)& =\delta_1c(G_4,G_3,4)+c'(G_4,1)-c'(G_3,1), \\
Dc(G_3, G_4,1,4)& =\delta_3c(G_3,G_4,1)+c'(G_3,4)+c'(G_4,4), \\
Dc(G_4,4,1)&=c'(G_4,4)-c'(G_4,1),\\
Dc(G_3,4,1)&=c'(G_3,4)-c'(G_3,1).
\end{split}
\]
Here, $c'(G_k,i)$'s are defined completely similarly as in Lemma  \ref{L2nd_bound_2,4}.  
If we set
\[
C_{22}=c(G_2,G_3,4,1)-c(G_3,G_4,1,4)+c(G_4,G_3,4,1)+c(G_3,4,1)+c(G_4,4,1),
\]
we have
\[
DC_{22}=\delta_1(c(G_2,G_3,4)+c(G_4,G_3,4))+\delta_3(c(G_2,1)-c(G_3,G_4,1)).
\]
\end{lem}
\begin{proof}
The cancellation of chains in this lemma is similar to previous lemma, so we omit details.
\end{proof}
\begin{lem}\label{L2nd_bound_1,3}
On $\delta_{13}[6]$, we have
\[
Dc(G_1,G_2,3,1)=\delta_1c(G_1,G_2,3)+\delta_2c(G_1,1)-\delta_2c(G_2,1).
\]
\hfill \qedsymbol
\end{lem}

\begin{lem}\label{L2nd_bound_1,2}
On $\delta_{12}[6]$, we have
\[
\begin{split}
Dc(G_2, G_3,2,1)& =\delta_1c(G_2,G_3,2)+\delta_1c(G_2,1)+c'(G_3,1), \\
Dc(G_3, G_4,1,2)& =\delta_1c(G_3,G_4,1)+c'(G_3,2)+\delta_1c(G_4,2), \\
Dc'(G_3,1,2)&=c'(G_3,1)-c'(G_3,2).
\end{split}
\]
Here, $c'(G_3,i)$'s are defined  similarly as in Lemma  \ref{L2nd_bound_2,4} except for the extra sign $sgn(\sigma_{H,1})=-1$ on the term of $f_1^{\pm}$, where $H=\delta_1(\partial_1G_3)$.  $c'(G_3,1,2)$ is also similar to $c(G,i,j)$ but has the same extra sign on the term $\mu_1^{\pm}$.
If we set 
\[
C_{23}=c(G_2,G_3,2,1)-c(G_3, G_4,1,2)-c'(G_3,1,2),
\]
we have
\[
DC_{23}=\delta_1(c(G_2,G_3,2)-c(G_3,G_4,1)+c(G_2,1)-c(G_4,2)).
\]

\end{lem}
\begin{proof}
One can see that $c'(G_3,1,2)$ is well-defined as in the proof of  Lemma \ref{Li-cont_homotopy}. In the first equation, the only point which we need to care about is the sign of the third term of the right hand side. The composition of standard bijections $E(G_2)\cong E(\delta_2G_2)= E(\delta_2G_3)\cong E(G_3)$ swaps the second and third edges. So the corresponding straight homotopy negates the terms except for the one corresponding to the removal of the first edge. This exception is complemented by the extra sign in $c'(G_3,1)$. In the second equation, the extra sign on $\delta_1$ of $c(G_3,G_4,1)$ in permuting edges is negative only for the term corresponding to $\delta_1\partial_1(G_3)$. This is complemented by the  sign $sgn(\sigma_{\partial_1G_3,2})$ in the definition of $c(G_3,G_4,1,2)$. The bijection $E(G_3)\cong E(\delta_1G_3)= E(\delta_1G_4)\cong E(G_4)$ swaps the first and second edges. This negates the terms except for the removal of the  third edge. This is complemented by the extra sign on $\delta_1$ of $c(G_4,2)$.
\end{proof}

\begin{lem}\label{L2nd_bound_all}
Let $C_1, C_{21}, C_{22}$, $C_{23}$ be as in Lemmas \ref{L1st_bound_ch3}, \ref{L2nd_bound_2,4}, \ref{L2nd_bound_1,4}, and \ref{L2nd_bound_1,2}. If we set 
\[
C_2=C_{21}+C_{22}+C_{23}+c(G_1,G_2,3,1),
\]
we have
$
D(C_2)=-\delta(C_1).
$ 
\end{lem}
\begin{proof}
Arguments similar to the proof of Lemmas \ref{Lh_i_welldef},  \ref{Lfundcyc}, together with Lemmas  \ref{L2nd_bound_2,4}, \ref{L2nd_bound_1,4},  \ref{L2nd_bound_1,3}
and \ref{L2nd_bound_1,2}  imply the claim.
\end{proof}
\subsection{Non-triviality of the $d_3$-differential}\label{SS3rd}
In this section, we simplify the cycle representing the $d_3$-differential and prove its non-triviality.
\begin{lem}\label{L3rd_zero}
Let $C_2$ be the chain given in Lemma \ref{L2nd_bound_all}. 
Any pushforward  which appears as a term of $\delta(C_2)$, other than the following four terms, is zero:
\[
\begin{split}
\delta_1\psi^-_{kl}(w_{31}) &\text{ of } \delta_1c(G_4,G_3,4,2), \ \delta_1\mu_{kl}^+(w_{31}) \text{ of } \delta_1c(G_2,2,4), \\
\delta_1\mu_{kl}^-(w_{31})& \text{ of } \delta_1c(G_3,2,4), \text{ and }\  \delta_1\psi^+_{kl}(w_{31}) \text{ of } \delta_1c(G_2,G_3,4,1). 
\end{split}
\]
\end{lem}
\begin{proof}
 Firstly, since the vertices $\{1\}$ and $\{5\}$ of $G_k$ are not discrete for each $1\leq k\leq 4$, $\delta_0$ and $\delta_3$ of all the chains are zero by an argument similar to  the proof of Lemma \ref{Lh_i_welldef} with the aid of   Lemmas \ref{Li-cont},  \ref{Li-cont2}, and \ref{Li-cont_homotopy}.  Secondary, if $\delta_{abc}G_k$ has a loop, the maps corresponding to its subgraphs are $*$. This also follows from an argument similar to Lemma \ref{Lh_i_welldef}. By this observation, we see that all the maps in the following chains are $*$:
\[
\begin{split}
\delta_1 \text{ of }& c(G_2,G_3,2,1), c(G_3,G_4,1,2), c'(G_3,1,2), c(G_1,G_2,3,1),\\
\delta_2 \text{ of }& c(G_2,G_3,4,2), c(G_4,G_3,4,2), c(G_2,2,4), c(G_3,2,4), c(G_2,G_3,2,4),\\ &c(G_2,G_3,4,1), c(G_4,G_3,4,1), c(G_3,G_4,1,4), c(G_3,4,1), \text{ and } c(G_4,4,1). 
\end{split}
\]Thirdly, some of the other terms are zero by the latter part of Lemma \ref{Lcondensed_collapse0}. For example, we consider the term $\delta_1\psi^\eps_{kl}(w_{31})$ of  $\delta_1c(G_2,G_3,2,4)$. The map $\psi^+_{kl}$ swaps $x$ and $y$ by the homotopy in the first, fourth, and fifth components and adds $-su$ and $su$ to the fourth and fifth components respectively, so it may be expressed as follows:
\[
\{\underline{x}\   x\ y\}\{\overrightarrow{\underline{x}}\  \overleftarrow{\underline{y}}\}.
\]
Here, the arrows $\rightarrow, \leftarrow$ indicate the addition of $su, -su$, respectively,  $\underline{x}$ (resp. $\underline{y}$) represents the homotopy from $x$ to $y$ (resp. $y$ to $x$), and $\{\ \}$ specifies a vertex (the vertex set of the graph corresponding to the map $\delta_1\circ \psi^+_{kl}$  is $\{\{0\},\{123\},\{45\},\{6\}\}$). We ignore $\pm sv$ as it does not affect the first coordinate. Write $\psi^+_{kl}=(\psi_1,\dots,\psi_5)$ for simplicity. If $x\leq_1y$, $\psi_1=\underline{x}\geq_1x=\psi_2$, and if $x>_1y$, $\psi_2=x>_1y=\psi_3$ since $\underline{x}$ and $\underline{y}$ always lie between $x$ and $y$. So we have $\delta_1'\circ \psi^+_{kl}=*$ by the latter part of Lemma \ref{Lcondensed_collapse0}. $\delta_1'\circ \psi^-_{kl}$ is expressed as $\{\underline{x}\   x\ y\}\{\overleftarrow{\underline{x}}\  \overrightarrow{\underline{y}}\}$ and we have $\delta_1'\circ\psi^-_{kl}=*$ similarly. As another example, we consider the term $\delta_1\mu^+_{kl}(w_{31})$ of $\delta_1c(G_4,4,1)$. $\delta_1'\circ \mu^+_{kl}$ is expressed as 
\[
\{\overrightarrow{x}\   \overleftarrow{y}\ y\}\{\overrightarrow{y}\  \overleftarrow{x}\}
\]
Write $\mu^+_{kl}=(\mu_1,\dots,\mu_5)$. We see that $\mu_4=\overrightarrow{y}>_1\overleftarrow{x}=\mu_5$ if $x<_1y$, and $\mu_1>\mu_2$ otherwise, which implies $\delta_1'\circ\mu^+_{kl}=*$.
The terms which cannot be seen to be zero by the above three kinds of observations are the following five terms, where we omit subscripts.
\[
\begin{split}
\delta'_1\circ\psi^\pm&\text{ of }\delta_1c(G_2,G_3,4,2),
\ \delta'_1\circ\mu^-\text{ of }\delta_1 c(G_2,2,4),\ \delta'_1\circ\psi^-\text{ of }\delta_1 c(G_2,G_3,4,1),\\
\delta_2'\circ\psi^{-}& \text{ of } \delta_2c(G_3,G_4,1,2), \text{ and }\delta_2'\circ\psi^\pm \text{ of }\delta_2c(G_2,G_3,2,1).
\end{split}
\]

We consider the first term. The map $\delta'_1\circ\psi^+$ is expressed as 
\[
\{x\   \overrightarrow{\underline{x}}\ \overleftarrow{\underline{y}}\}\{x\  y\}
\]
Put $P=\{\{0\},\{123\},\{45\},\{6\}\}$. Suppose $\psi^+(\tilde x)\in \nu_P$ for $\tilde x=(x,y,\dots)\in \RR^4\times [0,\infty)^3\times I$. By fourth and fifth components, we have $x<_1y$. Put $\psi^+=(\psi_1,\dots,\psi_5)$. Since the average of $\psi_2(\tilde x)=\overrightarrow{\underline{x}}$ and $\psi_3(\tilde x)=\overleftarrow{\underline{y}}$ is equal to $(x+y)/2$, we see
\[
\psi_3(\tilde x)-\psi_1(\tilde x)>_1\frac{1}{2}(\psi_5(\tilde x)-\psi_4(\tilde x)).
\]
This implies
$
c_1+2c_2+c_3>(c_4+c_5)/2-4\eps_P
$
by the assumption of $\tilde x$ and an argument similar to the proof of Lemma \ref{Ldiagonal_incl}.
This is impossible when $c_r/c_{r-1}$ is sufficiently large as assumed in Definition \ref{DPK}, which implies $\delta'_1\circ \psi^+=*$. We shall consider the fifth term. The map $\delta'_2\circ\psi^-$ for $(G_2,G_3,2,1)$ is expressed as 
\[
\{ \overleftarrow{\underline{x}}\  \overrightarrow{x}\ y\}\{\underline{x}\ \underline{y}\}.
\]
Put $\psi^-=(\psi'_1,\dots, \psi'_5)$. If $x\leq_1y$, we easily see 
\[
\psi'_3-\psi'_1=y-\overleftarrow{\underline{x}}>_1\underline{y}-\underline{x}=\psi'_5-\psi'_4,
\] which contradicts the assumption on $c_r$. If $x>_1y$, clearly $\psi'_3<_1\psi'_2$ which implies $\delta'_1\circ \psi^+=*$ by Lemma \ref{Lcondensed_collapse0}. The other maps are shown to collapse similarly.\end{proof}
We shall prove three of the four terms in Lemma \ref{L3rd_zero} cancel with one another.
\begin{defi}\label{D3rd_chain_2graphs}
Put $G_0=\delta_{124}\partial_{12}G_2$. 
This graph has only one edge $(\{123\},\{45\})$.
\begin{enumerate}
\item  Let $\tilde \psi^{1+}_{kl}$ (resp. $\tilde\psi^{2-}_{kl}$, $\hat\psi^{2-}_{kl}$, $\hat\psi^{1+}_{kl}$) denote $\psi^+_{kl}$ (resp. $\psi^-_{kl}$, $\psi^{-}_{kl}$, $\psi^+_{kl}$) in Definition  \ref{D2nd_chain_2graphs} for $(G,H,i,j)=(G_2,G_3,4,1)$ (resp. $(G_2,G_3,4,2)$, $(G_4,G_3,4,2)$, $(G_4,G_3,4,1)$). Let $\tilde \eta_{kl}$ (resp. $\hat \eta_{kl}$) be the straight homotopy from $\tilde \psi^{1+}_{kl}$ to $\tilde \psi^{2-}_{kl}$ (resp. from $\hat\psi^{2-}_{kl}$ to $\hat\psi^{1+}_{kl}$). Set
\[
B_1 =\frac{1}{2}\sum_{k<l}(-1)^{k+l+1}\tilde \eta_{kl}(w_{32}),\qquad 
B_2 =\frac{1}{2}\sum_{k<l}(-1)^{k+l+1}\hat \eta_{kl}(w_{32})\ \in \bar C_*(\TT_{G_0}).
\]

\item Put $G=G_3$, $E(G)=\{e_1<e_2<e_3\}$,  and $f=f_{G}$. Let $\lambda_{kl}$ be the straight homotopy from the $(e_k,e_l)$-contraction of $f$ for $G$ to the $(e_k,e_l)$-contraction of $f$ for $\delta_4G$ and  $\lambda^{i-}_{kl}$ the $(i,-)$-contraction of $\lambda_{kl}$ for $i=2,4$. Let $\tilde \lambda_{kl}$ be the straight homotopy from $\lambda^{2-}_{kl}$ and $\lambda^{4-}_{kl}$. Set 
\[
B_3=\frac{1}{2}\sum_{k<l}(-1)^{k+l+1}\tilde \lambda_{kl}(w_{32})\ \in \bar C_*(\TT_{G_0}).
\]
For $h=\tilde \eta_{kl},\ \hat \eta_{kl},$ and $\tilde \lambda_{kl}$,  let $t\in I$ denote the variable for the `horizontal' straight homotopy (see Figure \ref{Fhomotopy_ch3}, or the variable in the first component of $I^2$). For $h=\tilde \eta_{kl}$, $t$ is the variable appearing in the definition of  $\tilde \psi^{4+}_{kl}$ (and $\tilde \psi^{2-}_{kl}$) and defined similarly for $h=\hat\eta_{kl}$. For $h=\tilde \lambda_{kl}$, $t$ is the one for $\lambda^{2-}_{kl}$ (and $\lambda^{4-}_{kl}$). For these $h$, we denote by $h'$ by the restriction of $h$ to the subspace defined by $t=1$. 
Let $\theta_{kl}$ be the straight homotopy from $\tilde  \lambda_{kl}'$ to $\hat \eta_{kl}'$. Set 
\[
B_4=\frac{1}{2}\sum_{k<l}(-1)^{k+l+1} \theta_{kl}(w_{32})\ \in \bar C_*(\TT_{G_0}).
\]
\end{enumerate}
\end{defi}
The maps in this definition are drawn in Figure \ref{Fhomotopy_ch3}, where we omit the removal of edges in $\tilde \eta_{kl}$, $\hat \eta_{kl}$ and $\theta_{kl}$. For the equations and $\mu_{312}^-$ near the edges of squares, see the proof of Lemma \ref{L3rd_simplify}.
\begin{figure}
\begin{center}
\hspace{-5mm}\input{homotopych3.TEX}
\end{center}
\caption{homotopies in Definition \ref{D3rd_chain_2graphs}} \label{Fhomotopy_ch3}
\end{figure}
The maps $\tilde\eta_{kl}$, $\hat \eta_{kl}$ and $\tilde \lambda_{kl}$ induce maps $S^4\wedge [0,\infty]^{\wedge 3}\wedge (I_+)^2\to \TT_{G_0}$ by Lemma \ref{Li-cont_homotopy} so $B_1, B_2$ and $B_3$ are well-defined. One can verify that $\theta_{kl}$ induces a map between the pointed spaces as in the proof of the lemma so $B_4$ is also well-defined. 
\begin{lem}\label{L3rd_simplify}
The cycle $-\delta C_2$ is homologous to $\delta_1\left(\frac{1}{2}\sum_{k<l}(-1)^{k+l+1}\mu_{2kl}^+(w_{31})\right)$ in $(tr_1\CEE_0,d_0)$. Here, $\mu_{2kl}^+$ is the map $\mu_{kl}^+$ in Definition \ref{D2nd_chain_2contractions} for $(G,i,j)=(G_2,2,4)$. 
\end{lem}

\begin{proof}
Let $\mu_{3kl}^-$  denote  the map $\mu_{kl}^-$ in Definition \ref{D2nd_chain_2contractions} for $(G,i,j)=(G_3,2,4)$. The map for $(k,l)=(1,2)$ fits into the left edge of the left bottom square in Figure \ref{Fhomotopy_ch3} (and similar for other $(k,l)$'s).  The symbols `$*$' and `$=*$' near an edge of  a square in this figure mean that the restriction of the map to the edge induces the constant map $*$ to $\TT_{G_0}$.    For example,  the collapse of the left edge of $\tilde \eta_{kl}$ follows from Lemma \ref{Lcondensed_collapse0} and that of the bottom edge of the  map is included in the proof of Lemma \ref{L3rd_zero} as this restriction is a term of $c(G_2,G_3,4,2)$. The collapse of the bottom edges of $\hat\eta_{kl}$ and $\theta_{kl}$ is also included in the same proof. Collapse of the other edges follows from Lemma \ref{Lcondensed_collapse0}.
Thus, we have
\[
\begin{split}
D(B_1) &=\frac{1}{2}\sum_{k<l}(-1)^{k+l}(\tilde \psi_{kl}^{1+}(w_{31})+\tilde \eta_{kl}'(w_{31}))
=-\delta_1c(G_2,G_3,4,1)+\tilde K,\\
D(B_2) &=\frac{1}{2}\sum_{k<l}(-1)^{k+l}(\hat \psi_{kl}^{2-}(w_{31})+\hat \eta_{kl}'(w_{31}))
=-\delta_1c(G_4,G_3,4,2)+\hat K,\\
D(B_3)& =\frac{1}{2}\sum_{k<l}(-1)^{k+l}(\tilde \lambda_{kl}'-\mu_{3kl}^-)(w_{31}),\\
D(B_4) &=\frac{1}{2}\sum_{k<l}(-1)^{k+l+1}(\hat \eta'_{kl}-\tilde \lambda_{kl}')(w_{31})
  =-\hat K+\frac{1}{2}\sum_{k<l}(-1)^{k+l}\tilde \lambda_{kl}'(w_{31}).
\end{split}
\]Here,  we set 
\[
\tilde K=\frac{1}{2}\sum_{k<l}(-1)^{k+l}\tilde \eta_{kl}'(w_{31}),\qquad \hat K=\frac{1}{2}\sum_{k<l}(-1)^{k+l}\hat \eta_{kl}'(w_{31}).
\]  Putting these equations and Lemma \ref{L3rd_zero} into together, we have
\[
-\delta C_2+D(B_1-B_2-B_3+B_4)=\tilde K-2\hat K+\frac{1}{2}\sum_{k<l}(-1)^{k+l+1}\mu_{2kl}^+(w_{31}).
\]
$\hat K$ is homologous to $-\tilde K$ since the directions of the variable for the vertical  homotopy  are opposite to each other (see Figure \ref{Fhomotopy_ch3}). We have $[\tilde K-2\hat K]=[-3\hat K]=0$. Thus, we have obtained the claim.
\end{proof}
We shall compute the remaining term.
\begin{lem}\label{Lfundcyc_ch3}
The chain $\delta_1\left(\frac{1}{2}\sum_{k<l}(-1)^{k+l+1}\mu_{2kl}^+(w_{31})\right)$  in Lemma \ref{L3rd_simplify} is a fundamental cycle of $\TT_{G_0}\simeq S^8$.
\end{lem}
\begin{proof}
The proof is similar to Lemma \ref{Lfundcyc}. By definition, $\TT_{G_0}$ is a Thom space associated to the disk bundle $\nu_P|_N$, where  $N=D_{\alpha\beta}\cap (\,\overline{\RR^4-E_P}\,)$, where $\alpha=\{1,2,3\}, \beta=\{4,5\}$ and $P=\{\{0\},\alpha,\beta,\{6\}\}$ in the notation of Definition  \ref{Dpartition}.  Write $\nu=\nu_P|_N$. Let $\Delta$ be the intersection of diagonal in $(\RR^2)^2$ and $N$, and  $\tilde r: \nu\to \nu|_{\Delta}$ the bundle map which covers the orthogonal projection $r: N\to \Delta$ and  restricts to the parallel transport taking center to center on each fiber. This map induces a map $\tilde r:\TT_{G_0}\to Th(\nu|_{\Delta})\cong S^8$.  We consider $\nu$ and $(\nu|_\Delta)$ as  subspaces of $\RR^{10}$. Put 
\[
F_{kl}:=\tilde r\circ \mu^{+}_{2kl}:(\mu^+_{2kl})^{-1}(\nu)\to \nu|_{\Delta}\qquad \text{for}\qquad 1\leq k<l\leq 3.
\]
We shall write down $F_{kl}$ concretely. By definition, $\nu$ is the tubular neighborhood of the map
\[
e_P:(a,b)\mapsto \left(a-\frac{\rho}{2}(c_2+c_3)u,\ a+\frac{\rho}{2}(c_1-c_3)u,\ a+\frac{\rho}{2}(c_1+c_2)u,\ b-\frac{\rho}{2}c_5u,\ b+\frac{\rho}{2}c_4u\right).
\]
By elementary calculation, the point $\pi_P(c,d,e,f,g)$ is given by
\[
\pi_P(c,d,e,f,g)=\left(\frac{1}{3}(c+d+e+\rho(c_3-c_1)u),\ \frac{1}{2}(f+g+\frac{\rho}{2}(c_5-c_4)u)\right).
\]
Similarly, the map $r$ is given by $r(a,b)=(a+b)/2$.
For a while, we omit the subscript $kl$  and $A_{e_k,e_l}(s_1,s_2)$ in Definition \ref{Dcondensed2} is also abbreviated as $A=(A^1,\dots, A^5)v$, and  $\mu^+_{2kl}$ is abbreviated as $\mu$.  For $\tilde x=(x,y,s_1,s_2,s_3,t)$, we have 
\[
\mu(\tilde x)=(x+A^1v,x+(1-t)s_3u+A^2v, y-(1-t)s_3u+A^3v,x-ts_3u+A^4v, y+ts_3u+A^5v). 
\] Straightforward computation shows 
\[
r\circ \pi_P\circ \mu(\tilde x)=\frac{1}{12}\left(7x+5y+\rho\bigl(2c_3-2c_1+\frac{3}{2}c_5-\frac{3}{2}c_4\bigr)u-(s_1+s_2)v\right)\quad (=:w)
\] 
for any $k,l$.
We denote the right hand side of this equation by $w$. For simplicity, we move the fiber of $\nu$ over $\pi_P(\mu(\tilde x))$ by the parallel transport which sends its center to $0$. By this move, $\mu(\tilde x)$ is sent to 
\[
\mu(\tilde x)-e_P(\mu(\tilde x))=(p,q,-p-q, m,-m)
\] 
where 
\[
\begin{split}
p &=\frac{1}{3}(x-y)+\frac{\rho}{6}(2c_1+3c_2+c_3)u+\frac{1}{3}(2A^1-A^2-A^3)v\\
q &=\frac{1}{3}(x-y)+\frac{\rho}{6}(c_3-c_1)u+(1-t)s_3u+\frac{1}{3}(-A^1+2A^2-A^3)v\\
m&=\frac{1}{2}(x-y)+\frac{\rho}{4}(c_4+c_5)u-ts_3u+\frac{1}{2}(A^4-A^5)v\\
\end{split}
\] The map $F_{kl}$ is given by $\tilde x\mapsto (w,p,q,m)$. The fiber of $\pi_P|_\nu$ is a disk of radius $\epsilon_P$. To prove the lemma, it is enough to show that there exists a point $\tilde x$ such that $\bar w=w$, $\bar p=p$, $ \bar q=q$, $\bar m=m$  for a given point $(\bar w,\bar p,\bar q,\bar m)$ with $|(\bar p,\bar q,\bar m)|\leq \epsilon_P$, and the combination of such a point and numbers $k,l$ is unique unless $s_1$ or $s_2=0$. We fix $\bar w, \bar p, \bar q,\bar m$ and suppose $(\bar w, \bar p,\bar q,\bar m)=(w, p,q,m)$.  By using the formula for $m$, we eliminate $x-y$ in the formulas of $p,q$. We have
\[
\begin{split}
\bar p&=\frac{1}{3}(2\bar m+(2ts_3+\frac{\rho}{2}(2c_1+3c_2+c_3-c_4-c_5))u+(2A^1-A^2-A^3-A^4+A^5)v),\\
\bar q&=\frac{1}{3}(2\bar m+((3-t)s_3+\frac{\rho}{2}(c_3-c_1-c_4-c_5))u+(-A^1+2A^2-A^3-A^4+A^5)v).
\end{split}
\]
By these formulas, we have
\[
\begin{split}
s_3&=\frac{1}{4}(2\bar p_1+2\bar q_1-4\bar m_1+\rho(c_4+c_5-c_2-c)),\\
t&=\frac{6\bar p_1-4\bar m_1+\rho(c_4+c_5-2c_1-3c_2-3c_3)}{2\bar p_1+4\bar q_1-4\bar m_1+\rho (c_4+c_5-c_2-c_3),}
\end{split}
\]
where the subscript $1$ means the first coordinate. The conditions on $c_i$ and $(\bar p,\bar q,\bar m)$ ensure $s_3>0$ and $0<t<1$. We consider the  last terms of the above formulas for $\bar p,\bar q$. For $(k,l)=(1,2), (1,3),$ and $(2,3)$, we have
\[
\begin{split}
(2A^1-A^2&-A^3-A^4+A^5, -A^1+2A^2-A^3-A^4+A^5)\\
&=(4s_1-2s_2,-2s_1+4s_2), \quad 
 (4s_1-2s_2,-2s_1-2s_2), \ \text{ and }\ (-2s_1-2s_2, 4s_1-2s_2),
\end{split}
\]
respectively. When $s_1, s_2\geq 0$ vary, this point runs through a domain $D_{kl}\subset \RR^2$ for each $(k,l)$. It is easy to see that $\cup_{kl}D_{kl}=\RR^2$, $\partial D_{kl}$ corresponds to $s_1$ or $s_2=0$, and $D_{kl}\cap D_{k',l'}\subset \partial D_{kl} \cap \partial D_{k',l'} $ if $(k,l)\not=(k',l')$. The values of $\bar p_1, \bar q_1, \bar m$ determine (and are determined by)  $s_3, t, m$. Since $w$ is given by the same formula for any $(k,l)$ and the formulas of  $\bar p$ and $\bar q$ for different $(k,l)$'s are only different in the coefficients of $v$,  the pair of second coordinates $(\bar p_2,\bar q_2)$ determines a unique pair $(k,l)$ for which the image of $F_{kl}$ contains the point $(\bar w,\dots, \bar m)$ unless $(\bar p_2,\bar q_2)$ corresponds to a point in  $\cup_{kl}\partial D_{kl}$. The values of $(\bar p,\bar q,\bar m)$ determine $x-y$  but $7x+5y$ can take any value, so we can set $w$ freely. 
\end{proof}
In view of Lemma \ref{Lfundcyc_ch3} and Proposition  \ref{Pd2_ch3}, the proof of the following theorem is completely similar to Theorem \ref{Tdifferential}.
\begin{thm}\label{Tch3}
In dimension $d=2$ and over a field of  characteristic $3$,  the  element $[-c(G_1)+c(G_2)+c(G_3)+c(G_4)]\in tr_1\CEE^{-5,3}_1$ lifts to an element $ g\in \EE^{-5,3}_1$ which persists up to $\EE^{-5,3}_3$ and satisfies $d_3(g)\not=0$.
\hfill \qedsymbol
\end{thm}

\begin{proof}[Proof of Corollary \ref{Cnon-formality_ch3}]
Here, we denote by $\ASS$ both of the algebraic and topological (discrete) associative operads. If the map $\ASS\to C_*(\KK_2)$ induced by the topological map $\ASS\to \KK_2$ in Definition \ref{DKontsevich} is multiplicatively formal, i.e. connected with the  map $\ASS\to H_*(\KK_2)$ induced on homology by a zigzag of quasi-isomorphisms fixing $\ASS$, the $d_r$-differential of Sinha's sequence is zero for any $r\geq 2$ (see \cite{moriya}). By part 3 of Theorem \ref{Tmain}, the map $\ASS\to C_*(\KK_2)$ is not multiplicatively formal over charcteristic 3. By the argument of the proof of Theorem 1.3 in \cite{moriya}, the same map is also not formal in the sense in the Introduction. It is well-known that the map $\ASS\to C_*(\KK_2)$ and the map $C_*(E_1)\to C_*(E_2)$ in the claim are connected by a zigzag of quasi-isomorphisms. We have proved the claim.  
\end{proof}
\begin{rem}\label{Rdiscussion}
In this remark, we informally explain the author's present understanding on Sinha's spectral sequence and what direction we can proceed to. The proofs of the claims in this remark will be given  elsewhere. Our method can be used to prove some general claims. For example, we can prove all the $d_2$-differentials are zero on $\EE_2$ if the characteristic of $ \kk$ is not $2$ (or $\kk$ is  a ring having $1/2$) and $d=2$ as in the computations given in this and previous sections. For $d=3$, the case drawing much attention, the  $d_2$ and $d_4$-differentials are zero for degree reasons.  In this case, the version of $e$-contraction using the third direction is available. Using this, we can describe the $d_3$-differential in terms of the maps given in section \ref{Scondensed}. If the base ring has $1/2$, we can prove $d_3=0$ and describe $d_5$. If the ring also has $1/3$, we can prove $d_5=0$. On the other hand, if the ring does not have $1/2$ or $1/3$, the author can not prove the vanishing of $d_3$ or $d_5$, respectively from the description for a general cycle. This situation agrees with a result in \cite{BH} which includes the claim that  the first possibly non-trivial differential (after $d_1$) is $d_{1+(d-1)(p-1)}$ over $p$-adic integers. So the author considers it  better to examine concrete examples. What elements are worth computing and  suit  our methods ? To find a cycle itself is a non-trivial problem. In low degree, computer calculations are given in \cite{turchin}. (Similar calculations for the diagonal of Vassiliev's sequence are given in \cite{bar-natan}.) From the calculations and the rational collapse in \cite{LTV, moriya, tsopmene}, we see that the elements whose number of edges is $\leq 4$  have trivial higher differentials for $d\geq 3$. The cycles given in section \ref{Sch2} and  this section are examples of divided products introduced in \cite{turchin}. The divided product $\langle H_1, H_2\rangle$ of two graphs $H_1, H_2$ is a signed sum of the graphs made by permuting vertices of $H_1$ and $H_2$ by shuffles having the leftmost vertex of  $H_1$ on the left from the leftmost vertex of $H_2$. Let $Z_k$ be the graph  with exactly $k+1$-vertices and $k$-edges such that  the leftmost vertex has valence $k$ and the other vertices have valence 1. The cycles given in section \ref{Sch2} and this section  are a reversed version of $\langle Z_1,Z_1\rangle$, $\langle Z_1, Z_2\rangle$, respectively. The elements produced from $Z_k$'s via the divided product have trivial higher differentials if $d\geq 4$ since they correspond to the elements of $H^*(\Omega^2 S^{d-1},\kk)$, and in $d=3$ they may have non-trivial higher differentials but they have nothing to do with the differentials of Vassiliev's sequence for the original space of long knots (without modulo immersions).  So we need to consider different cycles. While only the products of $Z_k$'s are considered in \cite{turchin}, the  product gives cycles for some of other elements. For example, for a cycle $C$,  the product $\langle C, C\rangle$ and a multiple version $C^{\langle m\rangle}$ are cycles at least in characteristic 2. As we saw in previous (sub)sections, even if $C$ has trivial higher differentials, those of  its products are not necessarily  trivial. The products include many terms, but they are well-arranged. It will be an interesting problem to  compute the higher differentials of the (multiple) divided products of cycles in relatively low degrees.
\end{rem}

\section{Absolute non-formality in characteristic 2}\label{Sabsolute}
In this section, we prove Corollary \ref{Cnon-formality_ch2}, and we assume that $\kk$ is a field of characteristic $2$ and $d=2$. Let $\ASS_\infty$ denote the cellular chain operad of Stasheff's associahedral operad. Precisely speaking,  
$\ASS_\infty$ is generated by the set $\{\, \mu_k\in \ASS_\infty(k) \, \}_{k\geq 2}$ (\ $|\mu_k|=k-2$\ )  with partial compositions, freely as a planar graded operad. The differential is given by the following formula:
\[
d\mu _k=\sum_{
l,p,q}\,\mu_l\circ_{p+1}\mu_q.
\]
where $l,p,q$ run through the range $l, q\geq 2, 0\leq p\leq l-1,$ and $l+q=k+1$ .
\begin{defi}\label{Dhochschild}
For a vector space or complex  $U$ over $\kk$, we denote by $U^\vee$  its linear dual (with the induced differential). Let $f:\ASS_\infty\to \oper$ be a map of (planar) chain operads. Let $\mu'_l\in \oper(l)$ be the image of $\mu_l\in\ASS_\infty(l)$ by $f$. We define a linear map $(-\circ_i\mu_l):\oper(m)^\vee\to \oper(m-l+1)^\vee$ for integers $m\geq l$ and $1\leq i\leq l$ as the following composition
\[
\xymatrix{\oper(m)^\vee\ar[r]^{(-\circ_i-)^\vee\quad \qquad}& (\oper(l)\otimes\oper(m-l+1))^\vee\ar[r]& \oper(m-l+1)^\vee }
\] 
where the right arrow is the evaluation of the first factor on  $\mu'_l$. We also define $(\mu_l\circ_i-): \oper(m)^\vee\to \oper(m-l+1)^\vee$ for integers $m\geq l$ and $1\leq i\leq m-l+1$ similarly using the evaluation of the second factor. 
We define a chain complex  $(\Hoch\oper, \tilde d)$ called \textit{Hochschild complex of $\oper$},  as follows.
Set $\Hoch^{-p,q}\oper=(\oper (p)_q)^\vee$.  The  differential $\tilde d$ is given as a map
\[
\tilde d =d+\delta : \underset{q-p=k}{\bigoplus}\Hoch^{-p,q} \oper \longrightarrow \underset{q-p=k+1}{\bigoplus}\Hoch^{-p,q} \oper.
\] 
Here $d$ is the internal (original) differential on $\oper(p)^\vee$ and $\delta$ is given by the  formula 
\[
\delta (x)=\sum_{ 2\leq l\leq p}\mu_l* x,\qquad \text{ where }\qquad \mu_l*x=x\circ_1\mu_l+x\circ_l\mu_l+\sum_{i=1}^{p-l+1}\mu_l\circ_ix
\]
for $x\in \oper (p)^\vee$. We define a  spectral sequence $E^{-p,q}_r(\oper)$ by filtering $(\Hoch\oper, \tilde d)$ by the arity $p$.
\end{defi}
We call a map $f:\oper\to \aoper$ of chain operads  a {\em quasi-isomorphism} if it induces a quasi-isomorphism $\oper(p)\to \aoper(p)$ for each $p$. 
The following lemma is clear.
\begin{lem}\label{Lspec_natural}
\begin{enumerate}
\item Let $\ASS_\infty\to \oper$ and $\ASS_\infty\to \oper'$ be two maps of operads and $f:\oper\to \oper'$ a quasi-isomorphism compatible with the maps from $\ASS_\infty$. Then $f$ induces an isomorphism $E_r(\oper)\cong E_r(\oper')$ compatible with the differentials for $r\geq 1$.
\item Let $\ASS_\infty\to C_*(\KK_2)$ be the composition $\ASS_\infty\to \ASS\to C_*(\KK_2)$ of the fixed maps (see Definition \ref{DKontsevich}, and $\ASS$ denotes the algebraic associative operad here). The spectral sequence $E_r(C_*(\KK_2))$ is isomorphic to $\EE_r$ (see Definition \ref{Dspecseq}). \hfill\qedsymbol
\end{enumerate}
\end{lem}
The following lemma is easily obtained by unwinding the definition of the spectral sequence $E_r(\oper)$.
\begin{lem}\label{Lspec_diff_description}
Let $\ASS_\infty\to \oper$ be a map of operads. 
\begin{enumerate}
\item $d_1([x])=[\mu_2*x]$ for an element $[x]\in E^{-p,q}_1(\oper)$ represented by $x\in (\oper(p)_q)^\vee$.
\item For an element $[x]\in E^{-p,q}_2(\oper)$ represented by $x\in (\oper(p)_q)^\vee$, we  can take an element \\$y\in (\oper(p-1)_{q-1})^\vee$ with $dy=\mu_2*x$. We have $d_2[x]=[\mu_2*y+\mu_3*x]$.\hfill \qedsymbol
\end{enumerate}
\end{lem}

\begin{proof}[Proof of Corollary \ref{Cnon-formality_ch2}]
Let $\ASS_\infty\to C_*(\KK_2)$ be the map given in Lemma \ref{Lspec_natural}. We use the projective model structure on the category of planar chain operads (see e.g.\cite{moriya,tsopmene}). We take a cofibrant replacement $\ASS_\infty\rightarrowtail \oper\stackrel{\sim}{\twoheadrightarrow} C_*(\KK_2)$. 
Suppose  $C_*(\KK_2)$ is formal. By this assumption, we  can take a quasi-isomorphism of operads $\oper\to H_*(\KK_2)$. By considering the composition $f:\ASS_\infty\to \oper\to H_*(\KK_2)$, we obtain an isomorphism of spectral sequences $E_r(\oper)\cong E_r(H_*(\KK_2))$. Let $g:\ASS\to H_*(\KK_2)$ be the map induced by the map $\ASS\to \KK_2$ in Definition \ref{DKontsevich}. $E_r(H_*(\KK_2))$ might not be isomorphic to the sequence induced by the map $\ASS_\infty\to \ASS\stackrel{g}{\to} H_*(\KK_2)$.  Let $\mu'_l\in H_{l-2}(\KK_2(l))$ be the image of $\mu_l$ by $f$. By definition of $\ASS_\infty$, we see 
\[
\mu'_2\circ_1\mu'_3+\mu'_2\circ_2\mu'_3+\mu'_3\circ_1\mu'_2+\mu'_3\circ_2\mu'_2+\mu'_3\circ_3\mu'_2=d\mu'_4=0.
\]
Since $\mu_2'$ is, up to scalar multiple,  the image of the generator $\mu_2$ by $g$, 
this equation means that $\mu'_3$ is a cycle for the differential of the  chain complex of  the cosimplicial vector space associated to the map $g: \ASS\to H_*(\KK_2)$. By easy (and well-known) computation, this differential which is given by the sum of coface maps is a monomorphism on $H_1(\KK_2(3))$, so we have $\mu'_3=0$. This observation and Lemma \ref{Lspec_diff_description} imply $d_2=0$ for $E_2(H_*(\KK_2))$. Since $E_r(\oper)$ is isomorphic to $E_r(C_*(\KK))\cong \EE_r$, this vanishing of differential contradicts to Theorem \ref{Tdifferential}.
\end{proof}
\begin{rem}
To deduce the absolute non-formality of the little 2-disks in characteristic 3 from the non-vanishing of the $d_3$-differential, we need $\mu'_4=0$ in the notation of the proof of Corollary \ref{Cnon-formality_ch2}, but this does not follows from an argument similar to $\mu'_3=0$. It might be still possible to prove the non-formality, but it will require a considerable amount of algebraic calculation. This problem will be considered elsewhere.
\end{rem}

\end{document}